\documentclass[11pt,reqno]{amsart}

\usepackage{amsfonts,amsmath,amsthm}
\usepackage{amssymb,epsfig}
\usepackage{enumerate} 

\usepackage{color} 
\definecolor{vert}{rgb}{0,0.6,0}

\usepackage[text={425pt,650pt},centering]{geometry}
\usepackage{graphicx}
\usepackage{epsfig}
\usepackage{tikz,esvect}
\usepackage{caption}
\usepackage{tikz-3dplot}
\usetikzlibrary{3d,calc,shapes}
\usetikzlibrary{shapes.geometric}
\usepackage{adjustbox}
\usepackage{subfig}
\usepackage[normalem]{ulem}

\usepackage{caption}
\usepackage{color} 
\definecolor{vert}{rgb}{0,0.6,0}

\numberwithin{figure}{section}

\theoremstyle{plain}
\newtheorem{thm}{Theorem}[section]

\newtheorem{defn}{Definition}

\newtheorem{lem}[thm]{Lemma}
\newtheorem{cor}[thm]{Corollary}
\newtheorem{prop}[thm]{Proposition}
\theoremstyle{remark}
\newtheorem{rem}{\bf{Remark}}
\numberwithin{equation}{section}
\newcommand{\red}{\color{red}}

\newcommand{\N}{\mathbb{N}}

\newcommand{\R}{\mathbb{R}}

\newcommand{\T}{\mathbb{T}}
\newcommand{\Z}{\mathbb{Z}}

\newcommand{\tm}{\widetilde{m}}
\newcommand{\fep}{f_\varepsilon}

\newcommand{\Lbar}{\overline{L}}

\newcommand{\del}{\delta}
\newcommand{\ep}{\varepsilon}

\newcommand{\Om}{\Omega}

\newcommand{\ol}{\overline}

\usepackage[unicode=true]{hyperref}
\hypersetup{
     colorlinks,
     linkcolor={black!10!blue},
     linkbordercolor = {black!100!blue},
     citecolor={red}
}
\usepackage{enumitem}
\makeatletter
\def\namedlabel#1#2{\begingroup
    #2%
    \def\@currentlabel{#2}%
    \phantomsection\label{#1}\endgroup
}
\makeatother

\begin{document}

\title[Homogenization on perforated domains with Dirichlet conditions]
{Quantitative homogenization of Hamilton--Jacobi equations on perforated domains with Dirichlet boundary conditions}

\author[Y. Han, S. Tu] 
{Yuxi Han and Son Tu}

\thanks{
}

\address[Y. Han]
{Department of Mathematics, Purdue University, West Lafayette, IN}
\email{han891@purdue.edu}

\address[S. Tu]
{Department of Mathematics, Baylor University, Waco, TX
}
\email{son\_tu@baylor.edu}

\date{\today}

\maketitle

\begin{abstract}
We study the periodic homogenization of convex Hamilton–Jacobi equations on perforated domains with Dirichlet boundary conditions. By analyzing the optimal control representation of the solutions and the properties of the metric function associated with the running cost, we establish the optimal convergence rate $\mathcal{O}(\varepsilon)$ for homogenization. A key aspect of our approach is the treatment of the singularity that arises when the optimal path does not fully utilize the available time.
\end{abstract}


\section{Introduction}
We study the following homogenization problem on perforated domains with Dirichlet boundary conditions. Consider an open and connected set $\Omega \subset \mathbb{R}^n$ with $C^1$ boundary and assume $\Omega$ is $\mathbb{Z}^n$-periodic, which means it satisfies the condition $\Omega+ \mathbb{Z}^n = \Omega$. Note that for $\Om$ to be connected, we need $n\geq 2$ (see \cite[Figures~1.1--1.3]{han_quantitative_2024} for examples of domains \(\Omega\)). For $\varepsilon >0$, consider the domain $\Omega_\varepsilon: = \varepsilon \Omega$ and let $u^\varepsilon$ be the unique viscosity solution to the Cauchy problem with a Dirichlet boundary condition 
\begin{equation}\label{eqn:PDE_epsilon}
    \left\{\begin{aligned}
    u_t^\varepsilon+H\left(\frac{x}{\varepsilon}, Du^\varepsilon\right) & = 0 \qquad \qquad \qquad \text{in } \Omega_\varepsilon \times (0, \infty),\\
    u^\varepsilon (x, t) & = b\left(x, \frac{x}{\varepsilon}\right) \qquad \, \,\,\, \text{on } \partial \Omega_\varepsilon \times (0, \infty), \\
    u^\varepsilon(x,0) & =g(x)  \qquad \qquad \, \,\, \text{on } \overline{\Omega}_\varepsilon \times \{t=0\}, 
    \end{aligned}
    \right.
\end{equation}
where $H$ is a given Hamiltonian, $b \in \mathrm{BUC}\left(\mathbb{R}^n\times \partial \Omega \right)$ prescribes the boundary condition for $t>0$, and $g$ is the initial condition. For each $x \in \mathbb{R}^n$, the map $y \mapsto b(x,y)$ is $\mathbb{Z}^n$-periodic; that is, for all $(x,y) \in \mathbb{R}^n \times \partial \Omega$ and $z \in \mathbb{Z}^n$,
\[
b(x, y+z)=b(x, y).
\]
By homogenization, we expect that, under suitable assumptions, as $\varepsilon \to 0^+$, $u^\varepsilon$ converges locally uniformly on $\overline{\Omega}_\varepsilon \times [0,\infty)$ to $u$, the unique viscosity solution of the following equation (see \cite[Theorem~1.2]{Horie1998Homogenization})
\begin{equation}\label{eqn:PDE_limit}
    \left\{\begin{aligned}
        \max\{u-\overline{b},    u_t+\overline{H} \left(Du\right)\} & = 0 \qquad \quad \text{in } \mathbb{R}^n \times (0, \infty), \\
    u(x, 0)&=g(x) \, \, \,\quad \text{on } \mathbb{R}^n \times \left\{t=0\right\}, 
    \end{aligned}
    \right. 
\end{equation}
where $\overline{H}$ is the effective Hamiltonian associated with the state-constraint cell problem on $\ol \Omega$, and $\overline{b}: \R^n \to \R$ is defined by 
\[
\overline{b}(x):= \min_{z \in\partial \Omega} b(x,z).
\]
More precisely, for each $p \in \mathbb{R}^n$, $\overline{H}(p) = \overline{H}_\Omega(p)$ is the unique constant for which the following state-constraint (see \cite{Capuzzo-Dolcetta1990HamiltonJacobiConstraints, Horie1998Homogenization, mitake_asymptotic_2008}) cell problem admits a $\mathbb{Z}^n$-periodic viscosity solution $v$:
\begin{equation}\label{eqn:cell}
    \left\{\begin{aligned}
    H\left(y, p+Dv(y)\right) &\leq \overline{H}_\Omega(p) \quad \text{in }  \Omega, \\
    H\left(y, p+Dv(y)\right) &\geq \overline{H}_\Omega(p) \quad \text{on }  \overline{\Omega}.
    \end{aligned}
    \right.
\end{equation}
We write $\overline{H}=\overline{H}_\Omega$ when it is clear from context.
We are interested in how fast $u^\varepsilon$ converges to $u$.

\subsection{Relevant Literature} 

The study of homogenization for first-order Hamilton--Jacobi equations in the periodic setting was initiated in \cite{lions_homogenization_1986}. Subsequent works investigated the qualitative homogenization of Hamilton--Jacobi equations on perforated domains \cite{Horie1998Homogenization,alvarez_homogenization_1999,alvarez_hamilton-jacobi_2001}. For qualitative studies on state-constraint and Dirichlet problems, we refer to \cite{Capuzzo-Dolcetta1990HamiltonJacobiConstraints, Horie1998Homogenization, mitake_asymptotic_2008, Soner1986OptimalI, mitake_large_2009, ishii_vanishing_2017}.

Quantitative homogenization for first-order Hamilton--Jacobi equations in periodic setting was first studied in \cite{capuzzo-dolcetta_rate_2001}, where a rate of $\mathcal{O}(\varepsilon^{1/3})$ is provided for general coercive Hamiltonians.
A rate of $\mathcal{O}(\varepsilon)$ was later achieved in \cite{MitakeTranYu2019}, and the optimal rate of convergence was established in \cite{tran_optimal_2021}. Since then, there has been substantial progress in this direction; see \cite{jing_effective_2020, MitakeTranYu2019, tu_rate_2021, tran_hamilton-jacobi_2021, tran_yu_2022_weak_kam, hu_polynomial_2025, mitake_rate_2023} and the references therein. In particular, the framework of \cite{tran_optimal_2021} is robust and has been extended to various settings, including multiscale problems \cite{han_rate_2023}, spatio-temporal problems \cite{nguyen-tien_optimal_2024}, weakly coupled systems \cite{mitake_rate_2025}, perforated domains with state-constraint boundary conditions \cite{han_quantitative_2024}, perforated domains with Neumann boundary conditions \cite{mitake_quantitative_2025_Neumann}, and fully nonlinear problems \cite{mitake_quantitative_2025_contact}.

Our work is motivated by \cite{han_quantitative_2024}. To the best of our knowledge, the rate of convergence on perforated domains with Dirichlet boundary conditions has not yet been established in the literature. In addition, we provide an optimal control representation formula for the solution $u$ to the limiting equation \eqref{eqn:PDE_limit}, which appears to be new.



\subsection{Settings}

The following assumptions will be used throughout this paper. Most theorems rely on \ref{itm:A1}-\ref{itm:A5}, while the main result requires \ref{itm:A1}-\ref{itm:A6}.

\begin{itemize}
    \item[\namedlabel{itm:A1}{$(\mathcal{A}_1)$}] $H\in C(\R^n\times \R^n)$; and for $p\in \R^n$, $y\mapsto H(y,p)$ is $\Z^n$-periodic.
    \item[\namedlabel{itm:A2}{$(\mathcal{A}_2)$}] $\lim_{\left|p\right| \to \infty}\left(\inf_{y \in \R^n}H\left(y,p\right)\right) = +\infty$.
    \item[\namedlabel{itm:A3}{$(\mathcal{A}_3)$}] For each $y\in \R^n$, the map $p \mapsto H(y, p)$ is convex.
    \item[\namedlabel{itm:A4}{$(\mathcal{A}_4)$}] $b \in \mathrm{BUC}\left(\mathbb{R}^n\times \partial \Omega \right)$, and the function $y\mapsto b(x, y)$ is periodic with $\Z^n$ for every $x$. That is, if $(x, y) \in \R^n \times \partial\Omega$ and $z \in \Z^n$, then
    \[
    b(x, y+z)=b(x, y).
    \]
    \item[\namedlabel{itm:A5}{$(\mathcal{A}_5)$}] $g \in \mathrm{BUC}(\mathbb{R}^n) \cap \mathrm{Lip} (\mathbb{R}^n)$, and $g(x)\leq \overline{b}(x)$ for any $x\in \mathbb{R}^n$, where $\overline{b}: \R^n \to \R$ is defined by 
    \[
    \overline{b}(x):= \min_{z \in\partial \Omega} b(x,z).
    \]
    \item[\namedlabel{itm:A6}{$(\mathcal{A}_6)$}]
     $b \in \mathrm{Lip} \left(\mathbb{R}^n\times \partial \Omega\right)$.
\end{itemize}

\subsection{Simplifications}
We recall some standard simplifications (see \cite{han_quantitative_2024,tran_optimal_2021}) and necessary estimates.

Let $\T^n=\R^n/\Z^n$ be the usual flat $n$ dimensional torus. 
Then, we can also write $H \in C(\T^n \times \R^n)$. Although \eqref{eqn:PDE_epsilon} and \eqref{eqn:cell} only require $H$ to be defined on $\ol{\Omega}\times \R^n$, it is more convenient for us to consider $H$ on $\R^n \times \R^n$. Note that \ref{itm:A1} and \ref{itm:A2} imply that there exists a constant $C_1>0$ so that
\begin{equation}\label{eq:H-lower-bd}
H(y, p) \geq -C_1 \qquad \text{ for all } (y, p) \in \R^n \times \mathbb{R}^n.
\end{equation}
The well-posedness of \eqref{eqn:PDE_epsilon} has been well studied in the literature (see \cite{{Ishii1989_dirichletHJ}}). 
Furthermore, one can show that the solution $u^\varepsilon$ is globally Lipschitz, that is, for $\ep\in (0,1)$,
\begin{equation}\label{eqn:u_prior}
	\left\|u^\varepsilon_t\right\|_{L^\infty\left(\overline{\Omega} \times [0, \infty)\right)}+\left\|Du^\varepsilon\right\|_{L^\infty(\overline{\Omega} \times [0, \infty))} \leq C_0,
\end{equation}
where $C_0>0$ is a constant that depends only on $H$ and $\left\|Dg\right\|_{L^\infty(\mathbb{R}^n)}$.
Indeed, by \eqref{eq:H-lower-bd}, the function $g(x)+C_{1}t$ is a supersolution of \eqref{eqn:PDE_epsilon}, with the boundary condition interpreted in the generalized sense (see \cite{Ishii1989_dirichletHJ}). 
Moreover, assumptions \ref{itm:A1} and \ref{itm:A5} imply that, for 
\begin{equation*}
    C = C\!\left(H,\|Dg\|_{L^\infty(\R^n)}\right)>0    
\end{equation*}
sufficiently large, the function $g(x)-Ct$ is a subsolution of \eqref{eqn:PDE_epsilon}. 
Applying the comparison principle (with generalized boundary conditions, \cite[Theorem~2.1]{Ishii1989_dirichletHJ}, suitably adapted to the unbounded domain and time-dependent setting), we obtain
\[
g(x)-Ct \;\leq\; u^\ep(x,t) \;\leq\; g(x)+C_{1}t, 
\qquad (x,t) \in \overline{\Omega} \times [0,\infty).
\]
Hence, $\|u^\ep_t(\cdot,0)\|_{L^\infty(\ol \Om)} \leq C$.
By using the comparison principle once more, we get 
\begin{equation*}
	\|u^\ep_t\|_{L^\infty(\ol \Om\times[0,\infty))} \leq C. 
\end{equation*}
Finally, we use this bound, \eqref{eqn:PDE_epsilon}, and \ref{itm:A2} to obtain \eqref{eqn:u_prior}.

Based on \eqref{eqn:u_prior}, we can modify $H(y, p)$ for $|p| > 2C_0+1$ without changing the solutions to \eqref{eqn:PDE_epsilon} and ensure that
\begin{equation}\label{eqn:K_0H}
    \frac{|p|^2}{2}-K_0 \leq H(y, p) \leq \frac{|p|^2}{2}+K_0, \qquad (y,p) \in \R^n\times \mathbb{R}^n, 
\end{equation}
for some constant $K_0 >0$ that depends only on $H$ and $\left\|Dg\right\|_{L^\infty(\mathbb{R}^n)}$. Consequently, 
\begin{equation}\label{eqn:K_0L}
\frac{|v|^2}{2}-K_0 \leq L(y, v) \leq \frac{|v|^2}{2}+K_0,\qquad (y,v)  \in \R^n \times \mathbb{R}^n,
\end{equation}
where $L: \R^n \times \mathbb{R}^n \to \mathbb{R
}$ is the Legendre transform of $H$, defined by
\begin{align}\label{eq:L}
    L(x,v) = \sup_{p\in \R^n} \left(p\cdot v - H(x,p)\right), \qquad (x,v)\in \R^n\times \R^n. 
\end{align}
Under \eqref{eqn:K_0H}, \eqref{eqn:K_0L}, we obtain
\begin{equation}\label{eq:Hbar-Lbar}
    \frac{|p|^2}{2}-K_0 \leq \ol H(p) \leq \frac{|p|^2}{2}+K_0,\qquad
    \frac{|v|^2}{2}-K_0 \leq \ol L(v) \leq \frac{|v|^2}{2}+K_0
\end{equation}
for all $p,v \in \R^n$, where $\overline{L}$ is the Legendre transform of $\overline{H}$ defined in \eqref{eqn:cell}.
We always assume \eqref{eqn:K_0H}, \eqref{eqn:K_0L}, and \eqref{eq:Hbar-Lbar} in our analysis from now on.
\medskip 

\subsection{Main results}
Our main result establishes the convergence rate $\mathcal{O}(\varepsilon)$ for $u^\varepsilon \to u$, where $u^\varepsilon$ and $u$ are solutions to \eqref{eqn:PDE_epsilon} and \eqref{eqn:PDE_limit}, respectively. We impose the following assumption on the domain:
\begin{itemize} \item[\namedlabel{itm:D1}{$(\mathcal{D}_1)$}] $\Omega \subset \mathbb{R}^n$ with $C^1$ boundary and assume $\Omega$ is $\mathbb{Z}^n$-periodic, that is, $\Omega+ \mathbb{Z}^n = \Omega$. \end{itemize}


\begin{thm}\label{thm:main1}
    Assume \ref{itm:D1}, \ref{itm:A1}--\ref{itm:A6}. For $\varepsilon>0$, let $u^\varepsilon$ be the viscosity solution to \eqref{eqn:PDE_epsilon} and let $u$ be the viscosity solution to \eqref{eqn:PDE_limit}, respectively. Then we have
\begin{equation}
    \left\|u^\varepsilon- u\right\|_{L^\infty\left(\overline{\Omega}_\varepsilon \times [0, \infty)\right)} \leq C \varepsilon
\end{equation}
where $C>0$ is a constant depends on $n, \partial\Omega, H, \left\|Dg\right\|_{L^\infty\left(\mathbb{R}^n\right)}$ and $\left\|Db\right\|_{L^\infty \left(\mathbb{R}^n\times \partial \Om\right)}$. 
\end{thm}

We use the optimal control formula for $u^\varepsilon$ given in \eqref{eqn:ocfue} to establish the result. As in the state-constraint problem, admissible paths are restricted to $\overline{\Omega}$. The main difference is that, for $u^\varepsilon(x,t)$, the admissible path can choose whether to use all the available time or exit the domain at the boundary with a terminal cost. This flexibility in time length is the primary difficulty in proving the rate of convergence.

We also define a cost function $\tilde{m}$ based on the running cost of $u^\varepsilon$ and extend it to $\mathbb{R}^n$, analogously to $m$ in \cite{han_quantitative_2024}, but with an additional parameter $\tau$ representing the effective time $t-\tau$ (see \eqref{def:mtilde}). When $\tau$ approaches $t$, $\tilde{m}$ diverges. In \cite{han_quantitative_2024}, this singularity is irrelevant since, in the state-constraint problem, paths must use all the available time. There, for $t>\varepsilon$, the rate of convergence follows from the super- and subadditivity of the metric, while for $t<\varepsilon$, it is handled via the comparison principle.

In our setting, however, $\tau$ can be arbitrarily close to $t$ even when $t$ is large, so the singularity must be carefully controlled. The key observation is that a uniform bound can be established for $\tau$ away from $t$ when $u < \overline{b}$ (see Lemma \ref{lem:ueptimelbd}). Combined with a lower bound on the limiting metric (see Lemma \ref{lem:tmbstarbd}), this allows us to manage the singularity and obtain the rate of convergence.

\medskip

Next, we consider the dilute and defective domain cases for the Dirichlet problem following \cite{han_quantitative_2024}. The dilute setting \ref{itm:D2} occurs when $\mathbb{R}^n \setminus \overline{\Omega}_\varepsilon$ has no unbounded components and obstacles shrink faster than $\varepsilon$, while the defective domain \ref{itm:D3} involves non-diluted holes, some of which may be missing.
\begin{itemize} 
    \item[\namedlabel{itm:D2}{$(\mathcal{D}_2)$}] Let $D\subset\subset \left(-\frac{1}{2}, \frac{1}{2}\right)^n$ containing $0$ with connected $C^1$-boundary. 
    Let $\eta:\left[0,\frac{1}{2}\right)\to \left[0,\frac{1}{2}\right)$ be such that $\lim_{\varepsilon\to 0}\eta(\varepsilon) = 0$. We denote by $\Omega^{\eta(\varepsilon)} = \R^n\backslash \bigcup_{z\in \Z^n} \left( \eta(\varepsilon)\overline{D} + z\right)$ and 
    $\Omega_\varepsilon := \varepsilon \Omega^{\eta(\varepsilon)}$ for $\varepsilon \in [0,\tfrac{1}{2})$.
    
    \item[\namedlabel{itm:D3}{$(\mathcal{D}_3)$}] $\Omega$ satisfies \ref{itm:D2} with $\eta \equiv 1$. Let $I\subsetneq \Z^n$ with $I\neq \emptyset$ be an index set that denotes the places where the holes are missing. Define $W = \R^n\backslash \bigcup_{z\in \Z^n\backslash I} (\overline{D}+z)$ and $W_\varepsilon = \varepsilon W$. Assume that there exists a modulus $\omega$ such that 
    \[
    \frac{|I \cap [-k,k]^n|}{k} = \omega\!\left(\tfrac{1}{k}\right), \qquad k \to \infty.
    \]
\end{itemize}

Under the additional assumption \ref{itm:A7}, the results for both the dilute and defective domain cases of the Dirichlet problem follow from Theorem \ref{thm:main1} together with the state-constraint results \cite[Theorems 1.2 and 1.3]{han_quantitative_2024}.

\begin{prop}\label{prop:DilutedDefect} 
Assume \ref{itm:A1}--\ref{itm:A6}, together with the additional assumption:
\begin{itemize}
    \item[\namedlabel{itm:A7}{$(\mathcal{A}_7)$}] For each $y \in \mathbb{R}^n$, $\min{p \in \mathbb{R}^n} H(y,p) = H(y,0) = 0$.
\end{itemize}
Then the following results hold.
    \begin{itemize}
        \item[$\mathrm{(i)}$] 
        Assume \ref{itm:D2}. Let $\overline{H}_{\mathbb{R}^n}$ denote the effective Hamiltonian in \eqref{eqn:cell} posed on $\mathbb{T}^n$ in place of $\Omega$, and let $\tilde{u}$ be the solution of the effective equation
        \begin{equation*}
            \begin{cases}
                \begin{aligned}
                    \tilde{u}_t + \overline{H}_{\R^n}(D\tilde{u}) &= 0 &&\text{in}\;\R^n, \\
                    \tilde{u}(x,t) &= g(x) &&\text{on}\;\R^n.
                \end{aligned}
            \end{cases}
        \end{equation*}
        Then there exists a constant \( C = C(n, \partial D, H, \mathrm{Lip}(g)) \) such that, for all \( t \geq 0 \), the solution \( u^\varepsilon \) to \eqref{eqn:PDE_epsilon} posed on \( \Omega_\varepsilon \) as in \ref{itm:D2} satisfies
        \begin{equation*}
            \Vert u^\varepsilon(\cdot, t) - \tilde{u}(\cdot, t)\Vert_{L^\infty(\overline{\Omega}_\varepsilon)} \leq C(\varepsilon + \eta(\varepsilon)t). 
        \end{equation*}
        
        \item[$\mathrm{(ii)}$] Assume \ref{itm:D3}. There exists constants $C = C(n, \partial D, H, \mathrm{Lip}(g))$ and $M_0  = M_0(H, \mathrm{Lip}(g))$ such that the solution $w^\varepsilon$ to \eqref{eqn:PDE_epsilon} posed on $W_\varepsilon$ satisfies
        \begin{align*}
            |w^\varepsilon(x,t) - u(x,t)| \leq C(M_0t+|x|+1)\omega\left(\frac{\varepsilon}{M_0 t + |x|}\right) + C\varepsilon.
        \end{align*}
        for all $(x,t) \in \overline{W}_\varepsilon \times [0,\infty)$. 
    \end{itemize}
\end{prop}
\begin{rem}

Here, $M_0$, which depends on $H$ and $\mathrm{Lip}(g)$, serves as a velocity bound for the minimizers of the value function $w^\varepsilon$ (see Lemma \ref{lem:velocity bound}). By adapting the techniques developed in \cite{mitake_asymptotic_2008, ishii_asymptotic_2008}, we derive this estimate in a manner that eliminates the dependence of $M_0$ on $n$ and $\partial \Omega$, in contrast to \cite[Lemma 2.1]{han_quantitative_2024}. The issue of such dependence was previously noted in \cite[Remark 7]{han_quantitative_2024}; here, we provide a complete resolution.
\end{rem}



\subsection*{Notations} 
We write $B(x,r) = B_r(x)$ to denote the open ball centered at $x\in \R^n$ with radius $r>0$. 
We write $Y=[-\frac{1}{2},\frac12]^n$ 
as the unit cube in $\R^n$. 
For $g \in \mathrm{Lip}(\mathbb{R}^n)$, the space of Lipschitz functions on $\mathbb{R}^n$, let $\mathrm{Lip}(g)$ denote its Lipschitz constant.
For $a, b \in \mathbb{R}^n$, $a \vee b$ denotes $\max\{a, b\}$, and $a \wedge b$ denotes $\min\{a, b\}$. For $p,q \in \mathbb{R}^n$, we denote by $[p,q]$ the line segment connecting $p$ and $q$, that is, $[p,q] := \{\, tp + (1-t)q : t \in [0,1] \,\}$. Let $\mathrm{AC}(J,U)$ be the set of absolutely continuous curves from $J \to U$. For $S \subset \mathbb{R}^n$, define 
\begin{equation*}
    \mathcal{C}(x,t;S) = \left\{\gamma\in \mathrm{AC}([0,t];S): \gamma(t)=x \right\}.     
\end{equation*}
For \(a \in \mathbb{R}\), we denote by \(\lfloor a \rfloor\) 
the greatest integer that does not exceed \(a\). For $\varepsilon, T>0$, set
\begin{align}\label{eq:QTeps}
	Q_T^\varepsilon := \Omega_\varepsilon \times (0,T), 
\qquad 
\partial Q_T^\varepsilon := (\partial\Omega_\varepsilon \times (0,T)) \cup (\overline{\Omega}_\varepsilon \times \{0\}).
\end{align}
If $\varepsilon=0$, we set 
\begin{equation}\label{eq:QT0}
	Q_T := \Omega\times (0,T), \qquad \partial Q_T:= (\partial\Omega\times (0,T)) \cup (\overline{\Omega}\times \{0\})
\end{equation}
Let $\phi \in C_c^\infty(B_1)$ satisfy $0 \le \phi \le 1$ and $\int_{\R^n} \phi(x)\,dx = 1$.  
For $\varepsilon > 0$, define 
\begin{equation}\label{eq:mollifier}
    \phi_\varepsilon(x) = \varepsilon^{-n}\phi(x/\varepsilon), \qquad x\in \R^n,
\end{equation}  
which is the standard mollifier.

\subsection*{Organization of this paper}
In Section \ref{sec:prelim}, we present the optimal control formulas for the solutions \(u^\varepsilon\) of \eqref{eqn:PDE_epsilon} and \(u\) of \eqref{eqn:PDE_limit}, along with several properties of these solutions that will be used in subsequent sections. In Section \ref{sec:costfunction}, we define the cost (metric) function in our setting, derive its properties based on \cite{han_quantitative_2024}, extend it to the whole space, and relate its limiting metric to the solution \(u\) of \eqref{eqn:PDE_limit}. The proof of Theorem \ref{thm:main1} is provided in Section \ref{sec:ProofThm1}. Section \ref{sec:discussion} examines the optimality of Theorem \ref{thm:main1}, its implications for dilute and defective domains, and includes the proof of Proposition \ref{prop:DilutedDefect}. Finally, the Appendix contains auxiliary results from the state-constraint setting and a technical lemma on smooth domain shrinkage.


\section{Representation of solutions via optimal control} \label{sec:prelim}

We present the optimal control formulas for \(u\) and \(u^\varepsilon\), giving full details only for \(u\) since the formula for \(u^\varepsilon\) is well known.
For clarity and simplicity, we first introduce the following auxiliary functions.
\begin{defn}
    Define $\fep: \left(\partial \Omega_\varepsilon \times (0, \infty)\right) \cup \left(\overline{\Omega}_\varepsilon \times \{0\} \right) \to \mathbb{R}$ by
    \begin{equation}\label{eqn:deff_ep}
        f_\varepsilon(x, t) = \left\{
        \begin{aligned}
        &g(x), \qquad \qquad\text{ if } (x, t) \in \overline{\Omega}_\varepsilon \times \{0\} \\
        &b\left(x, \frac{x}{\varepsilon}\right), \qquad\text{ if } (x, t) \in \partial \Omega_\varepsilon \times (0, \infty).
        \end{aligned}\right.
    \end{equation}
    Define $f:\mathbb{R}^n \times [0, \infty) \to \mathbb{R}$ by
    \begin{equation}\label{eqn:deff}
        f(x, t) = \left\{
        \begin{aligned}
        &g(x), \text{ if } t=0\\
        &\overline{b}(x), \text{ if } t>0,
        \end{aligned}\right.
    \end{equation}
     where $\overline{b}: \R^n \to \R$ is defined by 
    \[
    \overline{b}(x)= \min_{z \in\partial \Omega} b(x,z).
    \]
\end{defn}
\begin{rem} Some remarks are in order. 
\begin{enumerate}[leftmargin=0.7cm]
    \item[(a)] Assume \ref{itm:A4}, \ref{itm:A5}. Since $g \leq \overline{b}$ on $\mathbb{R}^n$, we have
    \begin{align}\label{eqn:ftinc} 
        g(x) \leq  f(x,s)  \leq f(x,t)
    \end{align}
    for $x\in \R^n, 0\leq s\leq t$. Similarly, for $\varepsilon>0$,
    \begin{equation}\label{eqn:feptinc}
        g(x) \leq \fep (x, s)\leq \fep(x, t) 
    \end{equation}
    for $(x,t), (x,s) \in 
        \left(\partial \Omega_\varepsilon \times (0, \infty)\right) \cup \big(\overline{\Omega}_\varepsilon \times \{0\} \big)$ with $s \leq t$. We also note that
        \begin{equation}\label{eq:ffeps}
        \begin{aligned}
        &	f_\varepsilon(x,t) = b\left(x, \frac{x}{\varepsilon}\right) \geq \overline{b}(x) = f(x,t) 
        &&\qquad (x,t)\in \partial\Omega_\varepsilon \times (0,\infty), \\
		&   f_\varepsilon(x,t) = g(x) = f(x,t)         	
		&&\qquad (x,t)\in \overline{\Omega}_\varepsilon \times \{0\}.
        \end{aligned}
        \end{equation}

    \item[(b)] Assume \ref{itm:A4}, \ref{itm:A5}. Let $\ep>0$. Since $g\in \mathrm{Lip}\left(\mathbb{R}^n\right)$ and $b \in \mathrm{BUC}\left(\mathbb{R}^n\right)$, for any $(x,t), \left( \tilde{x},t \right) \in \left(\partial \Omega_\varepsilon \times (0, \infty)\right) \cup \left(\overline{\Omega}_\varepsilon \times \{0\} \right)$,
    \begin{equation}\label{eqn:fepucinx}
        \left|\fep(x,t) - \fep(\tilde{x}, t)\right|\leq \omega_{\fep}\left( \left|x-\tilde{x}\right|\right),
    \end{equation}
    where $\omega_{\fep}$ is some modulus of continuity that depends on $\ep$, $\left\|Dg\right\|_{L^\infty\left(\mathbb{R}^n\right)}$, and the modulus of continuity of $b$. Similarly, since $g\in \mathrm{Lip}\left(\mathbb{R}^n\right)$ and $\overline{b} \in \mathrm{BUC}\left(\mathbb{R}^n\right)$, for any $x, \tilde{x} \in \mathbb{R}^n$ and $t \geq 0$,
    \begin{equation}\label{eqn:fucinx}
        \left|f(x,t) - f(\tilde{x}, t)\right|\leq \omega_f\left( \left|x-\tilde{x}\right|\right),
    \end{equation}
    where $\omega_f$ is some modulus of continuity that depends on $\left\|Dg\right\|_{L^\infty\left(\mathbb{R}^n\right)}$ and the modulus of continuity of $b$. 
    \item[(c)] Assume \ref{itm:A4}--\ref{itm:A6}. Since $b$ and $\overline{b}$ are Lipschitz continuous, we have
    \begin{equation}\label{eqn:feplipinx}
        \left|\fep(x,t) - \fep(\tilde{x}, t)\right|\leq C_{\fep}\left( \left|x-\tilde{x}\right|\right)
    \end{equation}
   for $(x,t), \left( \tilde{x},t \right) \in \left(\partial \Omega_\varepsilon \times (0, \infty)\right) \cup \left(\overline{\Omega}_\varepsilon \times \{0\} \right)$, 
     where $C_{\fep}>0$ is a constant that depends on $\ep$, $\left\|Dg\right\|_{L^\infty\left(\mathbb{R}^n\right)}$, and $ \left\|Db\right\|_{L^\infty \left(\mathbb{R}^n\times \partial \Om\right)}$. For any $x, \tilde{x} \in \mathbb{R}^n$ and $t \geq 0$,
    \begin{equation}\label{eqn:flipinx}
        \left|f(x,t) - f(\tilde{x}, t)\right|\leq C_f\left( \left|x-\tilde{x}\right|\right),
    \end{equation}
    where $C_f>0$ is a constant that depends on $\left\|Dg\right\|_{L^\infty\left(\mathbb{R}^n\right)}$ and $ \left\|Db\right\|_{L^\infty \left(\mathbb{R}^n\times \partial \Om\right)}$.
\end{enumerate}
\end{rem}

\subsection{Optimal control formula for $u$}
We now provide an optimal control representation of $u$, analogous to the Hopf–Lax formula. 

\begin{prop}\label{prop:ValueU} Assume \ref{itm:A1}--\ref{itm:A5}. For $x\in \R^n$ and $t\geq0$, we have
\begin{equation}\label{eq:HLubar}
u(x, t)=\begin{cases}
\displaystyle\inf_{\tau \in [0,t),\, y \in \mathbb{R}^n}
\left\{
(t - \tau)\, \overline{L}\left(\frac{x - y}{t - \tau}\right) + f(y, \tau)
\right\}, & \text{if } t > 0, \\
g(x), & \text{if } t = 0.
\end{cases}   
\end{equation}
\end{prop}

Note that the optimal control formula on the right-hand side of \eqref{eq:HLubar} is very similar to the classical Hopf–Lax formula, except that the infimum is also taken over an additional parameter $\tau \in [0, t)$. In other words, the entire time interval $t$ need not be used; instead, we only consider straight-line trajectories that travel from $y$ to $x$ in time $t-\tau$.

Before proving Proposition \ref{prop:ValueU}, we establish in the following lemma several properties of the optimal control formula on the right-hand side of \eqref{eq:HLubar}, denoted by $w(x, t)$, including a version of the Dynamic Programming Principle (see also \cite[Theorem 6.1]{barles_introduction_2013}). Among these, part (b) is useful in proving Proposition \ref{prop:ValueU}: it shows that when $w(x, t)<\ol{b}(x)$ with $t>0$, the time $\tau$ is uniformly bounded away from $t$. This property will be particularly important when handling the case of all $\tau \in[0, t)$.

\begin{lem}\label{lem:DPPHbar}
Assume \ref{itm:A1}--\ref{itm:A5}. Define the function \( w : \mathbb{R}^n \times (0, \infty) \to \mathbb{R} \) by
\begin{equation}\label{eqn:ocfw}
w(x, t) :=
\begin{cases}
\displaystyle\inf_{\tau \in [0,t),\, y \in \mathbb{R}^n}
\left\{
(t - \tau)\, \overline{L}\left(\frac{x - y}{t - \tau}\right) + f(y, \tau)
\right\}, & \text{if } t > 0, \\
g(x), & \text{if } t = 0.
\end{cases}
\end{equation}
Then, the following statements hold: 
\begin{enumerate}[label=$\mathrm{(\alph*)}$, leftmargin=0.7cm]
	\item For all $(x,t)\in \R^n\times [0,\infty)$, $w(x,t) \leq \overline{b}(x)$. 
    \item[$\mathrm{(b)}$] Let \( t > 0 \) and \( x \in \mathbb{R}^n \). Suppose $w(x, t)< \overline{b}(x)$. For any $\ep >0$, by definition there exist $\tau_\ep\in [0,t)$ and $y_\ep \in \R^n$ such that
    \begin{equation}\label{eqn:appw}
    w(x,t) \leq (t-\tau_\ep) \overline{L}\left(\frac{x-y_\ep}{t-\tau_\ep}\right)+f(y_\ep, \tau_\ep)  \leq w(x,t) + \ep.
    \end{equation}
    Then there exists $\ep_0 > 0$ such that for all $0 < \ep < \ep_0$,
    \begin{equation}\label{eqn:unifbdtauep}
    0 \leq \tau_\ep < \tau_{x, t} < t,
    \end{equation}
    for some constant $\tau_{x, t} \in (0, t)$. In other words, $\{\tau_\ep: \ep \in \left(0, \ep_0\right)\}$ is uniformly bounded away from $t$. Moreover, for $\ep \in \left(0, \ep_0\right)$, 
    \[
    \left|x-y_\ep\right|\leq C_{x,t}|t-\tau_\ep|
    \]
    for some constant $C_ {x,t}=C_{x,t}\left(x, t, H, \|g\|_{L^\infty(\R^n)} \right)>0$.
    \item[$\mathrm{(c)}$] Let \( t > 0 \) and \( x \in \mathbb{R}^n \). There exists a constant $C_ {x,t}=C_{x,t}\left(x, t, H, \|g\|_{L^\infty(\R^n)}\right)>0$ such that
    \begin{equation}\label{eqn:ocfybdd}
        w(x, t) = \inf
    \left\{
    (t - \tau)\, \overline{L}\left(\frac{x - y}{t - \tau}\right) + f(y, \tau): \tau \in [0, t), \left|x-y\right|\leq C_{x, t}\left|t-\tau\right|
    \right\}.    
    \end{equation}
    
    \item[$\mathrm{(d)}$] \emph{(Dynamic Programming Principle)} Let \( t > 0 \) and \( x \in \mathbb{R}^n \). For any \( s \in (0, t) \), we have
    \begin{align}\label{eq:DPPHbar}
    w(x, t) = \inf_{\tau \in [0, t),\, y \in \mathbb{R}^n}
    \Big\{ 
    \left(t - (\tau \vee s)\right) \overline{L}\left( \frac{x - y}{t - (\tau \vee s)} \right)
    + \chi_{\{s > \tau\}}\, w(y, s)
    + \chi_{\{s \leq \tau\}}\, f(y, \tau)
    \Big\}.
    \end{align}
\end{enumerate}
\end{lem}

\begin{proof}[Proof of Lemma \ref{lem:DPPHbar}] \quad 
\begin{enumerate}[label=$\mathrm{(\alph*)}$, leftmargin=0.7cm]
\item 
For $x \in \mathbb{R}^n$ and $t > 0$, let $\tau := t - \varepsilon$ (with $\varepsilon > 0$) and $y := x$ in \eqref{eqn:ocfw}. Then we obtain
\[
w(x, t) \leq \varepsilon \overline{L}(0)+ f(x, t-\varepsilon).
\]
Sending $\varepsilon \to 0$, we have $w(x,t) \leq f(x,t) = \overline{b}(x)$ for $t>0$. Moreover, \ref{itm:A5} implies $w(x,0) = g(x)\leq \overline{b}(x)$.

    \item We proceed by contradiction. If such $\ep_0$ does not exist, then there exists a subsequence $\{\tau_{\ep_k}\}$ such that 
    \begin{equation}\label{eqn:limtaueqt}
     \lim_{k\to \infty} \tau_{\ep_k}=t.   
    \end{equation}
    Let $y=x$ and $\tau=0$ in \eqref{eqn:ocfw}. Then $w(x, t) \leq t\overline{L}(0) + g(x)$. 
    Combining this with \eqref{eq:Hbar-Lbar}, \eqref{eqn:ftinc}, and \eqref{eqn:appw}, we obtain
\[
\begin{aligned}
\frac{|x-y_\ep|^2}{2(t-\tau_\ep)}-K_0(t-\tau_\ep)+g(y_\ep)&\leq (t-\tau_\ep) \overline{L}\left(\frac{x-y_\ep}{t-\tau_\ep}\right)+f(y_\ep, \tau_\ep)\\  
&\leq w(x,t) + \ep \leq t\overline{L}(0) + g(x) +\ep,
\end{aligned}    
\]
which implies
\begin{equation}\label{eqn:xyleqsqrtttau}
|x-y_\ep|\leq C |t-\tau_\ep|^{1/2},    
\end{equation}
for some constant $C=C\left(t, \overline{L}(0),K_0, \|g\|_{L^\infty(\R^n)}\right)>0$. Hence, from \eqref{eqn:limtaueqt}, $y_{\ep_k}\to x$ as $k\to\infty$. Since $t>0$, $\tau_{\ep_k}>0$ for $k$ large. Hence, $f(y_{\varepsilon_k}, \tau_{\varepsilon_k}) = \overline{b}(y_{\varepsilon_k})\to \overline{b}(x)$ as $k\to \infty$. 
By \eqref{eqn:appw}, along $\tau_{\varepsilon_k}$ we have
\[
\lim_{k\to \infty} (t-\tau_{\ep_k}) \overline{L}\left(\frac{x-y_{\ep_k}}{t-\tau_{\ep_k}}\right) = w(x, t) - \overline{b}(x)<0.
\]
But from \eqref{eq:Hbar-Lbar}, we have
\begin{align*}
    \lim_{k\to \infty} (t-\tau_{\ep_k}) \overline{L}\left(\frac{x-y_{\ep_k}}{t-\tau_{\ep_k}}\right) \geq \lim_{k\to \infty}  -(t-\tau_{\ep_k}) K_0 = 0,
\end{align*}
which contradicts the previous inequality.
Thus, \eqref{eqn:unifbdtauep} holds, i.e., there exists $\ep_0 > 0$ such that for all $0 < \ep < \ep_0${\red ,} we have $0 \leq \tau_\ep < \tau_{x, t} < t$ for some constant $\tau_{x, t} \in (0, t)$. Note that the constant $\tau_{x, t} \in (0, t)$ may depend on $x, t$. Finally, since $\tau_\ep  < \tau_{x, t}$ for all $\ep < \ep_0$, it follows from \eqref{eqn:xyleqsqrtttau} that
\[
\frac{|x-y_\ep|^2}{(t-\tau_\ep)^2}\leq 
\frac{C}{t-\tau_\ep}\leq \frac{C}{t-\tau_{x, t}}.
\]
Hence, $|x-y_\varepsilon|\leq C_{x,t}|t-\tau_\varepsilon|$ for some constant $C_{x, t}=C_{x, t}\left(x, t, \overline{L}(0), K_0, \|g\|_{L^\infty(\R^n)}\right)>0$.

    \item[(c)]If $w(x, t) <\overline{b}(x)$, then the conclusion follows from part (b).
    Suppose $w(x, t) =\overline{b}(x)$. Choose $y=x$ and $\tau_\varepsilon=t-\varepsilon$ for $\varepsilon>0$. Then
    \[
    \lim_{\varepsilon\to 0} \left(\left(t-\tau_\varepsilon\right)\ol{L}\left(\frac{x-y}{t-\tau_\varepsilon}\right)+ f(y,\tau_\varepsilon)\right) 
    = 
    \lim_{\varepsilon\to 0} \Big(\varepsilon \ol{L}(0)+f(x, \tau_\varepsilon)\Big) =\ol{b}(x).
    \]
    Clearly, $0=|y-x|\leq \left|t-\tau_\varepsilon\right|=\varepsilon$ for every $\varepsilon>0$.
    Thus, 
    \[
     w(x, t) = \inf
    \left\{
    (t - \tau)\, \overline{L}\left(\frac{x - y}{t - \tau}\right) + f(y, \tau): \tau \in [0, t), \left|x-y\right|\leq C_{x, t}\left|t-\tau\right|
    \right\},
    \]
    with $C_{x, t}=1$ in this case.
    \item[(d)]
Define $I:[0, t)\times (0,t) \times \R^n \to \mathbb{R}$ by
\begin{equation}\label{eq:Idef}
	I(\tau, s, y):= \left(t- (\tau \vee s)\right)\overline{L}\left( \frac{x-y}{t- (\tau \vee s)} \right) + \chi_{\{s > \tau\}} \, w\left(y, s \right) 
     + \chi_{\{s \leq \tau\}} \, f\left(y, \tau \right).
\end{equation}
We aim to show that for any $s \in (0, t)$,
\[
w(x, t)=\inf \left\lbrace I(\tau, s, y): \tau\in[0,t), y \in \R^n \right\rbrace.
\]
Fix $s \in (0, t)$. Let $\tau \in [0,t), \, y \in \R^n$. We show that $w(x,t)\leq I(\tau, s, y)$. If $\tau \geq s$, then 
\begin{equation}\label{eqn:wleqsltau}
	I(\tau, s, y)=(t-\tau) \overline{L}\left(\frac{x-y}{t-\tau}\right) + f(y, \tau) \geq w(x,t),
\end{equation}
by the definition of $w(x,t)$ in \eqref{eqn:ocfw}. If $\tau<s$, then
\[
I(\tau, s, y)=(t-s) \overline{L}\left(\frac{x-y}{t-s}\right)+w(y, s).
\]
For $w(y, s)$, take $\varepsilon>0$ and choose $\tau_\varepsilon \in [0, s)$, $z_\varepsilon \in \R^n$ such that
\[
w(y, s) + \varepsilon \geq (s-\tau_\varepsilon)\overline{L}\left(\frac{y-z_\varepsilon}{s-\tau_\varepsilon}\right)+f(z_\varepsilon, \tau_\varepsilon).
\]
Since $\overline{L}$ is convex and 
\begin{equation*}
	\frac{x-z_\varepsilon}{t-\tau_\varepsilon}
	=
	\left(1-\lambda\right) \frac{x-y}{t-s} 
	+ 
	\lambda \frac{y-z_\varepsilon}{s-\tau_\varepsilon}, \qquad\text{where}\qquad \lambda = \frac{s-\tau_\varepsilon}{t-\tau_\varepsilon}, 
\end{equation*}
we have
\[
\left(t-\tau_\varepsilon\right) \overline{L}\left(\frac{x-z_\varepsilon}{t-\tau_\varepsilon}\right) \leq (t-s) \overline{L}\left(\frac{x-y}{t-s}\right)+\left(s-\tau_\varepsilon\right)\overline{L}\left(\frac{y-z_\varepsilon}{s-\tau_\varepsilon}\right).
\]
Thus,
\[
\begin{aligned}
w(x,t) \leq& \left(t-\tau_\varepsilon\right) \overline{L}\left(\frac{x-z_\varepsilon}{t-\tau_\varepsilon}\right) + f(z_\varepsilon, \tau_\varepsilon)\\
\leq & (t-s) \overline{L}\left(\frac{x-y}{t-s}\right)+\left(s-\tau_\varepsilon\right)\overline{L}\left(\frac{y-z_\varepsilon}{s-\tau_\varepsilon}\right) + f(z_\varepsilon, \tau_\varepsilon)\\
\leq &(t-s) \overline{L}\left(\frac{x-y}{t-s}\right)+w(y, s) +\varepsilon.
\end{aligned}
\]
Since $\varepsilon$ is arbitrary, we conclude 
\begin{equation}\label{eqn:wleqsgtau}
    w(x,t) \leq (t-s) \overline{L}\left(\frac{x-y}{t-s}\right)+w(y, s)=I(\tau, s, y).
\end{equation}
Combining \eqref{eqn:wleqsltau} and \eqref{eqn:wleqsgtau}, we obtain $w(x,t)\leq I(\tau, s, y)$ for all $\tau\in[0,t)$ and $y\in\mathbb{R}^n$. Consequently, 
\[
	w(x, t) \leq \inf \left\lbrace I(\tau, s, y): \tau\in[0,t), y \in \R^n \right\rbrace.
\]

On the other hand, let $\varepsilon>0$. By the definition of $w(x, t)$, there exists $\tilde{\tau}_\varepsilon \in [0, t)$ and $\tilde{z}_\varepsilon \in \R^n$ such that 
\begin{equation}\label{eqn:wegz}
w(x, t)+\varepsilon \geq \left(t-\tilde{\tau}_\varepsilon\right)\overline{L}\left(\frac{x-\tilde{z}_\varepsilon}{t-\tilde{\tau}_\varepsilon}\right)+f(\tilde{z}_\varepsilon, \tilde{\tau}_\varepsilon).    
\end{equation}
\begin{enumerate}
    \item If $s\leq \tilde{\tau}_\varepsilon$, then the right-hand side of \eqref{eqn:wegz} equals $I(\tilde{\tau}_\varepsilon, s, \tilde{z}_\varepsilon)$, and we obtain
    \begin{equation}\label{eqn:wgeqiz}
    w(x+t) +\varepsilon \geq I(\tilde{\tau}_\varepsilon,s, \tilde{z}_\varepsilon).
    \end{equation}
    \item If $s>\tilde{\tau}_\varepsilon$, define 
	\begin{equation*}
		y_\ep = \frac{s-\tilde{\tau}_\varepsilon}{t-\tilde{\tau}_\varepsilon}x+\left(1-\frac{s-\tilde{\tau}_\varepsilon}{t-\tilde{\tau}_\varepsilon}\right)\tilde{z}_\varepsilon, 
		\qquad \text{then}\qquad 
		\frac{x-y_\ep}{t-s}=\frac{x-\tilde{z}_\varepsilon}{t-\tilde{\tau}_\varepsilon}=\frac{y_\ep-\tilde{z}_\varepsilon}{s-\tilde{\tau}_\varepsilon}. 
	\end{equation*}    
	Since $s>\tilde{\tau}_\varepsilon$, we have
	\begin{align}
		I(\tilde{\tau}_\ep, s, y_\ep)
			&=
		\left(t-s\right)\overline{L}\left(\frac{x-y_\ep}{t-s}\right)+w(y_\ep, s) \nonumber \\
			&\leq 
		(t-s) \overline{L}\left(\frac{x-y_\ep}{t-s}\right) +(s-\tilde{\tau}_\ep) \overline{L}\left(\frac{y_\ep-\tilde{z}_\ep}{s-\tilde{\tau}_\ep}\right) +f\left(\tilde{z}_\ep, \tilde{\tau}_\ep \right) \nonumber \\ 
			&= (t-\tilde{\tau}_\ep) \overline{L}\left(\frac{x-\tilde{z}_\ep}{t-\tilde{\tau}_\ep}\right) +f\left(\tilde{z}_\ep, \tilde{\tau}_\ep\right) \leq w(x,t) + \varepsilon \label{eqn:eqn:wgeqiy},
	\end{align}
    where the second line follows from the definition of $w(y_\ep, s)$.
\end{enumerate}
From \eqref{eqn:wgeqiz} and \eqref{eqn:eqn:wgeqiy}, for every $\ep>0$, we can find $\tau_\ep \in [0, t)$ and $y_\ep \in \R^n$ such that $w(x, t) +\ep \geq I(\tau_\ep, s, y_\ep)$. Therefore, it follows that
\[
	w(x, t) \geq \inf \left\lbrace I(\tau, s, y): \tau\in[0,t), y \in \R^n \right\rbrace
\]
and the proof is complete. 
\end{enumerate}
\end{proof}

Before proving Proposition \ref{prop:ValueU}, we show that, as a consequence of the Dynamic Programming Principle, $w$ is uniformly continuous on $\mathbb{R}^n \times [0,\infty)$ under assumptions \ref{itm:A1}--\ref{itm:A5}, and Lipschitz continuous on $\mathbb{R}^n \times [0,\infty)$ under assumptions \ref{itm:A1}--\ref{itm:A6}.

\begin{prop}\label{prop:ocfcont} Assume \ref{itm:A1}--\ref{itm:A5}.
\begin{enumerate}[leftmargin=0.7cm]
    \item[$\mathrm{(a)}$] The function $w$ defined in \eqref{eqn:ocfw} is uniformly continuous on $\mathbb{R}^n \times [0, \infty)$. In particular, for any $t \geq 0$, the map $x \mapsto w(x,t)$ is uniformly continuous, and for any $x \in \mathbb{R}^n$, the map $t \mapsto w(x,t)$ is Lipschitz continuous, where the Lipschitz constant depends only on $\left\|Dg\right\|_{L^\infty(\mathbb{R}^n)}, \overline{L}(0), K_0$.
    \item[$\mathrm{(b)}$] Assume further that \ref{itm:A6} holds. Then $w$ is Lipschitz continuous on $\mathbb{R}^n \times [0, \infty)$, where $\mathrm{Lip}(w)$ depends only on $\left\|Dg\right\|_{L^\infty(\mathbb{R}^n)}, \left\|Db\right\|_{L^\infty(\mathbb{R}^n \times \partial \Om)},  \overline{L}(0), K_0$.
\end{enumerate}
\end{prop}

\begin{proof}[Proof of Proposition \ref{prop:ocfcont}]
We present the proof of $\mathrm{(a)}$ only, as the proof of  $\mathrm{(b)}$ follows the same steps, with the additional consideration of tracking the Lipschitz constant of $f$. 

We first show that $x\mapsto w(x,t)$ is uniformly continuous for $t\geq 0$. Let $x, \tilde{x}\in \mathbb{R}^n$. If $t=0$ then 
\begin{equation*}
	\left|w(x, 0)-w(\tilde{x}, 0)\right|=\left|g(x)-g(\tilde{x})\right| \leq \left\|Dg\right\|_{L^\infty\left(\mathbb{R}^n\right)} |x -\tilde{x}|. 
\end{equation*}
If $t>0$, for every $\varepsilon>0$, there exists $y_\varepsilon \in \mathbb{R}^n$ and $\tau_\varepsilon \in [0, t)$ such that 
 \begin{align*}
 	 (t-\tau_\varepsilon)\overline{L}\left(\frac{x-y_\varepsilon}{t-\tau_\varepsilon}\right) + f\left(y_\varepsilon, \tau_\varepsilon\right)\leq  w(x, t) + \varepsilon,
 \end{align*}
 and by the definition of $w$,
 \begin{align*}
 	w(\tilde{x},t) 
 	\leq 
 	(t-\tau_\varepsilon)\overline{L}\left(\frac{\tilde{x} - (\tilde{x}-x+y_\varepsilon)}{t-\tau_\varepsilon}\right)  + f(\tilde{x}-x+y_\varepsilon, \tau_\varepsilon). 
 \end{align*}
 Therefore
 \begin{align*}
 	w(\tilde{x}, t)-w(x, t) \leq f(\tilde{x}-x+y_\varepsilon, \tau_\varepsilon) - f(y_\varepsilon, \tau_\varepsilon) + \varepsilon \leq \omega_f(|\tilde{x} - x|) + \varepsilon, 
 \end{align*}
where $\omega_f$ is as in \eqref{eqn:fucinx}. Since $\varepsilon>0$ is arbitrary, we have $w(\tilde{x}, t)-w(x, t) \leq \omega_f \left(\left|\tilde{x}-x\right|\right)$. By a symmetric argument, we have $\left|w(x, t)-w(\tilde{x}, t)\right| \leq \omega_f \left(\left|\tilde{x}-x\right|\right)$. 
Thus, $x\mapsto w(x, t)$ is uniformly continuous for all $t \geq 0$. \medskip

Next, we show that $t\mapsto w(t,x)$ is uniformly continuous for all $x \in \mathbb{R}^n$. 
Let $y=x$ and $\tau = 0$ in \eqref{eqn:ocfw}. Then
\begin{align}\label{eq:g1}
	w(x,t) \leq t\overline{L}(0) + g(x).
\end{align}
Since $g \in \mathrm{Lip}(\mathbb{R}^n)$ by \ref{itm:A5} and $\overline{H}$ is continuous, we have $\overline{H}(Dg(x)) \leq \tilde{C}$ for a.e. $x \in \mathbb{R}^n$, where
\begin{equation}\label{eq:tildeC}
\tilde{C} = \frac{1}{2} \|Dg\|_{L^\infty(\R^n)}^2 + K_0
\end{equation}
and $K_0$ is as given in \eqref{eq:Hbar-Lbar}. Let $\tau\in [0,t)$ and $y\in \R^n$. Define $\gamma(s) = y+\frac{x-y}{t-\tau}(s-\tau)$ for $s\in [\tau, t]$
be the straight line connecting $y$ to $x$ in time $s\in[\tau,t]$. By a smoothing argument (or using \cite[Proposition 2.5]{ishii_asymptotic_2008}, which ensures that $s \mapsto g(\gamma(s))$ is absolutely continuous), we obtain
\begin{align*}
	g(x) - g(y) 
	&= \int_\tau^t Dg(\gamma(s))\cdot \dot{\gamma}(s)\; \\
	&\leq \int_{\tau}^t\Big( \overline{L}(\dot{\gamma}(s))\;ds + \overline{H}(Dg(\gamma(s))) \Big)\;ds 
	\leq 
	(t-\tau)\left( \tilde{C} + \overline{L}\left(\frac{x-y}{t-\tau}\right) \right)\;. 
\end{align*}
Rearranging the terms and using $g(y)\leq f(y,\tau)$ from \eqref{eqn:ftinc}, we obtain 
\begin{align*}
    g(x) \leq (t-\tau)\,\overline{L}\!\left(\tfrac{x-y}{t-\tau}\right) + f(y,\tau) + \tilde{C}(t-\tau).
\end{align*}
Since $\tau \in [0,t)$ and $y \in \mathbb{R}^n$ are arbitrary, it follows that  
\begin{equation}\label{eq:g2}
    g(x) \leq w(x,t) \leq \tilde{C}t.
\end{equation}
Combining \eqref{eq:g1} and \eqref{eq:g2}, we obtain
\begin{equation}\label{eq:g3}
    |w(x,t) - g(x)| \leq Ct,
\end{equation}
where $C = \max \left\{ |\overline{L}(0)|, \tfrac{1}{2}\,\|Dg\|^2_{L^\infty(\mathbb{R}^n)}+K_0 \right\}$. In other words, $t\mapsto w(x,t)$ is Lipschitz at $t=0$.

For $t>0$, fix $s\in(0,t)$ and split the infimum in \eqref{eq:DPPHbar} over $\tau\in[0,s)$ and $\tau\in[s,t)$, giving  
\begin{equation}\label{eq:minAB}
	w(x,t) = \min\{A,B\}, 
\end{equation}
where 
\begin{align*}
	A &:= \inf_{\tau\in [0, s), \, y\in \R^n} \left\{(t-s) \overline{L}\left(\frac{x-y}{t-s}\right) + w(y, s)\right\} 
    =\inf_{y\in \R^n} \left\{(t-s) \overline{L}\left(\frac{x-y}{t-s}\right) + w(y, s)\right\},   \\
    B &:= \inf_{\tau\in [s, t), \, y\in \R^n} \left\{(t-\tau) \overline{L}\left(\frac{x-y}{t-\tau}\right) + f(y, \tau)\right\}.
\end{align*}
If $w(x,t) =A$, then by \eqref{eq:g3},    \begin{equation}\label{eqn:weqA}
    \begin{aligned}
     w(x, t) &= \inf_{y\in \R^n} \left\{(t-s) \overline{L}\left(\frac{x-y}{t-s}\right) + g(y)+ w(y, s)-g(y)\right\}\\
     &\geq \inf_{y\in \R^n} \left\{(t-s) \overline{L}\left(\frac{x-y}{t-s}\right) + g(y)\right\}-Cs\\
     &\geq \inf_{\tilde{\tau} \in[0,t-s), \, y\in \R^n} \left\{(t-s-\tilde{\tau}) \overline{L}\left(\frac{x-y}{t-s-\tilde{\tau}}\right) + f\left(y,\tilde{\tau}\right)\right\}-Cs\\
     &= w(x, t-s)-Cs,
    \end{aligned}   
    \end{equation}
    where
    $C:=\max\left\{\left|\overline{L}(0)\right|,\frac{\left\|Dg\right\|^2_{L^\infty(\R^n)}}{2} + K_0\right\}$ and the fourth line follows from the definition of $w(x, t-s)$.
    If $w(x,t) =B$, then
    \begin{equation}\label{eqn:weqB}
        \begin{aligned}
    w(x, t)&= \inf_{\tau\in [s, t), \, y\in \R^n} \left\{(t-\tau) \overline{L}\left(\frac{x-y}{t-\tau}\right) + f(y, \tau)\right\}\\
    &= \inf_{\tilde{\tau}\in[0, t-s),\, y\in \mathbb{R}^n} \left\{(t - s- \tilde{\tau})\, \overline{L}\left(\frac{x - y}{t -s- \tilde{\tau}}\right) + f(y, \tilde{\tau}+s)\right\}\\
    &\geq \inf_{\tilde{\tau}\in[0, t-s),\, y\in \mathbb{R}^n} \left\{(t - s- \tilde{\tau})\, \overline{L}\left(\frac{x - y}{t -s- \tilde{\tau}}\right) + f(y, \tilde{\tau})\right\}\\
    &= w(x, t-s),
    \end{aligned}    
    \end{equation}
    where the third line follows from \eqref{eqn:ftinc}.
    Combining \eqref{eqn:weqA} and \eqref{eqn:weqB}, we conclude
    \begin{equation}\label{eqn:wlipintg}
     w(x, t) \geq w(x, t-s) -Cs,
    \end{equation}
    where
    $C:=\max\left\{\left|\overline{L}(0)\right|,\frac{\left\|Dg\right\|^2_{L^\infty(\R^n)}}{2} + K_0\right\}$.
    
For the other direction of \eqref{eqn:wlipintg}, by the definition of $w$, we have
    \begin{equation}\label{eqn:wtsbreak}
    w(x, t-s) = \min \left\{A' , B'\right\}  ,  
    \end{equation}
    where
    \begin{align*}
    	A'&:=\inf_{y\in \R^n} \left\{(t-s) \overline{L}\left(\frac{x-y}{t-s}\right) +g(y)\right\},\\
    	B'&:=\inf_{\tilde{\tau} \in(0,t-s), \, y\in \R^n} \left\{(t-s-\tilde{\tau}) \overline{L}\left(\frac{x-y}{t-s-\tilde{\tau}}\right) + f\left(y,\tilde{\tau}\right)\right\}.
    \end{align*}    	
    From \eqref{eq:minAB}, we have $w(x, t) \leq A$, that is,
    \begin{equation}\label{eqn:wleqAp}
    \begin{aligned}
    w(x, t) & \leq  \inf_{y\in \R^n} \left\{(t-s) \overline{L}\left(\frac{x-y}{t-s}\right) + g(y)+w(y, s)-g(y)\right\}\\
    &\leq \inf_{y\in \R^n} \left\{(t-s) \overline{L}\left(\frac{x-y}{t-s}\right) + g(y)\right\}+Cs
    =A'+Cs
    \end{aligned}        
    \end{equation}
    where the second line follows from \eqref{eq:g3}.
    Similarly, from \eqref{eq:minAB} we also have $w(x, t) \leq B$, that is,
    \begin{equation}\label{eqn:wleqBp}
    \begin{aligned}
    w(x, t) & \leq \inf_{\tau\in [s, t),\, y\in \R^n} \left\{(t-\tau) \overline{L}\left(\frac{x-y}{t-\tau}\right) + f(y, \tau)\right\}\\
    &=\inf_{\tilde{\tau}\in[0, t-s),\, y\in \mathbb{R}^n} \left\{(t - s- \tilde{\tau})\, \overline{L}\left(\frac{x - y}{t -s- \tilde{\tau}}\right) + f(y, \tilde{\tau}+s)\right\}\\
    &\leq \inf_{\tilde{\tau}\in(0, t-s),\, y\in \mathbb{R}^n} \left\{(t - s- \tilde{\tau})\, \overline{L}\left(\frac{x - y}{t -s- \tilde{\tau}}\right) + f(y, \tilde{\tau}+s)\right\}\\
    &= \inf_{\tilde{\tau}\in(0, t-s),\, y\in \mathbb{R}^n} \left\{(t - s- \tilde{\tau})\, \overline{L}\left(\frac{x - y}{t -s- \tilde{\tau}}\right) + f(y, \tilde{\tau})\right\} = B',
    \end{aligned}    
    \end{equation}
    where we exclude the case $\tilde{\tau}=0$ in the third line, and the fourth line follows from the definition of $f$, namely $f(y, \tilde{\tau}+s)=f(y, \tilde{\tau})=\overline{b}(\tilde{y})$ for $\tilde{\tau}, s >0$. 
    Combining \eqref{eqn:wtsbreak}, \eqref{eqn:wleqAp} and \eqref{eqn:wleqBp}, we have
    \begin{equation}\label{eqn:wlipintl}
    \begin{aligned}
    w(x, t) &\leq \min\left\{A', B'\right\} +Cs = w(x, t-s)+Cs.
    \end{aligned}        
    \end{equation}
    Hence, combining \eqref{eqn:wlipintg} and \eqref{eqn:wlipintl}, we conclude that for any $t> s >0$,
    \[
    \left| w(x, t) -w(x, t-s)\right| \leq  C s,
    \]
    where $C:=\max\left\{\left|\overline{L}(0)\right|,\frac{\left\|Dg\right\|^2_{L^\infty(\R^n)}}{2} + K_0\right\}$. 
\end{proof}

\begin{proof}[Proof of Proposition \ref{prop:ValueU}] 
We first show that $w$ defined in \eqref{eqn:ocfw} is a subsolution to \eqref{eqn:PDE_limit}. 

Let $\varphi\in C^1(\R^n\times (0,\infty))$ be such that $w-\varphi$ attains its maximum at $(x_0,t_0) \in \R^n\times (0,\infty)$ and $(w-\varphi)(x_0,t_0) = 0$. We aim to show that
\[
\max\big\{w(x_0,t_0) - \overline{b}(x_0), \partial_t \varphi(x_0,t_0) + \overline{H}\left( D\varphi(x_0,t_0)\right)\big\} \leq 0.
\]
By Lemma \ref{lem:DPPHbar}, we know $w(x_0,t_0) \leq \overline{b}(x_0)$. Hence, it suffices to prove
\[
 \partial_t \varphi(x_0,t_0) + \overline{H}\left( D\varphi(x_0,t_0)\right) \leq 0.
\]
Let $v\in \R^n$ and $s\in (0, t_0)$. Choose $y=x_0+v(s-t_0) $ and $\tau\in[0, s)$ in \eqref{eq:DPPHbar} to obtain
\begin{align*}
    \varphi(x_0,t_0)=w(x_0,t_0) 
    &\leq w(x_0+v(s-t_0), s) + (t_0-s) \overline{L}\left(v\right) \\
    &\leq \varphi(x_0+v(s-t_0), s) + (t_0-s) \overline{L}\left(v\right).
\end{align*}
Therefore,
\begin{align*}
    \frac{\varphi(x_0,t_0) - \varphi(x_0+v(s-t_0), s)}{t_0-s} 
    \leq \overline{L}\left(v\right)
    \qquad\Longrightarrow\qquad 
    D\varphi(x_0,t_0)\cdot v + \partial_t \varphi(x_0,t_0) \leq \overline{L}\left(v\right)
\end{align*}
by letting $s\to t_0$. We deduce 
\begin{align*}
    \partial_t \varphi(x_0,t_0) + 
    \left(
    D\varphi(x_0,v_0)\cdot v - \overline{L}(v) 
    \right) \leq 0 
    \qquad \Longrightarrow \qquad
    \partial_t \varphi(x_0,t_0) + \overline{H}\left(D\varphi(x_0,t_0)\right) \leq 0
\end{align*}
since $v\in \mathbb{R}^n$ is arbitrary. \medskip

Next, we show that $w$ is a supersolution to \eqref{eqn:PDE_limit}. Let $\psi\in C^1(\R^n\times (0,\infty))$ be such that $w-\psi$ attains a minimum at $(x_0,t_0) \in \R^n\times (0,\infty)$ with $(w-\psi)(x_0,t_0) = 0$. We aim to show that 
\begin{equation*}
    \max\big\{
    w(x_0,t_0) - \overline{b}(x_0), \partial_t\psi(x_0,t_0) + H(x_0,D\psi(x_0,t_0))
    \big\} \geq 0. 
\end{equation*}
By Lemma \ref{lem:DPPHbar}, we know $w(x_0,t_0) \leq \overline{b}(x_0)$. If $w(x_0,t_0) = \overline{b}(x_0)$, we are done. Otherwise, assume $w(x_0,t_0) < \overline{b}(x_0)$. For every $m \in \mathbb{N}$, there exist $\tau_m \in [0,t_0)$ and $y_m \in \R^n$ such that
\begin{equation}\label{eqn:appu}
    w(x_0,t_0) \leq (t_0-\tau_m) \overline{L}\left(\frac{x_0-y_m}{t_0-\tau_m}\right)+f(y_m, \tau_m)  \leq w(x_0,t_0) + \frac{1}{m^2}.    
\end{equation}
By Lemma \ref{lem:DPPHbar}, there exists $m_0 \in \mathbb{N}$ such that for all $m > m_0$,
\begin{equation}\label{eqn:unifbdtaum}
    0 \leq \tau_m \leq t_0 - \frac{1}{m_0} < t_0 - \frac{1}{m} < t_0,
\end{equation}
Moreover, for $m> m_0$,
\begin{equation}\label{eqn:xydist}
    \frac{|x_0-y_m|}{t_0-\tau_m}\leq C
\end{equation}
for some constant $C=C\left(x_0, t_0, H, \|g\|_{L^\infty(\R^n)} \right)>0$. 
For $m>m_0$, define 
\begin{equation*}
	z_m = (1-\lambda) x_0 + \lambda y_m, \qquad\text{where}\;\lambda = \frac{1}{m}\cdot \frac{1}{t_0-\tau_m}. 
\end{equation*}
Then 
\begin{equation}\label{eqn:slveleq}
m\left(x_0-z_m\right)=\frac{x_0-y_m}{t_0-\tau_m}=\frac{z_m-y_m}{t_0-\frac{1}{m}-\tau_m},
\end{equation}
and from \eqref{eqn:xydist}, it follows that
\begin{equation}\label{eqn:xzdist}
    |x_0-z_m| \leq \frac{C}{m}
\end{equation}
for some constant $C=C\left(x_0, t_0, H, \|g\|_{L^\infty(\R^n)} \right)>0$. By \eqref{eqn:ocfw} and \eqref{eqn:unifbdtaum}, we have
\begin{align*}
 \psi\left(z_m,t_0-\frac{1}{m} \right) \leq w\left(z_m,t_0-\frac{1}{m} \right)
 	&\leq \left(t_0-\frac{1}{m}-\tau_m\right) \overline{L}\left(\frac{z_m-y_m}{t_0-\frac{1}{m}-\tau_m}\right)+f(y_m, \tau_m) \\
 	&\leq -\frac{1}{m}\overline{L}\left(\frac{x_0-y_m}{t_0-\tau_m}\right) + w(x_0,t_0) + \frac{1}{m^2},
\end{align*}
thanks to \eqref{eqn:slveleq} and \eqref{eqn:appu}. We deduce that 
\begin{equation}\label{eqn:psi_ineq}
\begin{aligned}
    \frac{1}{m} \overline{L}\left(m(x_0-z_m)\right) + \psi\left(z_m,t_0-\frac{1}{m} \right)
    \leq w(x_0, t_0) +\frac{1}{m^2} = \psi(x_0, t_0)+\frac{1}{m^2}.
\end{aligned}
\end{equation}
Rearranging terms in~\eqref{eqn:psi_ineq}, we get
\begin{align*}
	- \frac{1}{m^2} 
	& \leq 
	 \psi (x_0,t_0) -\psi\left(z_m,t_0-\frac{1}{m} \right) - \frac{1}{m} \overline{L}\left(m(x_0-z_m)\right) \\ 
	& = \int_{t_0-\frac{1}{m}}^{t_0} 
    \Bigg[
        \partial_t\psi \left(
            z_m + m(x_0-z_m)\left(s-t_0+\tfrac{1}{m}\right),\, s
        \right) \\
    &\qquad \qquad\qquad+ D\psi \left(
            z_m + m(x_0-z_m)\left(s-t_0+\tfrac{1}{m}\right),\, s
        \right) \cdot m(x_0-z_m) - \overline{L}\big(m(x_0-z_m)\big)
    \Bigg] \, ds  \\
    &\leq 
    \int_{t_0-\frac{1}{m}}^{t_0} 
    \Bigg[
        \partial_t\psi \Big(
            z_m + m(x_0-z_m)\big(s-t_0+\tfrac{1}{m}\big),\, s
        \Big)\\
        &\qquad \qquad\qquad+ \overline{H}\Big( D\psi \Big(
            z_m + m(x_0-z_m)\left(s-t_0+\tfrac{1}{m}\right),\, s \Big) \Big)
    \Bigg] \, ds .
\end{align*}
Note that for $s \in \left[t_0-\frac{1}{m}, t_0\right]$, 
\begin{equation}\label{eqn:xsldist}
    \begin{aligned}
    \left|z_m + m(x_0-z_m)\left(s-t_0+\tfrac{1}{m}\right)-x_0\right|=&|(x_0-z_m)m(s-t_0)|
    \leq |x_0-z_m| \leq  \frac{C}{m}.
    \end{aligned}
\end{equation}
By the continuity of $\partial_t \psi$, $D\psi$, and $\overline{H}$ at $(x_0, t_0)$, and using \eqref{eqn:xsldist}, we deduce that for $m > m_0$,
\begin{align*}
    -\frac{1}{m^2} \leq \int_{t_0-\frac{1}{m}}^{t_0} 
    \left(
        \partial_t\psi \left(x_0, t_0
        \right)+ \overline{H}\left( D\psi \left(
            x_0, t_0 \right) \right)+ \omega \left(\frac{1}{m}\right)
    \right)\, ds,
\end{align*}
where $\omega$ is a modulus of continuity depending on \( \partial_t \psi  (x_0, t_0) \), \( D\psi  (x_0, t_0) \), and \( \overline{H}(D\psi(x_0, t_0)) \).
Multiplying both sides by $m$ yields
\begin{align*}
    -\frac{1}{m} \leq m\int_{t_0-\frac{1}{m}}^{t_0} 
    \left(
        \partial_t\psi \left(x_0, t_0
        \right)+ \overline{H}\left( D\psi \left(
            x_0, t_0 \right) \right)+ \omega \left(\frac{1}{m}\right)
    \right)\, ds .
\end{align*}
Sending $m\to \infty$, we obtain $\partial_t \psi(x_0,t_0) + \overline{H}(D\psi(x_0,t_0)) \geq 0$.
\end{proof}

\subsection{Optimal control for $u^\varepsilon$}
We present the optimal control formula for $u^\varepsilon$. For $\varepsilon>0$, the solution $u^\varepsilon$ to \eqref{eqn:PDE_epsilon} admits the representation formula (see \cite{Ishii1989_dirichletHJ, mitake_large_2009}) 
\begin{align*}
    u^\varepsilon(x,t) = \inf\left\lbrace 
	    \int_\tau^t L\left( \frac{\xi(s)}{\varepsilon} ,\dot{\xi}(s)\right) \,ds
	    	+
	    	f_\varepsilon \left(\xi(\tau),\tau \right)
	    	:  
	    \tau \in [0, t], \xi\in \mathcal{C}(x,t;\overline{\Omega}_\varepsilon), (\xi(\tau), \tau) \in \partial Q^\varepsilon_t
    \right\rbrace, 
\end{align*}
for any $x\in \overline{\Omega}_\varepsilon$ and $t \geq 0$. The following proposition shows that for $t>0$, the case \(\tau = t\) can be excluded, which will simplify later arguments.


\begin{prop}\label{prop:uepocf}
Assume \ref{itm:A1}--\ref{itm:A5}. Let $\varepsilon > 0$, and $(x,t)\in \overline{\Om}_\ep\times (0,\infty)$. We have
\begin{align} \label{eqn:ocfue} 
    &u^\varepsilon(x,t) 
    = 
    \inf
    \left\lbrace 
	    \int_\tau^t L\left( \frac{\xi(s)}{\varepsilon} ,\dot{\xi}(s)\right) \,ds
	    	+
	    	f_\varepsilon \left(\xi(\tau),\tau \right)
	    	:  
	    \tau \in [0, t), \xi\in \mathcal{C}(x,t;\overline{\Omega}_\varepsilon), (\xi(\tau), \tau) \in \partial Q^\varepsilon_t
    \right\rbrace \nonumber \\
    &\quad = \inf_{\tau\in [0,t)}
    \left\lbrace
    	\varepsilon \int_{\frac{\tau}{\varepsilon}}^{\frac{t}{\varepsilon}} 
    		L\left(\gamma(s), \dot{\gamma}(s)\right)\;ds + 	
    		f_\varepsilon\left(\varepsilon\gamma\left(\tfrac{\tau}{\varepsilon}\right), \tau\right): 
    		\gamma \in 
    		\mathcal{C}\left(\tfrac{x}{\varepsilon}, \tfrac{t}{\varepsilon};\overline{\Omega}\right), 
    		\left(\gamma(\tfrac{\tau}{\varepsilon}), \tau\right) \in \partial Q_t
    \right\rbrace, 
\end{align}
where $Q^\varepsilon_t$ and $\partial Q^\varepsilon_t$ are defined in \eqref{eq:QTeps}. 
\end{prop}

Note that the admissible paths in the optimal control formula for $u^\varepsilon$ are restricted on $\overline{\Omega}_\varepsilon$ or $\ol \Om$ after rescaling.

\begin{proof} Define
\begin{align*}
	A &:= 
    \left\lbrace 
	    \int_\tau^t L\left( \frac{\xi(s)}{\varepsilon} ,\dot{\xi}(s)\right) \,ds
	    	+
	    	f_\varepsilon \left(\xi(\tau),\tau \right)
	    	:  
	    \tau \in [0, t], \xi\in \mathcal{C}(x,t;\overline{\Omega}_\varepsilon), (\xi(\tau), \tau) \in \partial Q^\varepsilon_t
    \right\rbrace, \\ 
    B &:= 
    \left\lbrace 
	    \int_\tau^t L\left( \frac{\xi(s)}{\varepsilon} ,\dot{\xi}(s)\right) \,ds
	    	+
	    	f_\varepsilon \left(\xi(\tau),\tau \right)
	    	:  
	    \tau \in [0, t), \xi\in \mathcal{C}(x,t;\overline{\Omega}_\varepsilon), (\xi(\tau), \tau) \in \partial Q^\varepsilon_t
    \right\rbrace. 
\end{align*}
For $\varepsilon>0$, the solution $u^\varepsilon$ to \eqref{eqn:PDE_epsilon} admits the representation formula (see \cite{Ishii1989_dirichletHJ, mitake_large_2009}) $u^\varepsilon(x,t)= \inf A$ for any $x\in \overline{\Omega}_\varepsilon$ and $t \geq 0$. 
For any $t>0$ and $x\in \overline{\Om}_\ep$, we claim that $\inf A = \inf B$.

    Indeed, if $x\in \Omega_\varepsilon$ and $\tau=t$, there is no admissible path $\xi \in \mathcal{C}(x,t;\overline{\Omega}_\varepsilon)$ with $\xi(t)\in \partial\Omega_\varepsilon$. Hence, $A=B$ in this case. If $x \in \partial \Omega_\varepsilon$ and $\tau = t$, then for any path $\xi \in \mathcal{C}(x,t;\overline{\Omega}_\varepsilon)$ with $\xi(t) = x$, we obtain $f_\ep(x,t) \in A$. 

On the other hand, by the $C^1$-regularity of $\partial \Omega_\varepsilon$, there exist sequences $\{\tau_k\}_{k=1}^\infty \subset [0, t)$ and $\{y_k\}_{k=1}^\infty \subset \partial \Omega_\varepsilon$ such that 
    \[
    \tau_k \to t, \quad y_k \to x, \quad |y_k - x| \le M |t - \tau_k|
    \]
    for some constant $M > 0$. Without loss of generality, we may assume that $y_k$ and $x$ lie in the same connected component of $\partial \Omega_\varepsilon$. Consider the curve $\xi_k: [0, t] \to \R^n$ defined by 
    \[
    \xi_k (s) = \begin{cases}
			y_k, & \text{if } \, s \in [0, \tau_k),\\
            \frac{s-\tau_k}{t-\tau_k}(x-y_k)+y_k, & \text{if } \,  s\in [\tau_k, t],
		 \end{cases}
    \]
    which satisfies $\big\|\dot{\xi}\big\|_{L^\infty([0, t])} \leq M$. By \cite[Lemma A.1]{han_quantitative_2024}, we may modify $\xi_k$ into a $C^1$ path $\tilde{\xi}_k : [0,t] \to \partial\Omega_\varepsilon$ joining $y_k$ to $x$ with
    \[
    \tilde{\xi}_k(\tau_k)= y_k,\qquad \tilde{\xi}_k(t)=x, \qquad \left\|\dot{\tilde{\xi}}_k\right\|_{L^\infty([0, t])} \leq C_b M,
    \]
    for a constant $C_b>0$ depending only on $n$ and $\partial \Omega_\varepsilon$. By \eqref{eqn:K_0L}, we have
    \[
      -(t-\tau_k)K_0\leq \int_{\tau_k}^t L\left( \frac{\tilde{\xi}_k(s)}{\varepsilon} ,\dot{\tilde{\xi}}_k(s)\right) \,ds  \leq (t-\tau_k) \frac{C_b^2M^2}{2} + (t-\tau_k)K_0.
    \]
    Sending $k \to \infty$, we obtain
    \[
    \lim_{k\to \infty} \int_{\tau_k}^t L\left( \frac{\tilde{\xi}_k(s)}{\varepsilon} ,\dot{\tilde{\xi}}_k(s)\right) \,ds + \fep \left(\tilde{\xi}(\tau_k), \tau_k\right) =\fep (x,t),
    \]
    with
    \[
    \int_{\tau_k}^t L\left( \frac{\tilde{\xi}_k(s)}{\varepsilon} ,\dot{\tilde{\xi}}_k(s)\right) \,ds + \fep \left(\tilde{\xi}_k(\tau_k), \tau_k\right) \in B
    \]
    for each $k \in \mathbb{N}$. Hence, $\inf A=\inf B$.
\end{proof}

We state the Dynamic Programming Principle for $u^\varepsilon$ (see \cite[Theorem 6.1]{barles_introduction_2013}) and omit its proof.
\begin{lem} \label{lem:uepDPP}
Assume \ref{itm:A1}--\ref{itm:A5}. Let $\ep>0$ and let $u^\ep$ be the solution to \eqref{eqn:PDE_epsilon}. Then for any $t>s>0$ and $x\in \ol\Om_\ep$, we have
\begin{equation}\label{eq:DPP}
\begin{aligned}
    u^\ep(x,t) 
    &= \inf_{
    \substack{
        \tau\in [0,t), \xi\in \mathcal{C}(x,t;\overline{\Omega}_\varepsilon) \\
    \left(\xi (\tau), \tau\right) \in \partial Q^\varepsilon_t}}
    \left \lbrace 
     \int_{ \tau \vee s }^t L\left( \frac{\xi(r)}{\ep}, \dot{\xi}(r) \right) \, dr 
     + \chi_{\{s > \tau\}} \, u^\ep\left( \xi(s), s \right) 
     + \chi_{\{s\leq \tau\}} \, f_\ep\left( \xi(\tau), \tau \right) \right\rbrace .
\end{aligned}    
\end{equation}
\end{lem}

The following result provides a gradient estimate for $u^\varepsilon$. This follows from the fact that $u^\varepsilon$ solves \eqref{eqn:PDE_epsilon} and that the comparison principle holds (see \cite{mitake_large_2009} and \cite[Theorem~2.1]{Ishii1989_dirichletHJ}). Alternatively, the Lipschitz continuity of $u^\varepsilon$ may be derived directly from its optimal control representation under our assumptions, without invoking the comparison principle. The proof is therefore omitted.

\begin{prop}\label{prop:ueplip}
Assume \ref{itm:A1}--\ref{itm:A5}. Let $\varepsilon>0$. Then,
\[
\left\|u^\varepsilon_t\right\|_{L^\infty\left(\overline{\Om}_\ep \times [0, \infty)\right)}+\left\|Du^\varepsilon\right\|_{L^\infty(\Om_\ep \times [0, \infty))} \leq C_0,
\]
where $C_0=C_0\left(H, \mathrm{Lip}(g)\right)>0$ is a constant independent of $\ep$.
\end{prop}

The following lemma (cf. \cite[Lemma A.2]{han_quantitative_2024} for the state-constraint case) shows that all the optimal paths for $u^\varepsilon$ admit a uniform velocity bound, which will be useful in simplifying the optimal control formula for $u^\varepsilon$. 

\begin{lem}\label{lem:velocity bound}
Assume \ref{itm:A1}--\ref{itm:A5}. Let $\varepsilon, t>0$ and $x\in {\overline{\Omega}_\ep}$. 
Suppose $\gamma: \left[\tau, \frac{t}{\varepsilon}\right] \to \overline{\Omega}$ is a minimizing curve for $u^\varepsilon(x, t)$ in the sense of \eqref{eqn:ocfue}, with $\gamma$ absolutely continuous for some $\tau \in [0, t)$, and
\begin{equation}
u^\varepsilon (x, t) = \varepsilon \int^\frac{t}{\varepsilon
}_{\frac{\tau}{\varepsilon}} L\left(\gamma(s),\dot{\gamma}(s)\right) \,ds+\fep \left(\varepsilon \gamma\left(\frac{\tau}{\varepsilon}\right),\tau\right) , 
	\qquad \gamma\left(\tfrac{t}{\varepsilon}\right)=\tfrac{x}{\varepsilon}, 
	\qquad \gamma\left(\tfrac{\tau}{\varepsilon}, \tau\right) \in \partial Q_t. 
\end{equation}
Then, there exists a constant $M_0=M_0\left(H, \|Dg\|_{L^\infty(\mathbb{R}^n)}\right)>0$ such that 
\begin{align}\label{eq:LipMinimizers}
	\left|\dot{\gamma} (s) \right |\leq M_0 \qquad\text{for a.e.}\;s \in \left[\tfrac{\tau}{\varepsilon}, \tfrac{t}{\varepsilon}\right]. 
\end{align}
\end{lem}

\begin{rem} The bound in \eqref{eq:LipMinimizers} is independent of $n$ and $\partial\Omega$. While \cite[Remark~7]{han_quantitative_2024} notes this fact in the state-constraint setting without proof, and their main argument yields a bound depending on $n$ and $\Omega$, our approach—built upon Lemma~\ref{lem:acdiff} and extending \cite[Proposition~4.1]{mitake_asymptotic_2008}—provides a direct proof of this independence.
\end{rem}

\begin{proof}[Proof of Lemma \ref{lem:velocity bound}] Define $\eta:[0,t]\to \overline{\Omega}_\varepsilon$ by $\eta(s) = \varepsilon \gamma\!\left(\tfrac{s}{\varepsilon}\right)$ for $s\in [0,t]$. Then $\eta(t) = x$, $(\eta(\tau), \tau)\in \partial Q^\varepsilon_t$, and 
\begin{align*}
    u^\varepsilon(x,t) 
    &= \int_\tau^t L\left(\frac{\eta(s)}{\varepsilon}, \dot{\eta}(s)\right)\,ds 
       + f_\varepsilon(\eta(\tau), \tau),
\end{align*}
so it remains to show that $|\dot{\eta}(s)| \leq M_0$ for a.e. $s \in [0,t]$ for some $M_0=M_0\left(H, \|Dg\|_{L^\infty(\mathbb{R}^n)}\right)>0$. Take $h>\tau$. Using $\eta$ as an admissible path in 
\eqref{eq:DPP} of Lemma \ref{lem:uepDPP}, we obtain
\begin{align*}
	u^\varepsilon(x,t) \leq \int_h^t L\left(\frac{\eta(s)}{\varepsilon}, \dot{\eta}(s)\right)\;ds  + u^\varepsilon(\eta(h), h) .
\end{align*}
Therefore,
\begin{align*}
	\int_\tau^h L\left(\frac{\eta(s)}{\varepsilon}, \dot{\eta}(s)\right)\,ds + f_\varepsilon(\eta(\tau), \tau) \leq u^\varepsilon(\eta(h), h).
\end{align*}
By the optimal control representation of $u^\varepsilon(\eta(h),h)$, we obtain
\begin{align*}
	u^\varepsilon(\eta(h), h) = \int_\tau^h L\left(\frac{\eta(s)}{\varepsilon}, \dot{\eta}(s)\right)\,ds + f_\varepsilon(\eta(\tau), \tau)
\end{align*}
for any $h\in (\tau, t)$. Since $u^\varepsilon$ is Lipschitz in $\overline{\Omega}_\varepsilon\times [0,\infty)$, Lemma \ref{lem:acdiff} implies that $r\mapsto u^\varepsilon(\gamma(r),r)$ is absolutely continuous on $[0,t]$. Moreover, for a.e. $t_0\in (0,t)$,
\begin{align*}
	\frac{d}{ds}u^\varepsilon(\eta(t_0),t_0) = \tilde{p}(t_0) + p(t_0)\cdot \dot{\eta}(t_0),
\end{align*}
where $(\tilde{p}(t_0), p(t_0)) \in \R^{n+1}$ such that $|\tilde{p}(t_0)| + |p(t_0)| \leq C_0$ by Proposition \ref{prop:ueplip} and Lemma \ref{lem:acdiff}. Let $t_0 \in (0,t)$ such that $\dot{\gamma}(t_0)$ exists. Then, for any $\tilde{t} \in (t_0,t)$,
\begin{align}\label{eq:ll}
	\frac{u^\varepsilon\left(\eta\left(\tilde{t}\right),\tilde{t}\right) -u^\varepsilon\left(\eta(t_0), t_0\right)}{\tilde{t} - t_0} = \frac{1}{\tilde{t}-t_0}\int_{t_0}^{\tilde{t}} L\left(\frac{\eta(s)}{\varepsilon}, \dot{\eta}(s)\right)\;ds.
\end{align}
Taking the limit as $\tilde{t}\to t_0$, we obtain 
\begin{align*}
	\tilde{p}(t_0) + p(t_0)\cdot \dot{\eta}(t_0) = L\left(\frac{\eta(t_0)}{\varepsilon}, \dot{\eta}(t_0)\right) 
	\qquad\Longrightarrow\qquad 
	C_0 + C_0|\dot{\eta}(t_0)| \geq \frac{|\dot{\eta}(t_0)|^2}{2} - K_0,
\end{align*}
which follows from \eqref{eqn:K_0L}. This implies that $|\dot{\eta}(t_0)|\leq M_0$ for some constant $M_0$ depending only on $C_0, K_0$. 
\end{proof}

\begin{rem} 
One can also apply Lemma~\ref{lem:A1} to obtain a constant that depends on $n$, $\partial\Omega$, $H$, and $\mathrm{Lip}(g)$ as follows:
\begin{align*}
    \frac{u^\varepsilon \left(\eta\left(\tilde{t}\right),\tilde{t}\right) - u^\varepsilon(\eta(t_0), t_0)}{\tilde{t} - t_0} 
    &= 
    \frac{u^\varepsilon \left(\eta \left(\tilde{t}\right),\tilde{t}\right) - u^\varepsilon \left(\eta(t_0), \tilde{t}\right)}{\tilde{t} - t_0} 
    + \frac{u^\varepsilon\left(\eta(t_0),\tilde{t}\right) - u^\varepsilon\left(\eta(t_0), t_0\right)}{\tilde{t} - t_0} \\
    &\leq C_0 \, C_b \, \frac{\left|\eta \left(\tilde{t}\right) - \eta(t_0)\right|}{\tilde{t}-t_0} + C_0,
\end{align*}
where $C_0$ denotes the Lipschitz bound for $u^\varepsilon$ from Proposition~\ref{prop:ueplip}, and $C_b$ is the constant from Lemma~\ref{lem:A1}. This argument follows the approach in \cite{han_quantitative_2024}.
\end{rem}

For convenience, we state the following lemma, which computes the derivative of $w \in  W^{1,\infty}(\Omega)\cap C(\overline{\Omega})$ along an absolutely continuous curve and generalizes \cite[Proposition 4.1]{mitake_asymptotic_2008} to the unbounded periodic setting considered in this paper. A proof of Lemma~\ref{lem:acdiff} is given in the Appendix, along with a simpler alternative for $C^2$ domains, included for clarity.

\begin{lem}\label{lem:acdiff} Consider an open and connected set $\Omega \subset \mathbb{R}^n$ with $C^1$ boundary and assume $\Omega$ is $\mathbb{Z}^n$-periodic. 
Let $w \in W^{1,\infty}(\Omega)\cap C(\overline{\Omega})$ and $\gamma\in \mathrm{AC}([a,b];\overline{\Omega})$. Then the map $s\mapsto w(\gamma(s))$ is absolutely continuous, and there exists a function $p\in L^\infty((a,b);\R^n)$ such that
\begin{align*}
	\frac{d}{ds}w(\gamma(s)) = \dot{\gamma}(s)\cdot p(s), \qquad \text{and} \qquad p(s)\in \partial_c w(\gamma(s)) \quad \text{for a.e. } s\in (a,b),
\end{align*}
where $\partial_c w$ denotes the Clarke differential of $w$, defined by
\begin{align*}
	\partial_c w(x) = \bigcup_{r>0} \overline{\mathrm{co}} \{Dw(y): y\in \Omega\cap B_r(x), \ w \text{ is differentiable at } y\}, \qquad x\in \overline{\Omega}.
\end{align*}
\end{lem}

Next, analogous to part (b) of Lemma \ref{lem:DPPHbar}, we present a corresponding result for $u^\ep$: when $u(x, t) < \ol{b}(x)$ with $t>0$, for sufficiently small $\ep>0$ and any $x_\ep \in \ol{\Om}_\ep$ close to $x$, the optimal running time $t-\tau_\ep$ in the representation of $u^\ep(x_\ep, t)$ is uniformly bounded away from $0$. This property will again play a crucial role when considering all $\tau \in [0, t)$ later.

\begin{lem}\label{lem:ueptimelbd}
Assume \ref{itm:A1}--\ref{itm:A5}. Let $t>0$ and $x\in \R^n$. Suppose $u(x, t) < \overline{b}(x)$. Let $\ep>0$ and $x_\ep \in \left(x+\ep Y\right)\cap \ol{\Omega}_\ep$. 
Suppose $\gamma_\varepsilon \in \mathcal{C}\Big(\frac{x_\varepsilon}{\varepsilon}, \frac{t}{\varepsilon}; \overline{\Omega}\Big)$ is a minimizing curve for $u^\varepsilon(x_\varepsilon, t)$ in the sense of \eqref{eqn:ocfue}, that is,
\begin{equation}\label{eqn:uepoptimalpath}
	u^\varepsilon (x_\ep, t) = \varepsilon \int^\frac{t}{\varepsilon
}_{\frac{\tau_\ep}{\varepsilon}} L\left(\gamma_\ep(s),\dot{\gamma}_\ep (s)\right) \,ds+\fep \left(y_\ep,\tau_\ep\right)  , 
	\qquad 
	\gamma_\ep\left(\frac{t}{\varepsilon}\right)=\frac{x_\varepsilon}{\ep}, \;\; 
	\gamma_\ep\left(\frac{\tau_\ep}{\varepsilon}\right)=\frac{y_\ep}{\ep}
\end{equation}
for some $\left( \tfrac{y_\ep}{\varepsilon}, \tau_\varepsilon \right) \in \partial Q_t$, and $\left|x_\ep-y_\ep\right|\leq M_0(t-\tau_\ep)$, where $M_0$ is given in Lemma \ref{lem:velocity bound}. 
Then there exists $\ep_0 > 0$ such that for all $0 < \ep < \ep_0$,
    \begin{equation}
    0 \leq \tau_\ep < \tau_{x, t} < t,
    \end{equation}
    for some constant $\tau_{x, t} \in (0, t)$, i.e. , $\{\tau_\ep: \ep \in \left(0, \ep_0\right)\}$ is uniformly bounded away from $t$.
\end{lem}

\begin{proof}
    We proceed by contradiction and suppose such $\ep_0$ does not exist. Then there exists a subsequence $\{\tau_{\ep_k}\}$ such that $\lim_{k\to \infty} \tau_{\ep_k}=t$.
    Since $\left|x_{\ep_k}- y_{\ep_k}\right| \leq M_0 (t-\tau_{\ep_k})$, we have $y_{\varepsilon_k}\to x$ as $k\to \infty$. By \cite{Horie1998Homogenization}, we know
\[
\lim_{\ep \searrow 0} \sup_{x\in \ol\Om_\ep}\left|u^\ep(x, t) -u(x, t)\right|=0,
\]
and since $u(x, t) <\overline{b}(x)$, there exists $\tilde{\ep} >0$ such that for any $\ep \in (0, \tilde{\ep})$, we have
\begin{equation}\label{eqn:uepleqbbar}
  u^\varepsilon (x_\ep, t) < \overline{b}(x).
\end{equation}
By \eqref{eqn:K_0L} and the definition of $\ol{b}$, we have
\[
\begin{aligned}
u^{\ep_k} (x_{\ep_k}, t) &= \varepsilon_k\int^\frac{t}{\ep_k
}_{\frac{\tau_{\ep_k}}{\ep_k}} L\left(\gamma_{\ep_k}(s),\dot{\gamma}_{\ep_k} (s)\right) \,ds+f_{\ep_k} \left(y_{\ep_k},\tau_{\ep_k}\right) \\
&\geq -(t-\tau_{\ep_k})K_0 + b\left(y_{\ep_k}, \frac{y_{\ep_k}}{\ep_k}\right) 
\geq -(t-\tau_{\ep_k})K_0+\ol{b}(y_{\ep_k}),  
\end{aligned}
\]
which gives $\liminf_{k\to \infty} u^{\ep_k} (x_{\ep_k}, t) \geq \ol{b}(x)$.
This contradicts \eqref{eqn:uepleqbbar}.
\end{proof}

\section{Extension of the cost function}\label{sec:costfunction}
In this section, we introduce a cost function, analogous to that in \cite{han_quantitative_2024}, which represents the cost of traveling between two points in \(\overline{\Omega}\) 
This cost function provides a representation of the value function $u^\varepsilon$ solving \eqref{eqn:PDE_epsilon}. We then derive a large-time average metric that characterizes the limiting solution \(u\) to \eqref{eqn:PDE_limit}. Since this metric is closely related to the state-constraint metric in \cite{han_quantitative_2024}, we recall the relevant definitions and results in Appendix \ref{appendix:costm} for completeness.

\subsection{The cost function with parameter $\tau$}
In the optimal control formula \eqref{eqn:ocfue} for $u^\ep$, the parameter $\tau \in [0, t)$ determines the starting time point of the trajectory. The running cost, given by the integral term in \eqref{eqn:ocfue}, is accumulated over $[\tau, t]$ and therefore depends on the elapsed time $t-\tau$. In other words, the full time horizon $t$ does not need to be used. This observation motivates fixing $\tau \in [0, t)$ and defining the cost function in terms of the running time $t-\tau$.

\begin{defn}\label{def:m}
Let $t>0$, $\tau \in[0,t)$, and $x, y\in \overline{\Omega}$. We define 
\begin{equation}\label{def:mtilde}
\begin{aligned}
    &\tilde{m}(\tau, t, y, x) 
    =\inf \left\lbrace \int_{\tau}^t L\left(\gamma(s),\dot{\gamma}(s)\right) \,ds: \gamma\in \mathcal{C}\left(x, t;\overline{\Omega}\right), \gamma\left(\tau \right)=y \right\rbrace.
\end{aligned}
\end{equation}
\end{defn}

\begin{rem}
Observe that $\tilde{m}(\tau, t, y, x)$ coincides with $m(t-\tau, y, x)$ as defined in \cite[Definition~1]{han_quantitative_2024}, so most properties of $\tilde{m}$ follow directly from those of $m$. The primary difference lies in the singularity: for $\tilde{m}$, if $y \neq x$ and $\tau \to t$, the cost blows up. Similarly, $m(t, y, x)$ in \cite{han_quantitative_2024} blows up as $t \to 0$ with $y \neq x$. In \cite{han_quantitative_2024}, the analysis about the averaging metric only involves $t>\ep$, so the singularity is irrelevant. Here, however, $\tau$ can be arbitrarily close to $t$ even for large $t$, and this singularity must be handled carefully in subsequent arguments. (See Lemma \ref{lem:ueptimelbd}, Lemma \ref{lem:tmbstarbd}, and the proof of Theorem \ref{thm:main1})
\end{rem}

Using Definition \ref{def:m} and Lemma \ref{lem:velocity bound}, we can rewrite the optimal control formula for $u^\varepsilon$ as the following:
\begin{equation}\label{eqn:uepocftm}
\begin{aligned}
    u^\varepsilon\left(x, t\right) 
    &=\inf\left\{\varepsilon \tm\left( \frac{\tau}{\varepsilon}, \frac{t}{\varepsilon}, \frac{y}{\varepsilon}, \frac{x}{\varepsilon}\right)+ \fep \left(y, \tau \right): 
    \left(y, \tau\right) \in \left(y, \tau\right) \in \partial Q^\varepsilon_t
    \right\}\\
    &= \inf\left\{\varepsilon \tm\left(\frac{\tau}{\varepsilon}, \frac{t}{\varepsilon}, \frac{y}{\varepsilon}, \frac{x}{\varepsilon}\right)+ \fep\left(y, \tau\right): \left|x-y\right|\leq M_0\left(t-\tau\right), 
    \left(y, \tau\right) \in \partial Q^\varepsilon_t
    \right\}.
\end{aligned}
\end{equation}

Note that in Definition \ref{def:m}, $\tm(\tau, t, y, x)$ is only defined for $t \in (0, \infty)$ and $\tau \in[0,t)$, and $x, y\in \overline{\Omega}$. We extend this metric to the larger domain $t \in (0, \infty), \tau \in [0, t)$, and $x, y \in \mathbb{R}^n$ 
as follows.
\begin{defn}
    Let $x, y\in \mathbb{R}^n$, $t > 0$, and $\tau \in [0,t)$. Define 
\begin{equation}\label{eq:defmstar}
     \tm^{\ast}\left(\tau, t, y, x\right)= \inf \left\lbrace \tm \left(\tau, t, \tilde{y}, \tilde{x}\right): \tilde{x}, \tilde{y} \in \partial \Omega, \tilde{x} - x \in Y, \tilde{y} - y \in Y \right\rbrace. 
\end{equation} 
\end{defn}

\begin{rem}
Note that $\tm^\ast(\tau, t, y, x)$ coincides with $m^\ast(t-\tau, y, x)$ defined in \cite[Definition 2]{han_quantitative_2024}.
\end{rem}

Since \(\tm^\ast(\tau, t, y, x) = m^\ast(t - \tau, y, x)\), where \(m^\ast\) is defined in \cite[Definition~2]{han_quantitative_2024}, we have the following result for the upper bound of $\tm^\ast$.

\begin{prop}[Restatement of {\cite[Proposition 2.2]{han_quantitative_2024}}] \label{prop:mstarbound}
    Assume \ref{itm:A1}--\ref{itm:A5}. Let $t \geq \delta > 0$ and $\tau \in [0, t - \delta]$ for some $\delta > 0$, and let $x, y \in \mathbb{R}^n$ satisfy $|x - y| \leq M (t - \tau)$ for some constant $M>0$.
Then there exists an absolutely continuous curve $\gamma : [0, t] \to \overline{\Omega}$ such that $\xi(\tau)=\tilde{y}$ and $\xi(t)=\tilde{x}$, for some $\tilde{x}, \tilde{y} \in \partial \Omega$ with $\tilde{x} - x \in Y, \tilde{y} - y \in Y$. Moreover, there exists a constant $C_b > 0$, depending only on $\partial \Omega$ and $n$, such that
    \begin{equation}
    \left|\dot{\gamma}(s)\right| \leq C_b \left(M+\frac{2\sqrt{n}}{\delta}\right) \quad \text{for all } s \in [0,t],
    \end{equation}
    and
    \begin{equation}
         \quad \displaystyle \tm^\ast \left(\tau, t, y, x\right) \leq \left( \frac{C_b^2}{2} \left(M+\frac{2\sqrt{n}}{\delta}\right)^2 +K_0 \right) (t-\tau).
    \end{equation}
\end{prop}

Further, for points $x, y \in \overline{\Omega}$, the cost given by $\tm^\ast$ only differs from that of $\tm$ by a constant.

\begin{prop}[{Restatement of \cite[Proposition 2.4]{han_quantitative_2024}}]
\label{prop:mstarm}
    Assume \ref{itm:A1}--\ref{itm:A5}. Let $t\geq 1$, $\tau \in [0, t-1]$, and $x, y\in \overline{\Omega}$ satisfy $|x-y| \leq M(t-\tau)$ for some constant $M>0$. Then, there is a constant $C=C\left(n, \partial \Omega, M, K_0\right)>0$ such that
    \[
    \left|\tm^\ast(\tau, t, x, y)-\tm(\tau, t, x, y)\right|<C. 
    \]
\end{prop}

The following proposition establishes a basic property of $\tilde{m}$: perturbing one endpoint of the curve within the unit cube changes the cost only by a constant. This fact will be used later.
\begin{prop}\label{prop:sptm}
    Assume \ref{itm:A1}--\ref{itm:A5}. Let $t\geq 1$, $\tau \in [0, t-1]$, and $x, y, z\in \overline{\Omega}$ satisfy $|x-y| \leq M(t-\tau)$ for some constant $M>0$, and $|y-z| \leq \sqrt{n}$. Then, there is a constant $C=C(n, \partial\Omega, M, K_0)$, such that
    \begin{equation}\label{eq:perturbtm}
    	\tm(\tau, t, z, x) \leq \tm(\tau, t, y, x) + C.
    \end{equation}
\end{prop}

\begin{proof} Let $\gamma: [0, t] \to \overline{\Omega}$ be an optimal path for $\tm(\tau, t, y, x)$, that is, $\gamma(\tau)=y$, $\gamma(t)=x$, and
    \[
    \tm(\tau, t, y, x)  = \int_\tau^t L\left(\gamma(s), \dot{\gamma}(s)\right)\,ds.
    \]
We provide an upper bound for $\widetilde{m}(\tau, t, y,x)$. Consider the curve $\xi: [0, t] \to \R^n$ defined by 
    \[
    \xi (s) = \begin{cases}
			y, & \text{if } \, s \in [0, \tau),\\
            \frac{s-\tau}{t-\tau}(x-y)+y, & \text{if } \,  s\in [\tau, t],
		 \end{cases}
    \]
    which satisfies $\left\Vert \dot{\xi}\right\Vert_{L^\infty([0, t])} \leq M$. 
    By applying \cite[Lemma~A.1]{han_quantitative_2024} 
to the line segment of \(\xi([\tau,t])\), we may modify $\xi$ to obtain a new path $\tilde{\xi}:[0,t] \to \overline{\Omega}$ with 
\begin{equation}\label{eqn:xivelubd}
	\tilde{\xi}(\tau)=y,\qquad  \tilde{\xi}(t)=x, \qquad \text{and}\qquad \left\Vert \dot{\tilde{\xi}}\right\Vert_{L^\infty([0, t])}  \leq C_b M
\end{equation}
for a constant $C_b>0$ depending only on $n$ and $\partial \Omega$. Therefore, by \eqref{eqn:K_0L},
    \begin{equation}\label{eqn:tmuptautxy}
    \tm(\tau, t, x, y) = \int_\tau^t L\left(\gamma(s), \dot{\gamma}(s)\right)\,ds \leq  \int_\tau^t L\left(\tilde{\xi}(s), \dot{\tilde{\xi}}(s)\right)\,ds \leq \left(\frac{C_b^2M^2}{2} + K_0\right)(t-\tau)   . 
    \end{equation}
    Furthermore, we claim that there exists $\displaystyle d\in \left\{0, \frac{1}{4},\frac{1}{2}, \frac{3}{4}, \cdots, \lfloor t-\tau\rfloor -\frac{1}{4} \right\}$ such that 
\begin{equation}\label{eqn:ld4ubd}
    \int_{\tau+d}^{\tau+d+\frac{1}{4}} L\left(\gamma(s), \dot{\gamma} (s)\right) \,ds\leq \frac{C_b^2 M^2}{2} +K_0.
\end{equation}
Indeed, if \eqref{eqn:ld4ubd} is not true, then  
\begin{equation}\label{eqn:Ltautlb}
    \int_\tau^{\tau +\lfloor t-\tau \rfloor} L\left(\gamma(s), \dot{\gamma} (s)\right)\,ds \geq 4 \lfloor t -\tau\rfloor \left(\frac{C_b^2 M^2}{2} +K_0\right).
\end{equation}
Then, from \eqref{eqn:tmuptautxy},
\begin{equation}
\begin{aligned}
    \left(\frac{C_b^2M^2}{2} + K_0\right)(t-\tau)  &\geq \int_\tau^{\tau +\lfloor t-\tau \rfloor} L\left(\gamma(s), \dot{\gamma} (s)\right)\,ds + \int_{\tau +\lfloor t-\tau \rfloor}^t L\left(\gamma(s), \dot{\gamma} (s)\right)\,ds\\
    &\geq 4 \lfloor t -\tau \rfloor \left(\frac{C_b^2M^2}{2} + K_0\right) -K_0,\\
\end{aligned}
\end{equation}
where the last inequality comes from \eqref{eqn:K_0L} and \eqref{eqn:Ltautlb}.
On the other hand, since $t -\tau \geq 1$, 
\[
4 \lfloor t -\tau \rfloor \left(\frac{C_b^2M^2}{2} + K_0\right) -K_0 > \left(\frac{C_b^2M^2}{2} + K_0\right)(t-\tau),
\]
which is a contradiction.

By a similar argument as in \eqref{eqn:xivelubd}, there exists a path $\eta:[0,1] \to \overline{\Omega}$ connecting $z$ to $y$ such that 
 \begin{equation}\label{eqn:etavelubd}
	\eta(0) = z, \qquad \eta(1) = y,\qquad \text{and}\qquad \left\|\dot{\eta}\right\|_{L^\infty\left([0,1]\right)} \leq \sqrt{n} C_b.
 \end{equation}
We concatenate \(\eta\) and \(\gamma\) to construct a path connecting \(z\) to \(x\), adjusting their time parametrizations so that each takes $\frac{1}{8}$ units of time, with the total duration of the resulting path equal to \(t - \tau\).
Precisely, define \(\zeta : [0,t] \to \overline{\Omega}\) by
\begin{equation}
  \zeta (s) :=  
  \begin{cases}
  \begin{aligned}
  &z, &&\text{if}\; s\in [0,\tau],\\
  &\eta\left(8(s-\tau\right)),&&\text{if}\; s\in \left[\tau, \tau + \tfrac{1}{8}\right],\\
  &\gamma \left(s-\tfrac{1}{8}\right), &&\text{if}\; s\in \left[\tau+\tfrac{1}{8}, \tau+d+\tfrac{1}{8}\right], \\
  &\gamma \left(2\left(s-\left(\tau+d+\tfrac{1}{8}\right)\right)+\tau +d\right), &&\text{if}\; s\in \left[\tau+d+\tfrac{1}{8}, \tau + d + \tfrac{1}{4}\right], \\
  &\gamma \left(s\right),&& \text{if}\; s\in \left[\tau + d+\tfrac{1}{4}, t\right],
  \end{aligned}
  \end{cases}
\end{equation}
which is an admissible path for $\tm(\tau, t, z, x)$. We compute
\begin{align}\label{eq:estzetaA}
	A : = \int_\tau^{\tau+\frac{1}{8}} L(\zeta(s), \dot{\zeta}(s))\;ds 
	= 
	\frac{1}{8} \int_0^1 L\left(\eta(s), 8\dot{\eta}(s)\right)\;ds  
	\leq 
	4nC_b^2 + \frac{K_0}{8}
\end{align}
using \eqref{eqn:K_0L} and \eqref{eqn:etavelubd}. Similarly, we have 
\begin{align}\label{eq:estzetaB}
	B := \int_{\tau+d+\frac{1}{8}}^{\tau+d+\frac{1}{4}} L(\zeta(s), \dot{\zeta}(s))\;ds 
	&= 
	\frac{1}{2}\int_{\tau+d}^{\tau+d+\frac{1}{4}} L\left(\gamma(s), 2\dot{\gamma}(s)\right)\;ds   \nonumber \\
	&\mkern-40mu \leq 2\int_{\tau+d}^{\tau+d+\frac{1}{4}} \left(\frac{|\dot{\gamma}(s)|^2}{2} - K_0\right)\;ds + \frac{5K_0}{8}\nonumber  \\
	&\mkern-40mu \leq 2 \int_{\tau+d}^{\tau+d+\frac{1}{4}}L(\gamma(s), \dot{\gamma}(s))\;ds + \frac{5K_0}{8} \leq C_b^2M^2 + \frac{21K_0}{8}
\end{align}
by \eqref{eqn:K_0L} and \eqref{eqn:ld4ubd}.
Using \eqref{eq:estzetaA}, \eqref{eq:estzetaB}, and \eqref{eqn:K_0L}, we compare 
\(\tilde{m}(\tau, t, z, x)\) and \(\tilde{m}(\tau, t, y, x)\): 
\begin{align*}
	\widetilde{m}(\tau, t, z, x) - \widetilde{m}(\tau, t, y, x) 
	&\leq 
	\int_\tau ^ t L\left(\zeta(s), \dot{\zeta}(s)\right)\;ds - \int_{\tau}^t L(\gamma(s), \dot{\gamma}(s))\;ds \\
	&\mkern-40mu= A + B - \int_{\tau+d}^{\tau+d+\frac{1}{4}} L(\gamma(s), \dot{\gamma}(s))\;ds \\
	&\mkern-40mu\leq \left(4nC_b^2 + \frac{K_0}{8}\right) + \left(C_b^2 M^2 + \frac{21K_0}{8}\right) + \frac{K_0}{4} = (M^2+4n)C_b^2 + 3K_0. 
\end{align*}
Therefore the conclusion \eqref{eq:perturbtm} follows, with $C=(M^2+4n)C_b^2 + 3K_0$.
\end{proof}

Moreover, the averaging of the cost function admits a well-defined limit together with the desired convergence rate.

\begin{prop}[{Restatement of \cite[Proposition~4.1 and Theorem~4.2]{han_quantitative_2024}}]\label{prop:mbarrate}
Assume \ref{itm:A1}--\ref{itm:A5}. For any \(t > 0\), \(\tau \in [0, t)\), and \(x, y \in \mathbb{R}^n\) with \(|x - y| \leq M (t - \tau)\), for some constant $M>0$, the limit
\[
\overline{m}^\ast(\tau, t, y, x):=\lim_{k\to \infty} \frac{1}{k} \tm^\ast (k\tau, kt, ky, kx)
\]
exists. Moreover, there exists a constant 
$C = C(n, \partial \Omega, M, K_0) > 0$ such that for all \(\varepsilon \in (0, t-\tau]\),
\[
 \left|\overline{m}^\ast \left(\tau, t, y, x\right) - \varepsilon \tm^\ast \left(\frac{\tau}{\varepsilon}, \frac{t}{\varepsilon}, \frac{y}{ \varepsilon}, \frac{x}{\varepsilon}\right)\right| \leq C \varepsilon.
\]
\end{prop}

We emphasize that the above rate is valid only for $\varepsilon \in [0, t-\tau]$. Hence, when considering all $\tau \in[0, t)$, the convergence rate is not uniform. The next lemma, though elementary, plays a crucial role in handling the regime $t-\tau< \ep$. In particular, while $\overline{m}^\ast$ may blow up as $\tau \to t$, an appropriate lower bound will be essential for establishing the main theorem.

\begin{lem}\label{lem:tmbstarbd}
Assume \ref{itm:A1}--\ref{itm:A5}. Let $t\geq\ep >0$ and $\tau \in (t-\ep, t)$. Let $x, y \in \R^n$ with $|x- y| \leq M|t-\tau|$ for some constant $M>0$. Then,
    \[
    \overline{m}^\ast(\tau, t, y, x) \geq- C\ep,
    \]
    for some constant $C=C\left(n, \partial \Omega, M, K_0\right)>0$.
\end{lem}

\begin{proof}
Choose $\ep' \in(0, t-\tau]$ small. Note that $0 < \ep'< t-\tau <\ep$.
By Proposition \ref{prop:mbarrate}, there is a constant 
\(C = C(n, \partial \Omega, M, K_0) > 0\) such that
\begin{align*}
\overline{m}^\ast(\tau, t, y, x) 
    &\ge \varepsilon' \, \tilde{m}^\ast\left(\frac{\tau}{\varepsilon'}, \frac{t}{\varepsilon'}, \frac{y}{\varepsilon'}, \frac{x}{\varepsilon'}\right) - C \varepsilon'  \\
    &\ge - (t-\tau) K_0 - C \varepsilon' \ge - (K_0 + C) \varepsilon,
\end{align*}
where the second inequality follows from \eqref{eqn:K_0L} to obtain the lower bound for \(\tilde{m}^\ast\), and \(t-\tau < \varepsilon\).
%
\end{proof}


\subsection{The value function associated with the average of the extended metric}
So far, no connection has been established between the limit $\overline{m}^\ast$ and the solution \(u\) of \eqref{eqn:PDE_limit}, which admits a Hopf–Lax type representation as in \eqref{eq:HLubar}. 
We now show that replacing the running cost in the Hopf-Lax representation \eqref{eq:HLubar} of $u$ with $\overline{m}^\ast$ still yields the limiting solution $u$. The case $t-\tau<\ep$ turns out to be delicate, and Lemma \ref{lem:ueptimelbd} plays a key role in handling it.

\begin{defn} Let $t > 0$ and $x \in \mathbb{R}^n$. Define
\begin{equation}\label{eqn:defubar}
\overline{u}(x, t):=
\begin{cases}
\displaystyle \inf_{\tau\in [0,t), y\in \R^n} 
	\Big\lbrace 
		\overline{m}^\ast \left(\tau, t, y, x \right) +f(y, \tau): |x -y | \leq M{_0} (t-\tau) 
	\Big\rbrace , & \text{if } t > 0, \\
	g(x), 	        & \text{if } t = 0,
\end{cases}
\end{equation}
where $M_0=M_0\left(n,\partial \Omega, H, \|Dg\|_{L^\infty(\mathbb{R}^n)}\right)>0$ is the constant from Lemma \ref{lem:velocity bound}.
\end{defn}

Before proving $\ol{u}=u$, we first show a property of $\ol{u}$ that will be useful in the proof.
\begin{lem}\label{lem:ubarlequbar}
    Assume \ref{itm:A1}--\ref{itm:A5}. 
    For all $(x,t)\in \R^n\times [0,\infty)$, $\ol{u}(x,t) \leq \overline{b}(x)$.
\end{lem}
\begin{proof} For $x\in \R^n$ and $t=0$, $\ol{u}(x, t)= g(x)\leq \ol{b}(x)$ by \ref{itm:A5}.
Let $x\in \R^n$ and $t>0$. Choose $y=x$ and $\tau=t-\ep$ for $\ep>0$ in \eqref{eqn:defubar} to obtain
    \begin{equation}\label{eqn:ubarubbtmstar}
    \begin{aligned}
    \ol{u}(x, t) \leq  
    \overline{m}^\ast\left(t-\ep,t, x,x\right)+f(x, t-\ep)
    \leq   \ep \tm^\ast \left(\frac{t-\ep}{\ep},\frac{t}{\ep}, \frac{x}{\ep}, \frac{x}{\ep}\right)	
    +C\ep
    +f(x, t-\ep),\\
    \end{aligned}        
    \end{equation}
    where the second inequality follows from Proposition \ref{prop:mbarrate}.
    Choose $\tilde{x} \in \partial\ol\Om_\ep$ such that $\frac{\tilde{x}}{\ep}-\frac{x}{\ep} \in Y$. Then by the definiton of $\tm^\ast$,
    \begin{equation*}
    \tm^\ast \left(\frac{t-\ep}{\ep},\frac{t}{\ep}, \frac{x}{\ep}, \frac{x}{\ep}\right) \leq \tm \left(\frac{t-\ep}{\ep},\frac{t}{\ep}, \frac{\tilde{x}}{\ep}, \frac{\tilde{x}}{\ep}\right).
    \end{equation*}
    Combining with \eqref{eqn:ubarubbtmstar}, we have
    \begin{equation}\label{eqn:ubarubbtm}
    \ol{u}(x, t) \leq   \ep \tm \left(\frac{t-\ep}{\ep},\frac{t}{\ep}, \frac{\tilde{x}}{\ep}, \frac{\tilde{x}}{\ep}\right)+f(x, t-\ep) + C \ep.        
    \end{equation}
    Choose $\gamma \in \mathcal{C}\left(\frac{t}{\ep}, \frac{\tilde{x}}{\ep}; \ol\Om\right)$ defined by $\gamma \equiv \frac{\tilde{x}}{\ep}$, and by the definition of $\tm$, we have
    \begin{equation}\label{eqn:tmubd}
    \begin{aligned}
    \tm \left(\frac{t-\ep}{\ep},\frac{t}{\ep}, \frac{\tilde{x}}{\ep}, \frac{\tilde{x}}{\ep}\right) 		&\leq  
   \int_{\frac{t-\ep}{\ep}}^{\frac{t}{\ep}}\ol{L}\left(\gamma(s),\dot{\gamma}(s)\right) \, ds
   		\leq \ol{L}\left(\frac{\tilde{x}}{\ep}, 0 \right)
	\leq  C_1,    
    \end{aligned}        
    \end{equation}
    where $C_1=C_1(H)>0$ is from \eqref{eq:H-lower-bd}.
    By \eqref{eqn:ubarubbtm} and \eqref{eqn:tmubd}, $\ol{u} (x,t) \leq C\ep+f(x, t-\ep)$,
    and sending $\ep \to 0$, we have $\ol{u}(x,t) \leq \ol{b}(x)$. 
\end{proof}

\begin{lem}\label{lem:ubarequ}
    Assume \ref{itm:A1}--\ref{itm:A5}. Then $\overline{u}(x, t) = u(x, t)$ for any $t \geq 0 $ and $x \in \mathbb{R}^n$.
\end{lem}

\begin{proof} If $t=0$, $u(x,0) = g(x) = \overline{u}(x,0)$ for $x\in \R^n$. \medskip 

\noindent 
{\bf Step 1.} Let $t > 0 $ and $x \in \mathbb{R}^n$. We show that
\begin{align*}
	\ol{u} (x, t) \leq u(x, t). 
\end{align*}

By Lemma \ref{lem:DPPHbar}, we know $u(x, t) \leq \ol{b}(x)$. If $u(x,t) = \overline{b}(x)$, then $\overline{u}(x,t) \leq \overline{b}(x) = u(x,t)$ by Lemma \ref{lem:ubarlequbar}. Suppose that $u(x,t) < \overline{b}(x)$. Let $\tilde{x}_\varepsilon \in \partial \overline{\Omega}_\varepsilon$ be such that $\frac{\tilde{x}_\varepsilon - x}{\varepsilon} \in Y$. Choose a minimizer of $u^\varepsilon(\tilde{x}_\varepsilon, t)$ in the sense of \eqref{eqn:uepocftm}, or equivalently \eqref{eqn:ocfue}; that is, there exists an optimal pair $\tau_\varepsilon \in [0, t)$ and $\gamma: \left[0, \frac{t}{\varepsilon}\right] \to \overline{\Omega}_\varepsilon$ such that
\begin{align*}
	\gamma_\varepsilon\left(\frac{t_\varepsilon}{\varepsilon}\right) = \frac{\tilde{x}_\varepsilon}{\varepsilon},
	 \qquad \gamma_\varepsilon\left(\frac{\tau_\varepsilon}{\varepsilon}\right) = \frac{y_\varepsilon}{\varepsilon}, 
	 \qquad \left(\frac{y_\varepsilon}{\varepsilon}, \tau_\varepsilon\right) \in \partial Q_t, \qquad |\tilde{x}_\varepsilon- y_\varepsilon|\leq M_0(t-\tau_\varepsilon),	 
\end{align*}
and
\begin{align}\label{eq:ubarAA}
	u^\varepsilon(\tilde{x}_\varepsilon, t) 
	&= \varepsilon \int_{\frac{\tau_\varepsilon}{\varepsilon}}^{\frac{t}{\varepsilon}} L(\gamma_\varepsilon(s), \dot{\gamma}_\varepsilon(s))\;ds + f_\varepsilon(y_\varepsilon, \tau_\varepsilon) 
	= 
	\varepsilon \widetilde{m}\left(\frac{\tau_\varepsilon}{\varepsilon}, \frac{t}{\varepsilon}, \frac{y_\varepsilon}{\varepsilon}, \frac{\tilde{x}_\varepsilon}{\varepsilon}\right)  + f_\varepsilon(y_\varepsilon, \tau_\varepsilon). 
\end{align}
By Lemma \ref{lem:ueptimelbd}, there exists $\varepsilon_0 > 0$ and $\tau_{x,t} \in (0,t)$ such that $0\leq \tau_\varepsilon < \tau_{x,t} < t $ for $\varepsilon \in (0,\varepsilon_0)$, i.e.,
\begin{equation*}
	0\leq \tau_\varepsilon < t-\varepsilon \qquad \text{for all}\; \varepsilon \in (0,\varepsilon_0). 
\end{equation*}
%
To connect $\widetilde{m}$ and $\widetilde{m}^*$, as $\tilde{x}_\varepsilon\in \partial\Omega_\varepsilon$ already, we take $\tilde{y}_\varepsilon\in \partial \Omega_\varepsilon$ such that $\frac{\tilde{y}_\varepsilon - y_\varepsilon}{\varepsilon} \in Y$. Then by the definition of $\widetilde{m}^*$ in \eqref{eq:defmstar} and Proposition \ref{prop:sptm}, we have
\begin{align}\label{eq:ubarA1}
	\widetilde{m}^*\left(\frac{\tau_\varepsilon}{\varepsilon}, \frac{t}{\varepsilon}, \frac{y_\varepsilon}{\varepsilon}, \frac{x}{\varepsilon}\right)
	\leq 
	\widetilde{m}\left(\frac{\tau_\varepsilon}{\varepsilon}, \frac{t}{\varepsilon}, \frac{\tilde{y}_\varepsilon}{\varepsilon}, \frac{\tilde{x}_\varepsilon}{\varepsilon}\right) 
	\leq 
	\widetilde{m}\left(\frac{\tau_\varepsilon}{\varepsilon}, \frac{t}{\varepsilon}, \frac{y_\varepsilon}{\varepsilon}, \frac{\tilde{x}_\varepsilon}{\varepsilon}\right) + C
\end{align}
for some $C=C(n, \partial\Omega, M_0, K_0)>0$. By Proposition \ref{prop:mbarrate}, we have
\begin{align}\label{eq:ubarA2}
	\overline{m}^*(\tau_\varepsilon, t, y_\varepsilon, x) - C\varepsilon \leq \varepsilon \widetilde{m}^*\left(\frac{\tau_\varepsilon}{\varepsilon}, \frac{t}{\varepsilon}, \frac{y_\varepsilon}{\varepsilon}, \frac{x}{\varepsilon}\right). 
\end{align}
From \eqref{eq:ubarA1}, \eqref{eq:ubarA2}, and \eqref{eq:ubarAA}, we obtain
\begin{align}\label{eq:ubarA3}
	u^\varepsilon(\tilde{x}_\varepsilon, t) 
	&= \varepsilon \widetilde{m}\left(\frac{\tau_\varepsilon}{\varepsilon}, \frac{t}{\varepsilon}, \frac{y_\varepsilon}{\varepsilon}, \frac{\tilde{x}_\varepsilon}{\varepsilon}\right) + f_\varepsilon(y_\varepsilon, \tau_\varepsilon) \nonumber \\
	&\geq \overline{m}^*(\tau_\varepsilon, t, y_\varepsilon, x) + f(y_\varepsilon,\tau_\varepsilon) - C\varepsilon
	\geq \overline{u}(x,t) - C\varepsilon. 
\end{align}
Here we use the fact that $f_\varepsilon(y_\varepsilon, \tau_\varepsilon) \geq f(y_\varepsilon, \tau_\varepsilon)$ for $(\tfrac{y_\varepsilon}{\varepsilon}, \tau_\varepsilon) \in \partial Q_t$ from \eqref{eq:ffeps}. 
\noindent 
Let $\delta>0$. By qualitative homogenization result in \cite{Horie1998Homogenization},
there exists $\ep_\del \in (0,\varepsilon_0)$ such that 
\begin{equation}\label{eqn:ueps-u}
    \Vert u^\varepsilon(\cdot, t) - u(\cdot,t)\Vert_{L^\infty(\R^n)} < \delta \qquad\text{for}\;\varepsilon\in(0,\varepsilon_\delta). 
\end{equation}
Use \eqref{eqn:ueps-u} and \eqref{eq:ubarA3}, we deduce that $u(\tilde{x}_\varepsilon, t) + \delta \geq \overline{u}(x,t) - C\varepsilon$. Recall that $|\tilde{x}_\varepsilon - x| \le \sqrt{n}\,\varepsilon$. Sending $\varepsilon \to 0$ first, we obtain $u(x, t) + \delta \ge \overline{u}(x, t)$; then, sending $\delta \to 0$, we obtain the conclusion $u \ge \overline{u}$.
\medskip 

\noindent 
{\bf Step 2.}
Next, we show that for $x\in \R^n$ and $t>0$,
\begin{align*}
	\ol{u} (x, t) \geq u(x, t). 
\end{align*}

Let $\delta>0$. By the definition of $\overline{u}$, there exists $\tau_\delta \in [0,t)$ and $y_\delta \in \R^n$ with $|x-y_\delta| \leq M_0(t-\tau_\delta)$ such that
\begin{equation}\label{eqn:oluappbyolmstar}
	\overline{u}(x,t) + \delta \geq \overline{m}^*(\tau_\delta, t, y_\delta, x) + f(y_\delta, \tau_\delta). 
\end{equation}
Let $\varepsilon \in (0,t-\tau_\delta]$. By Proposition \ref{prop:mbarrate}, we know
\begin{equation}\label{eqn:olmstarappbtmstar}
	\overline{m}^*(\tau_\delta, t, y_\delta, x) 
	\geq 
	\varepsilon \widetilde{m}^*\left(\frac{\tau_\delta}{\varepsilon}, \frac{t}{\varepsilon}, \frac{y_\delta}{\varepsilon}, \frac{x}{\varepsilon}\right) - C\varepsilon
\end{equation}
for $C = C(n, \partial\Omega, H, \mathrm{Lip}(g)) >0 $. There exist $\tilde{y}_\delta, \tilde{x}_\delta \in \partial\Omega_\varepsilon$ with $\frac{\tilde{y}_\delta - y_\delta}{\varepsilon} \in Y$ and $\frac{\tilde{x}_\delta-x}{\varepsilon}\in Y$ such that
\begin{equation} \label{eqn:tmstareqtm}
    \tm^\ast \left(\frac{\tau_\del}{\ep},\frac{t}{\ep}, \frac{y_\del}{\ep}, \frac{x}{\ep} \right)=\tm \left(\frac{\tau_\del}{\ep},\frac{t}{\ep}, \frac{\tilde{y}_\del}{\ep},  \frac{\tilde{x}_\delta}{\ep} \right).
\end{equation}
Since $\varepsilon < t-\tau_\delta$, we have $|\tilde{y}_\delta - y|\leq \sqrt{n}\varepsilon$ and $|\tilde{x}_\delta - x|\leq \sqrt{n}\varepsilon \leq 2\sqrt{n}(t-\tau_\delta)$. Therefore{\red ,}
\begin{equation}\label{eqn:txtydelbdbttau}
    \begin{aligned}
    \left|\tilde{x}_\delta-\tilde{y}_\del\right| &\leq \left|\tilde{x}_\delta-x\right|+\left|x-y_\delta\right|+\left|y_\delta-\tilde{y}_\del\right| 
    \leq \left(2\sqrt{n}+M_0\right)(t-\tau_\del).
    \end{aligned}
\end{equation}
By \eqref{eqn:oluappbyolmstar}, \eqref{eqn:olmstarappbtmstar}, and \eqref{eqn:tmstareqtm}, we have
\begin{equation}\label{eqn:ubartmtydel}
        \ol{u}\left(x, t\right) + \delta \geq  \ep \tm \left(\frac{\tau_\del}{\ep},\frac{t}{\ep}, \frac{\tilde{y}_\del}{\ep},  \frac{\tilde{x}_\delta}{\ep} \right) +f\left(y_\del,\tau_\delta\right)-C\ep.
\end{equation}
\begin{itemize}
\item If $\tau_\delta = 0$, then by \eqref{eqn:deff} we have $f(y_\delta, \tau_\delta) = f(y_\delta, 0) = g(y_\delta)$, and thus \eqref{eqn:ubartmtydel} becomes
\begin{align*}
	\ol{u}\left(x, t\right) + \delta  \geq \varepsilon \tilde{m}\left(0, \frac{t}{\varepsilon}, \frac{\tilde{y}_\delta}{\varepsilon}, \frac{\tilde{x}_\delta}{\varepsilon}\right) - C\varepsilon \geq u^\varepsilon(\tilde{x}_\delta,t) - C\varepsilon {\red ,}
\end{align*}
where the last inequality follows from \eqref{eqn:uepocftm}. 
\item If $\tau_\delta > 0$, then since $b(x,\cdot)$ is $\Z^n$-periodic for every $x$, we can choose $z_\delta \in \partial\Omega_\varepsilon$ with $\frac{z_\delta - y_\delta}{\varepsilon} \in Y$ such that 
\begin{equation}\label{eqn:zdeldef}
        \ol{b}(y_\del) = \min_{z \in \partial\Omega} b (y_\del, z)= b\left(y_\del, \frac{z_\del}{\varepsilon}\right).
\end{equation}
Since $\tau_\delta > 0$, we have $f_\varepsilon(z_\delta, \tau_\delta) = b\left(z_\delta, \tfrac{z_\delta}{\varepsilon}\right)$ and $f(y_\delta, \tau_\delta) =\overline{b}(y_\delta)$; thus,
\begin{align}\label{eq:zdeltaAA}
	f(y_\delta, \tau_\delta) 
	= \overline{b}(y_\delta) 
	&= b\left(y_\delta, \frac{z_\delta}{\varepsilon}\right) \nonumber \\
	&= b\left(y_\delta, \frac{z_\delta}{\varepsilon}\right) - b\left(z_\delta, \frac{z_\delta}{\varepsilon}\right) + b\left(z_\delta, \frac{z_\delta}{\varepsilon}\right) 
	\geq \omega_b(\sqrt{n}\varepsilon) + f_\varepsilon(z_\delta),
\end{align}
where $\omega_b(\cdot)$ denotes the modulus of continuity of $b$ from \ref{itm:A4}.
    Note that $|z_\delta-\tilde{y}_\delta| \leq |z_\delta-y_\delta| +|y_\delta-\tilde{y}_\delta| \leq 2\ep\sqrt{n}$.
    By \eqref{eqn:txtydelbdbttau} and Proposition \ref{prop:sptm},
     \begin{equation}\label{eqn:tmydeltmzdel}
         \tm \left(\frac{\tau_\del}{\ep},\frac{t}{\ep}, \frac{\tilde{y}_\del}{\ep},  \frac{\tilde{x}_\delta}{\ep} \right) \geq \tm \left(\frac{\tau_\del}{\ep},\frac{t}{\ep}, \frac{z_\del}{\ep},  \frac{\tilde{x}_\delta}{\ep} \right) - C 
     \end{equation}
     for some constant $C=C\left(n, \partial \Omega, M_0, K_0\right)>0$. Combining \eqref{eqn:ubartmtydel}, \eqref{eqn:tmydeltmzdel}, and \eqref{eq:zdeltaAA}{\red ,} we have
\begin{align*}
	\overline{u}(x,t) + \delta 
	&\geq 
		\varepsilon \tm \left(\frac{\tau_\del}{\ep},\frac{t}{\ep}, \frac{z_\del}{\ep},  \frac{\tilde{x}_\delta}{\ep} \right) + f_\varepsilon(y_\delta, \tau_\delta)  - C \varepsilon +f(y_\delta, \tau_\delta) -  f_\varepsilon(y_\delta, \tau_\delta)\\
	&\geq u^\varepsilon(\tilde{x}_\delta,t) - C\varepsilon - \omega_b(\sqrt{n\varepsilon}). 
\end{align*}
\end{itemize}
In each case of $\tau_\delta =0$ and $\tau_\delta > 0$, we have
\begin{align*}
	\overline{u}(x,t) +\delta \geq u^\varepsilon(\tilde{x}_\delta,t) - C_\varepsilon - \omega_b \left(\sqrt{n}\varepsilon\right), 
\end{align*}
for some constant $C=C(n, \partial \Omega, M_0, K_0)>0$ and for the modulus of continuity $\omega_b$ associated with $b$.
Define $\tilde{\ep}:=\min\left\{t-\tau_\delta, \varepsilon_\delta\right\}$ where $\varepsilon_\delta$ is as given in \eqref{eqn:ueps-u}. Then, for any $\ep \in (0, \tilde{\ep})$,
\begin{align*}
	\overline{u}(x,t) +\delta \geq u^\varepsilon(\tilde{x}_\delta,t) - C_\varepsilon - \omega_b(\sqrt{n}\varepsilon) \geq u(\tilde{x}_\delta, t) -\delta - C_\varepsilon - \omega_b(\sqrt{n}\varepsilon). 
\end{align*}
Note that $|\tilde{x}_\delta - x| \leq \sqrt{n}\varepsilon$. Sending $\ep \to 0$ first, we obtain $\overline{u}(x,t) +\delta \geq u(x,t) - \delta$; then, sending $\delta \to 0$, we obtain the conclusion $\overline{u}\geq u$.
\end{proof}

\section{Proof of Theorem \ref{thm:main1}} \label{sec:ProofThm1}

%

Now we are ready to prove Theorem \ref{thm:main1}.

\begin{proof}[Proof of Theorem \ref{thm:main1}]
Let $\varepsilon>0, x \in \overline{\Omega}_\varepsilon$, and $t>0$. The case $0<t<\varepsilon$ is easier to handled due to Lipschitz bound in time of $u^\varepsilon$ and $u$. Indeed, by Propositions \ref{prop:ueplip} and \ref{prop:ueplip}, there exists a constant $C = C(H, \Vert g\Vert_{\mathrm{Lip}(\R^n)})$ such that
\begin{equation*}
    \left|u^\varepsilon(x,t)-g(x)\right|\leq Ct 
    \qquad\text{and} \qquad 
    \left|u(x,t)-g(x)\right|\leq Ct
\end{equation*}
for all $x\in \R^n$. Hence, $\left|u^\varepsilon(x, t)-u(x, t)\right| \leq Ct \leq C \varepsilon$ for $x\in \R^n$. 
\medskip 
\noindent 
Let us consider the case $t \geq \varepsilon$. 

\noindent 
{\bf Step 1. The lower bound of $u^\varepsilon - u$ when $t\geq \varepsilon$.} We aim to show that 
\begin{equation*}
    u(x, t)-u^\ep(x, t) \leq C\varepsilon
\end{equation*}
for some constant $C=C\left(n, \partial \Omega, H, \|Dg\|_{L^\infty(\mathbb{R}^n)}, \|Db\|_{L^\infty \left(\mathbb{R}^n \times \partial \Om \right)}\right)>0$. For $u^\varepsilon \left(x, t\right)$, by Lemma \ref{lem:velocity bound}, there exist $\tau \in [0, t)$ and $y\in \ol\Om_\ep$, and an optimal path $\gamma \in \mathcal{C}\left(\frac{x}{\varepsilon}, \frac{t}{\varepsilon}; \overline{\Om}\right)$ such that
            \[
            \gamma\left(\frac{\tau}{\ep}\right)=\frac{y}{\varepsilon}, \qquad \left\|\dot{\gamma} \right\|_{L^\infty\left(\left[\frac{\tau}{\varepsilon}, \frac{t}{\varepsilon}\right]\right)}\leq M_0, \qquad |x-y| \leq M_0|t-\tau|,
            \]
            and
            \[
            u^\varepsilon (x, t) = \varepsilon \tm\left( \frac{\tau}{\varepsilon}, \frac{t}{\varepsilon}, \frac{y}{\varepsilon}, \frac{x}{\varepsilon}\right)+\fep\left(y, \tau \right)=\varepsilon \int^\frac{t}{\varepsilon}_{\frac{\tau}{\varepsilon}} L\left(\gamma(s),\dot{\gamma}(s)\right) \,ds+\fep \left(y,\tau\right).
            \]
There are two cases: $t-\tau < \varepsilon$ and $t-\tau \geq \varepsilon$. 
\begin{itemize}
	\item If $t-\tau< \varepsilon$, then by \eqref{eqn:K_0L}, 
            \[
            \ep \int^\frac{t}{\varepsilon}_{\frac{\tau}{\varepsilon}} L\left(\gamma(s),\dot{\gamma}(s)\right) \,ds \geq -(t-\tau )K_0\geq -\ep K_0 
	            \quad\Longrightarrow\quad 
            u^\varepsilon(x,t) \geq f_\varepsilon(y,\tau) -\varepsilon K_0,
            \]
            and hence 
	\begin{align}
		u(x, t) - u^\ep (x, t) 
        &\leq 
        (t-\tau) \Lbar\left(\frac{x-y}{t-\tau }\right)+f(y, \tau) - f_\varepsilon(y,\tau) + \varepsilon K_0 \nonumber \\ 
        &\leq  (t-\tau)\left(\frac{1}{2}\frac{\left|x-y\right|^2}{(t-\tau)^2}+K_0\right)+\ep K_0   
            \leq  \ep\left(M_0^2+2K_0\right),
	\end{align}
where the second inequality uses \eqref{eq:ffeps}, and the last follows from \eqref{eq:Hbar-Lbar}.
                \item If $t-\tau \geq \varepsilon$, by Lemma \ref{lem:ubarequ}, 
                \begin{align*}
                u(x, t) =\overline{u}\left(x,t\right)  
                &\leq \overline{m}^\ast \left(\tau, t, y, x \right) +f(y, \tau)\\
                &\leq  \varepsilon \tm^\ast \left(\frac{\tau}{\varepsilon},\frac{t}{\varepsilon}, \frac{y}{\varepsilon}, \frac{x}{\varepsilon} \right) +f(y, \tau)  + C \varepsilon\\
                &\leq \varepsilon \tm\left(\frac{\tau}{\varepsilon},\frac{t}{\varepsilon}, \frac{y}{\varepsilon}, \frac{x}{\varepsilon} \right) +\fep(y, \tau)  + \hat{C} \varepsilon
                =  u^\varepsilon \left(x, t\right)+ \hat{C} \varepsilon,
                \end{align*}	
                for some constant $\hat{C}=\hat{C}\left(n, \partial \Omega, M_0, K_0\right)>0$. Here the second inequality follows from Proposition \ref{prop:mbarrate}, and the third from Proposition \ref{prop:mstarm} together with $f(y,\tau) \leq f_\varepsilon(y,\tau)$ from \eqref{eq:ffeps}.
\end{itemize}

\noindent 
{\bf Step 2. The upper bound of $u^\varepsilon - u$ when $t\geq \varepsilon$.} We aim to show that 
\begin{equation*}
    u^\varepsilon(x, t)-u(x, t) \leq C\varepsilon
\end{equation*}
for some constant $C=C\left(n, \partial \Omega, H, \|Dg\|_{L^\infty(\mathbb{R}^n)}, \|Db\|_{L^\infty \left(\mathbb{R}^n \times \partial \Om \right)}\right)>0$. Fix $k\in \N$. By Lemma \ref{lem:ubarequ}, there exists $\tau_k\in[0, t)$ and $y_k \in \R^n$ such that
            \begin{equation}\label{eqn:umbarstarapp}
            \overline{m}^\ast \left(\tau_k, t, y_k, x \right) + f\left(y_k, \tau_k\right) \leq u(x, t) +\frac{1}{k}, \qquad \qquad   \left|x-y_k\right| \leq M_0 (t-\tau_k),
            \end{equation}
            where $M_0=M_0\left(n,\partial \Omega, H, \|Dg\|_{L^\infty(\mathbb{R}^n)}\right)>0$ is the constant from Lemma \ref{lem:velocity bound}.  Since $b(x, \cdot)$ is $\Z^n$-periodic for every $x$, we can choose $z_k \in \partial \Omega_\varepsilon$ with $\frac{z_k-x}{\ep}\in Y$ such that  
\begin{equation*}
	\ol{b}(y_k) = \min_{z \in \partial\Omega} b (y_k, z)= b\left(y_k, \frac{z_k}{\varepsilon}\right). 
\end{equation*}            
Note that 
	\begin{align}
	    & |z_k-x| \leq \ep\sqrt{n}, \label{eqn:zkxdifa}\\ 
		& |z_k-y_k|\leq |z_k-x| +|x-y_k| \leq \ep\sqrt{n} +M_0(t-\tau_k) \leq (M_0+\sqrt{n})\ep. \label{eqn:zkykdif}	
	\end{align}
            There are two cases $t-\tau_k <\varepsilon$ and $t-\tau_k \geq \varepsilon$.            
            \begin{enumerate}[label=$\mathrm{(\alph*)}$, leftmargin=0.7cm]
                \item[(i)] Suppose $t-\tau_k<\varepsilon$. By Lemma \ref{lem:tmbstarbd}, there exists $C=C\left(n, \partial \Omega, M_0, K_0\right)>0$ such that
                \begin{equation}\label{eqn:ttaukllbd}
                \overline{m}^\ast \left(\tau_k, t, y_k, x \right) \geq -\ep C. 
                \end{equation}
           Consider the constant path $\xi(s)\equiv z_k$ for $s \in [0, t]$, which is an admissible path for $u^\ep(z_k, t)$, and hence
            \begin{equation}\label{eqn:uepzkubd}
            u^\ep\left(z_k, t\right) 
            	\leq (t-\tau_k) L\left(\frac{z_k}{\ep}, 0\right) +f_\varepsilon (z_k, \tau_k) 
            	\leq C_1\ep  + f_\varepsilon(z_k, \tau_k) 
            \end{equation}
           where $C_1=C_1(H)>0$ is from \eqref{eq:H-lower-bd}. Since $t\geq \ep$ and $t-\tau_k <\ep$, $\tau_k>0$ and hence by the definition of $f, \fep$, and $z_k$, we have
           \begin{equation}\label{eqn:fepfdiff}
           \begin{aligned}
           \fep(z_k,\tau_k)-f(y_k, \tau_k) 
           		&= 
           	b\left(z_k,\frac{z_k}{\ep}\right)-\overline{b}(y_k) \\
           		&= 
           	b\left(z_k,\frac{z_k}{\ep}\right) - b\left(y_k,\frac{z_k}{\ep}\right)
 		        \leq \mathrm{Lip}(b)(M_0+\sqrt{n}) \ep,
           \end{aligned}       
           \end{equation}
           where the last inequality comes from \eqref{eqn:zkykdif}.
               Now combining \eqref{eqn:umbarstarapp}, \eqref{eqn:zkykdif}, \eqref{eqn:uepzkubd}, \eqref{eqn:fepfdiff} and Proposition \ref{prop:ueplip}, we obtain
      \begin{align*}
      	u^\varepsilon(x,t) - u(x,t) 
      	&\leq 
      	C_0|x-z_k| + u^\varepsilon(z_k,t) + \left(-\overline{m}^*\left(\tau_k t,y_k, x\right) - f(y_k,\tau_k) + \frac{1}{k}\right) \\
      	&\leq C_0\sqrt{n}\varepsilon + C_1\varepsilon + f_\varepsilon(z_k, t) - f(y_k, \tau_k) + C\varepsilon + \frac{1}{k} \leq C\varepsilon + \frac{1}{k}
      \end{align*}
            for some constant $C=C(n, \partial\Omega, H, \left\|Dg\right\|_{L^\infty\left(\mathbb{R}^n\right)}, \left\|Db\right\|_{L^\infty \left(\mathbb{R}^n\times \partial \Om\right)}])>0$. 

            \item [(ii)]Suppose $t-\tau_k\geq \varepsilon$.             
            From \eqref{eqn:umbarstarapp} and Proposition \ref{prop:mbarrate},
            \begin{equation}\label{eqn:tmstarappu}
            \begin{aligned}
            u(x, t) +\frac{1}{k} 
            		&\geq   \overline{m}^\ast \left(\tau_k, t, y_k, x \right) + f\left(y_k, \tau_k\right)\\
            		&\geq  \varepsilon \tm^\ast \left(\frac{\tau_k}{\varepsilon}, \frac{t}{\varepsilon}, \frac{y_k}{ \varepsilon}, \frac{x}{\varepsilon}\right)-C\ep + f\left(y_k, \tau_k\right)             
            \end{aligned}
            \end{equation}
            for some constant $C=C\left(n, \partial \Omega, M_0, K_0\right)>0$.
            There exists $\tilde{y}_k, \tilde{x} \in \partial\Omega$ with $\frac{\tilde{y_k}-y_k}{\varepsilon} \in Y$ and $\frac{\tilde{x} - x}{\varepsilon} \in Y$ such that 
            \begin{equation} \label{eqn:mstarmtild}
            \tm^\ast \left(\frac{\tau_k}{\varepsilon}, \frac{t}{\varepsilon}, \frac{y_k}{ \varepsilon}, \frac{x}{\varepsilon}\right) 
            = 
            \tm \left(\frac{\tau_k}{\varepsilon}, \frac{t}{\varepsilon}, \frac{\tilde{y}_k}{ \varepsilon}, \frac{\tilde{x}}{\varepsilon}\right)
            \end{equation}
            Note that since $t-\tau_k \geq \ep$, $|\tilde{x}-x| \leq \sqrt{n}\varepsilon$, $|\tilde{y}_k - y_k|\leq \sqrt{n}\varepsilon$, we have
            \begin{equation*}
            \left|\tilde{x}-\tilde{y}_k\right| \leq  |\tilde{x}-x|+|x-y_k|+|y_k-\tilde{y}_k|
            \leq  (2\sqrt{n}+M_0)(t-\tau_k).    
            \end{equation*}
            Therefore, by Proposition \ref{prop:sptm}, we have
            \begin{align}\label{eqn:txtykdif}
            	\tilde{m}\left(\frac{\tau_k}{\varepsilon}, \frac{t}{\varepsilon}, \frac{\tilde{y}_k}{\varepsilon}, \frac{\tilde{x}}{\varepsilon}\right) 
            	\geq 
            	\tilde{m}\left(\frac{\tau_k}{\varepsilon}, \frac{t}{\varepsilon}, \frac{z_k}{\varepsilon}, \frac{\tilde{x}}{\varepsilon}\right) - C
            \end{align}
            for some constant $C=C\left(n, \partial \Omega, M_0, K_0\right)>0$.
            By the definition of $\tm^\ast \left(\frac{\tau_k}{\varepsilon}, \frac{t}{\varepsilon}, \frac{z_k}{ \varepsilon}, \frac{x}{\varepsilon}\right)$, we have
            \begin{equation}\label{eqn:ytildx}
            \tm \left(\frac{\tau_k}{\varepsilon}, \frac{t}{\varepsilon}, \frac{z_k}{ \varepsilon}, \frac{\tilde{x}}{\varepsilon}\right) \geq \tm^\ast \left(\frac{\tau_k}{\varepsilon}, \frac{t}{\varepsilon}, \frac{z_k}{ \varepsilon}, \frac{x}{\varepsilon}\right)
            \end{equation}
            Hence, combining \eqref{eqn:tmstarappu}, \eqref{eqn:mstarmtild}, \eqref{eqn:txtykdif}, and \eqref{eqn:ytildx}, we obtain
    \begin{align}
    		u(x,t) + \frac{1}{k} 
    		     & \geq 
    		\varepsilon \widetilde{m}\left(\frac{\tau_k}{\varepsilon}, \frac{t}{\varepsilon}, \frac{\tilde{y}_k}{\varepsilon}, \frac{\tilde{x}}{\varepsilon}\right) - C\varepsilon + f(y_k, \tau_k)  \nonumber \\ 
    		&\geq 
    		\varepsilon \widetilde{m}\left(\frac{\tau_k}{\varepsilon}, \frac{t}{\varepsilon}, \frac{z_k}{\varepsilon}, \frac{\tilde{x}}{\varepsilon}\right) - \hat{C}\varepsilon + f(y_k, \tau_k)  \nonumber\\ 
    		&\geq \varepsilon \widetilde{m}^*\left(\frac{\tau_k}{\varepsilon}, \frac{t}{\varepsilon}, \frac{z_k}{\varepsilon}, \frac{x}{\varepsilon}\right) - \hat{C}\varepsilon + f(y_k, \tau_k)  \nonumber\\ 
    		&\geq \varepsilon \widetilde{m}\left(\frac{\tau_k}{\varepsilon}, \frac{t}{\varepsilon}, \frac{z_k}{\varepsilon}, \frac{x}{\varepsilon}\right) + f_\varepsilon(z_k,\tau_k) - \tilde{C}\varepsilon + f(y_k, \tau_k) -f_\varepsilon(z_k,\tau_k), \nonumber\\ 
    		&\geq u^\varepsilon(x,t)- \tilde{C}\varepsilon + f(y_k, \tau_k) -f_\varepsilon(z_k,\tau_k), 
    		\label{eqn:utmstarapp}
    \end{align}
    where the fourth inequality follows from Proposition \ref{prop:mstarm}. 
            We claim that
            \begin{equation}\label{eqn：fykfepzklbd}
            f\left(y_k, \tau_k\right)- \fep \left(z_k, \tau_k\right) \geq -C\ep
            \end{equation}
            for some constant $C=C\left(n,\|Dg\|_{L^\infty(\mathbb{R}^n)}, \|Db\|_{L^\infty \left(\mathbb{R}^n \times \partial \Om \right)} \right)>0$. Indeed, we have:
            \begin{itemize}
            \item If $\tau_k = 0$, then by \eqref{eqn:zkykdif}, 
            \begin{equation*}
            		f\left(y_k, \tau_k\right)- f_\ep\left(z_k, \tau_k\right)= g(y_k)-g(z_k) 
            			\geq - (M_0+\sqrt{n})\cdot \mathrm{Lip}(g)\cdot \varepsilon. 
            \end{equation*}
            \item If $\tau_k>0$, then by \eqref{eqn:zkykdif} and the definition of $z_k$,
			\begin{align*}
            		f\left(y_k, \tau_k\right)- f_\ep\left(z_k, \tau_k\right) 
            		&= \overline{b}(y_k)-b\left(z_k,\frac{z_k}{\ep}\right) \\
                &=b\left(y_k,\frac{z_k}{\ep}\right)-b\left(z_k,\frac{z_k}{\ep}\right)
            \geq - (M_0+\sqrt{n})\cdot \mathrm{Lip}(b)\cdot \varepsilon.
            \end{align*}
            \end{itemize}
            \noindent 
            Combining \eqref{eqn:utmstarapp} and \eqref{eqn：fykfepzklbd}, we have
            \begin{align*}
	            	u(x,t) + \frac{1}{k} \geq u^\varepsilon(x,t) - C\varepsilon
            \end{align*}
            for some constant $C=C(n,\partial \Omega, H, \|Dg\|_{L^\infty(\mathbb{R}^n)}, \left\|Db\right\|_{L^\infty \left(\mathbb{R}^n\times \partial \Om\right)})>0$. 
            \end{enumerate}
        In each case of $\mathrm{(i)}$ and $\mathrm{(ii)}$, we have
        \[
        u^\ep(x, t)-u(x, t) \leq C\ep+\frac{1}{k}
        \]
        for some constant $C=C\left(n, \partial \Omega, H, \|Dg\|_{L^\infty(\mathbb{R}^n)}, \|Db\|_{L^\infty \left(\mathbb{R}^n \times \partial \Om \right)}\right)>0$. Since $k\in \N$ is arbitrary, we have $u^\ep(x, t)-u(x, t) \leq C\ep$ and the proof is complete.
\end{proof}

\section{Discussions: Optimality, Diluted settings, and Domains with defects} \label{sec:discussion}
\subsection{Optimality of the rate of convergence}
To demonstrate the optimality of the convergence rate in Theorem \ref{thm:main1}, we present an example adapted from \cite[Proposition 4.3]{MitakeTranYu2019} with some modifications.

\begin{prop}
Let $n = 2$ and consider $\displaystyle H(y,p) = \frac{1}{2}|p|^2-V(y)$, where $V \in C(\mathbb{T}^2)$ satisfies 
\begin{equation*}
    0\leq V \leq 1, \qquad V = 1 \;\text{on}\; B_{1/8}(0), \qquad\;\mathrm{supp}(V)\subset B_{1/4}(0)    . 
\end{equation*}
Let $g \equiv 0$ and $b \equiv 1$. Set $x_0 = (\tfrac{1}{2}, \tfrac{1}{2})$, $ D = B_{1/4}(x_0) \subset [0,1]^2$ and define the domain (see Figure \ref{fig:example})
\begin{equation*}
    \Omega = \bigcup_{z \in \Z^2} \left(z + \big([0,1]^2 \setminus \overline{D}\big)\right).
\end{equation*}
For $\varepsilon > 0$, let $u^{\varepsilon}$ be the solution to \eqref{eqn:PDE_epsilon}, and let $u$ be the solution to \eqref{eqn:PDE_limit}. 
Then, for $\varepsilon \in (0,1)$,
\begin{equation}
u^{\varepsilon}(0,1) - u(0,1) \geq \frac{\varepsilon}{8}.
\end{equation}
\end{prop}

\begin{proof}
In this example, we have $u \equiv 0$. Thus, it suffices to prove that
\[
u^{\varepsilon}(0,1) \geq \frac{\varepsilon}{8}.
\]
From the optimal control formula for $u^\varepsilon(0,1)$ in \eqref{eqn:ocfue} with $(x,t) = (0,1)$, consider any admissible path $\gamma \in \mathcal{C}\!\left(0,\tfrac{1}{\varepsilon}; \overline{\Omega}\right)$
such that $(\gamma(\tau),\tau) \in  (\partial\Omega \times (0,1))\cup (\overline{\Omega}\times \{0\})$, so that 
\begin{align*}
    u^\varepsilon(0,1) = \varepsilon \int_{\frac{\tau}{\varepsilon}}^{\frac{1}{\varepsilon}} \left(\frac{1}{2}|\dot{\gamma}(s)|^2+V(\gamma(s))\right)\, ds +f_\varepsilon \left(\varepsilon \gamma\left(\frac{\tau}{\varepsilon}\right), \tau\right). 
\end{align*}

\noindent
{\bf Case 1.} If $\tau \in (0,1)$ then $\gamma(\frac{\tau}{\varepsilon}) \in \partial \Omega$. 
Then $f_\varepsilon \left(\varepsilon \gamma\left(\frac{\tau}{\varepsilon}\right), \tau\right)=b\left(\varepsilon \gamma\left(\frac{\tau}{\varepsilon}\right), \gamma\left(\frac{\tau}{\varepsilon}\right)\right)=1$ and

\[
\varepsilon \int_{\frac{\tau}{\varepsilon}}^{\frac{1}{\varepsilon}} 
    \left(\frac{1}{2}|\dot{\gamma}(s)|^2+V(\gamma(s))\right)\, ds +f_\varepsilon \left(\varepsilon \gamma\left(\frac{\tau}{\varepsilon}\right), \tau\right) \geq 1 \geq \frac{\varepsilon}{8}.
\]

\noindent
{\bf Case 2.} If $\tau = 0$, there are two possibilities: 
\begin{enumerate}[label=$\mathrm{(\alph*)}$, leftmargin=0.7cm]
    \item If $\gamma \left(\left[0,\frac{1}{\varepsilon}\right]\right)\subset\overline{B}_{1/8}$, then
    \begin{equation*}
        \varepsilon \int_{0}^{\frac{1}{\varepsilon}} 
    \left( \frac{1}{2}|\dot{\gamma}(s)|^2+V(\gamma(s))\right)\, ds + f_\varepsilon \left(\varepsilon \gamma\left(0\right), 0\right) 
        \geq 
    \varepsilon \int_{0}^{\frac{1}{\varepsilon}}  V(\gamma(s))\, ds = 1\geq \frac{\varepsilon}{8}.
    \end{equation*}
    
    \item If $\gamma([0, \frac{1}{\varepsilon}])\cap \T^2\backslash \overline{B}_{1/8} \neq \emptyset$. 
    We may assume that there exists $s_0\in (0, 1)$ be such that $\gamma\left(\frac{s_0}{\varepsilon}\right)=\frac{1}{8}$ and that $\gamma\left(\left(\frac{s_0}{\varepsilon},\frac{1}{\varepsilon}\right]\right)\subset \overline{B}_{1/8}$ without loss of generality. Then,
        \begin{align*}
            & \varepsilon\int_{0}^{\frac{1}{\varepsilon}}\left(\frac{1}{2}|\dot{\gamma}(s)|^2 +V(\gamma(s))\right)\, ds +f_\varepsilon \left(\varepsilon \gamma\left(0\right), 0\right) 
                \geq
                    \varepsilon\left(\frac{1}{2}\int_{\frac{s_0}{\varepsilon}}^{\frac{1}{\varepsilon}}|\dot{\gamma}(s)|^2 \, ds
                +
                    \int_{\frac{s_0}{\varepsilon}}^{\frac{1}{\varepsilon}}V(\gamma(s)) \, ds\right)\\
            & \qquad \qquad 
                \geq\varepsilon\left(\frac{\varepsilon}{2\left(1-s_0\right)}\left|\int_{\frac{s_0}{\varepsilon}}^{\frac{1}{\varepsilon}}\dot{\gamma}(s)\,ds\right|^2 + \frac{1-s_0}{\varepsilon}\right) =\varepsilon\left(\frac{\varepsilon}{128\left(1-s_0\right)}+\frac{1-s_0}{\varepsilon}\right)\geq\frac{\sqrt{2}}{8}\varepsilon.
\end{align*}
\end{enumerate}
Therefore, $\displaystyle u^{\varepsilon}(0,1) \geq \frac{\varepsilon}{8}$.
\end{proof}

\begin{figure}[htbp]
    \centering

\begin{minipage}[c]{0.3\textwidth}
\centering
\begin{tikzpicture}[scale=4]

    \begin{scope}
    \clip (-0.5,-0.5) rectangle (0.5,0.5);

    \fill[gray!40, even odd rule] 
    (-0.5,-0.5) rectangle (0.5,0.5) 
    (-0.5,0.5) circle (0.25)         
    (0.5,0.5) circle (0.25)          
    (0.5,-0.5) circle (0.25)         
    (-0.5,-0.5) circle (0.25);       

    \fill[white] (-1/2,1/2) circle (1/4);
    \draw[thick]  (-1/2,1/2) circle (1/4) node[above right=1pt] {};
    
    \fill[white] (1/2, 1/2) circle (1/4);
    \draw[thick]  (1/2, 1/2) circle (1/4) node[above right=1pt] {};
    
    \fill[white] (1/2,-1/2) circle (1/4);
    \draw[thick]  (1/2,-1/2) circle (1/4) node[above right=1pt] {};
    
    \fill[white] (-1/2,-1/2) circle (1/4);
    \draw[thick]  (-1/2,-1/2) circle (1/4) node[above right=1pt] {};

    \fill[gray!70] (0,0) circle (1/8);

    \draw[blue!70!black, thick, ->, samples=200, smooth, domain=0:0.9, variable=\t]
        plot ({9/10*0.12*cos(200*pi*\t)*exp(1.5*\t)*\t}, {0.12*sin(15*pi*\t)*exp(1.5*\t)});

    \draw[->] (-1.1,0) -- (1.1,0) node[right] {$x$};
    \draw[->] (0,-1.1) -- (0,1.1) node[above] {$y$};

    \draw[dashed, thick, gray!70] ( 0.5,-1) -- ( 0.5,1) node[above right] {};
    \draw[dashed, thick, gray!70] (-0.5,-1) -- (-0.5,1) node[above left] {};
    \draw[dashed, thick, gray!70] (-1, 0.5) -- (1, 0.5) node[above right] {};
    \draw[dashed, thick, gray!70] (-1,-0.5) -- (1,-0.5) node[above left] {};

    \draw[dashed, thick, gray!70]  (1/8,-1) --  (1/8,1) node[above right] {}; 
    \draw[dashed, thick, gray!70] (-1/8,-1) -- (-1/8,1) node[above right] {}; 
    \draw[dashed, thick, gray!70]  (-1,1/8) --  (1,1/8) node[above right] {}; 
    \draw[dashed, thick, gray!70] (-1,-1/8) -- (1,-1/8) node[above right] {}; 

    \draw[red!70]  (-1/2,-1/2) --  (-1/2,1/2) node[above right] {}; 
    \draw[red!70]  ( 1/2,-1/2) --  ( 1/2,1/2) node[above right] {}; 
    \draw[red!70]  (-1/2, 1/2) --  ( 1/2, 1/2) node[above right] {}; 
    \draw[red!70]  (-1/2,-1/2) --  ( 1/2,-1/2) node[above right] {}; 

    \end{scope} 
    
    \fill (0,0) circle (0.008) node[below left] {$0$};

    \fill (-1/2,-1/2) circle (0.008) node[below left] {$-\frac{1}{2}$};
    \fill ( 1/2,-1/2) circle (0.008) node[below right] {$\frac{1}{2}$};

    \fill (-1/4,-1/2) circle (0.008) node[below left] {$-\frac{1}{4}$};
    \fill ( 1/4,-1/2) circle (0.008) node[below right] {$\frac{1}{4}$};

    \fill (-1/8,-1/2) circle (0.008) node[below] {$-\frac{1}{8}$};
    \fill ( 1/8,-1/2) circle (0.008) node[below] {$\frac{1}{8}$};

    \fill (-1/2, -1/8) circle (0.008) node[left] {$-\frac{1}{8}$};
    \fill (-1/2, 1/8) circle (0.008) node[left] {$\frac{1}{8}$};
    \fill (-1/2, -1/4) circle (0.008) node[left] {$-\frac{1}{4}$};
    \fill (-1/2, 1/4) circle (0.008) node[left] {$\frac{1}{4}$};
    \fill (-1/2, 1/2) circle (0.008) node[left] {$\frac{1}{2}$};
    
\end{tikzpicture}
\caption*{Case 1}
\end{minipage}
\hfill 
\begin{minipage}[c]{0.3\textwidth}
\centering
     \begin{tikzpicture}[scale=4]

    \begin{scope}
    \clip (-0.5,-0.5) rectangle (0.5,0.5);

    \fill[gray!40, even odd rule] 
    (-0.5,-0.5) rectangle (0.5,0.5) 
    (-0.5,0.5) circle (0.25)         
    (0.5,0.5) circle (0.25)          
    (0.5,-0.5) circle (0.25)         
    (-0.5,-0.5) circle (0.25);       

    \fill[white] (-1/2,1/2) circle (1/4);
    \draw[thick]  (-1/2,1/2) circle (1/4) node[above right=1pt] {};
    
    \fill[white] (1/2, 1/2) circle (1/4);
    \draw[thick]  (1/2, 1/2) circle (1/4) node[above right=1pt] {};
    
    \fill[white] (1/2,-1/2) circle (1/4);
    \draw[thick]  (1/2,-1/2) circle (1/4) node[above right=1pt] {};
    
    \fill[white] (-1/2,-1/2) circle (1/4);
    \draw[thick]  (-1/2,-1/2) circle (1/4) node[above right=1pt] {};

    \fill[gray!70] (0,0) circle (1/8);

    \draw[blue!70!black, thick, ->, domain=0:1.5, samples=400, smooth, variable=\t]
    plot ({5/8*0.12*0.9*cos(100*pi*\t)*\t}, {5/8*0.12*0.9*sin(100*pi*\t)*\t});

    \draw[->] (-1.1,0) -- (1.1,0) node[right] {$x$};
    \draw[->] (0,-1.1) -- (0,1.1) node[above] {$y$};

    \draw[dashed, thick, gray!70] ( 0.5,-1) -- ( 0.5,1) node[above right] {};
    \draw[dashed, thick, gray!70] (-0.5,-1) -- (-0.5,1) node[above left] {};
    \draw[dashed, thick, gray!70] (-1, 0.5) -- (1, 0.5) node[above right] {};
    \draw[dashed, thick, gray!70] (-1,-0.5) -- (1,-0.5) node[above left] {};

    \draw[dashed, thick, gray!70]  (1/8,-1) --  (1/8,1) node[above right] {}; 
    \draw[dashed, thick, gray!70] (-1/8,-1) -- (-1/8,1) node[above right] {}; 
    \draw[dashed, thick, gray!70]  (-1,1/8) --  (1,1/8) node[above right] {}; 
    \draw[dashed, thick, gray!70] (-1,-1/8) -- (1,-1/8) node[above right] {}; 

    \draw[red!70]  (-1/2,-1/2) --  (-1/2,1/2) node[above right] {}; 
    \draw[red!70]  ( 1/2,-1/2) --  ( 1/2,1/2) node[above right] {}; 
    \draw[red!70]  (-1/2, 1/2) --  ( 1/2, 1/2) node[above right] {}; 
    \draw[red!70]  (-1/2,-1/2) --  ( 1/2,-1/2) node[above right] {}; 

    \end{scope} 
    
    \fill (0,0) circle (0.008) node[below left] {$0$};

    \fill (-1/2,-1/2) circle (0.008) node[below left] {$-\frac{1}{2}$};

    \fill (-1/4,-1/2) circle (0.008) node[below left] {$-\frac{1}{4}$};
    \fill ( 1/4,-1/2) circle (0.008) node[below right] {$\frac{1}{4}$};

    \fill (-1/8,-1/2) circle (0.008) node[below] {$-\frac{1}{8}$};
    \fill ( 1/8,-1/2) circle (0.008) node[below] {$\frac{1}{8}$};

    \fill (-1/2, -1/8) circle (0.008) node[left] {$-\frac{1}{8}$};
    \fill (-1/2, 1/8) circle (0.008) node[left] {$\frac{1}{8}$};
    \fill (-1/2, -1/4) circle (0.008) node[left] {$-\frac{1}{4}$};
    \fill (-1/2, 1/4) circle (0.008) node[left] {$\frac{1}{4}$};
    \fill (-1/2, 1/2) circle (0.008) node[left] {$\frac{1}{2}$};
    
\end{tikzpicture}
\caption*{Case 2 (a)}
\end{minipage}
\begin{minipage}[c]{0.3\textwidth}
\centering
\begin{tikzpicture}[scale=4]

    \begin{scope}
    \clip (-0.5,-0.5) rectangle (0.5,0.5);

    \fill[gray!40, even odd rule] 
    (-0.5,-0.5) rectangle (0.5,0.5) 
    (-0.5,0.5) circle (0.25)         
    (0.5,0.5) circle (0.25)          
    (0.5,-0.5) circle (0.25)         
    (-0.5,-0.5) circle (0.25);       

    \fill[white] (-1/2,1/2) circle (1/4);
    \draw[thick]  (-1/2,1/2) circle (1/4) node[above right=1pt] {};
    
    \fill[white] (1/2, 1/2) circle (1/4);
    \draw[thick]  (1/2, 1/2) circle (1/4) node[above right=1pt] {};
    
    \fill[white] (1/2,-1/2) circle (1/4);
    \draw[thick]  (1/2,-1/2) circle (1/4) node[above right=1pt] {};
    
    \fill[white] (-1/2,-1/2) circle (1/4);
    \draw[thick]  (-1/2,-1/2) circle (1/4) node[above right=1pt] {};

    \fill[gray!70] (0,0) circle (1/8);

    \draw[blue!70!black, thick, ->, samples=200, smooth, domain=0:0.7, variable=\t]
        plot ({-9/10*0.12*cos(200*pi*\t)*exp(1.5*\t)*\t}, {0.12*sin(15*pi*\t)*exp(1.5*\t)});
    \draw[blue!70!black, dashed, ->, samples=200, smooth, domain=0.7:0.9, variable=\t]
        plot ({-9/10*0.12*cos(200*pi*\t)*exp(1.5*\t)*\t}, {0.12*sin(15*pi*\t)*exp(1.5*\t)});

    \draw[->] (-1.1,0) -- (1.1,0) node[right] {$x$};
    \draw[->] (0,-1.1) -- (0,1.1) node[above] {$y$};

    \draw[dashed, thick, gray!70] ( 0.5,-1) -- ( 0.5,1) node[above right] {};
    \draw[dashed, thick, gray!70] (-0.5,-1) -- (-0.5,1) node[above left] {};
    \draw[dashed, thick, gray!70] (-1, 0.5) -- (1, 0.5) node[above right] {};
    \draw[dashed, thick, gray!70] (-1,-0.5) -- (1,-0.5) node[above left] {};

    \draw[dashed, thick, gray!70]  (1/8,-1) --  (1/8,1) node[above right] {}; 
    \draw[dashed, thick, gray!70] (-1/8,-1) -- (-1/8,1) node[above right] {}; 
    \draw[dashed, thick, gray!70]  (-1,1/8) --  (1,1/8) node[above right] {}; 
    \draw[dashed, thick, gray!70] (-1,-1/8) -- (1,-1/8) node[above right] {}; 

    \draw[red!70]  (-1/2,-1/2) --  (-1/2,1/2) node[above right] {}; 
    \draw[red!70]  ( 1/2,-1/2) --  ( 1/2,1/2) node[above right] {}; 
    \draw[red!70]  (-1/2, 1/2) --  ( 1/2, 1/2) node[above right] {}; 
    \draw[red!70]  (-1/2,-1/2) --  ( 1/2,-1/2) node[above right] {}; 

    \end{scope} 
    
    \fill (0,0) circle (0.008) node[below left] {$0$};

    \fill ( 1/2,-1/2) circle (0.008) node[below right] {$\frac{1}{2}$};

    \fill (-1/4,-1/2) circle (0.008) node[below left] {$-\frac{1}{4}$};
    \fill ( 1/4,-1/2) circle (0.008) node[below right] {$\frac{1}{4}$};

    \fill (-1/8,-1/2) circle (0.008) node[below] {$-\frac{1}{8}$};
    \fill ( 1/8,-1/2) circle (0.008) node[below] {$\frac{1}{8}$};

    \fill (-1/2, -1/8) circle (0.008) node[left] {$-\frac{1}{8}$};
    \fill (-1/2, 1/8) circle (0.008) node[left] {$\frac{1}{8}$};
    \fill (-1/2, -1/4) circle (0.008) node[left] {$-\frac{1}{4}$};
    \fill (-1/2, 1/4) circle (0.008) node[left] {$\frac{1}{4}$};
    \fill (-1/2, 1/2) circle (0.008) node[left] {$\frac{1}{2}$};
    
\end{tikzpicture}
\caption*{Case 2 (b)}
\end{minipage}
\caption{Different cases of $\gamma$ within a single cell $[-\frac{1}{2}, \frac{1}{2}]^2$ of the domain.}
\label{fig:example}
\end{figure}

\subsection{Homogenization in a dilute setting and domains with defects}

In this section, following the framework in \cite{han_quantitative_2024}, we study the dilute case and the defective domain case under \ref{itm:A7}. 
As stated in \cite{han_quantitative_2024}, if $H $ satisfies \ref{itm:A7}, then $L$ also satisfies a similar property \cite[Lemma 2.5]{han_quantitative_2024}.
\begin{lem}[Restatement of {\cite[Lemma 2.5]{han_quantitative_2024}}]\label{prop:Lmin0}
    Assume \ref{itm:A3} and \ref{itm:A7}.
    Then, for $y\in \R^n$,
    \[
    \min_{v\in \R^n} L(y,v)=L(y,0)=0.
    \]
\end{lem}

We note that under assumption \ref{itm:A7}, the solution $u^\varepsilon$ to the Cauchy problem with the Dirichlet boundary condition \eqref{eqn:PDE_epsilon} coincides with that of the state-constraint problem \eqref{eqn:PDE_sc}.

\begin{lem}\label{lem:dirichletsc}
Assume \ref{itm:A1}--\ref{itm:A5} and \ref{itm:A7}. For $\varepsilon>0$, let $u^\varepsilon$ be the viscosity solution to \eqref{eqn:PDE_epsilon} and let $u$ be the viscosity solution to \eqref{eqn:PDE_limit}, respectively. Then the following hold:
\begin{enumerate}
    \item[$\mathrm{(a)}$]$u^\varepsilon=\hat{u}^\varepsilon$ where $\hat{u}^\varepsilon$ solves
    \begin{equation}\label{eqn:PDE_sc}
    \left\{\begin{aligned}
    \hat{u}_t^\varepsilon+H\left(\frac{x}{\varepsilon}, D \hat{u}^\varepsilon\right) & \leq 0 \qquad \quad \text{in } \Omega_\varepsilon \times (0, \infty),\\
    \hat{u}_t^\varepsilon+H\left(\frac{x}{\varepsilon}, D\hat{u}^\varepsilon\right) & \geq 0 \qquad \quad \text{on } \overline{\Omega}_\varepsilon \times (0, \infty), \\
    \hat{u}^\varepsilon(x,0) & =g(x) \,\,\,\, \quad \text{on } \overline{\Omega}_\varepsilon \times \{t=0\}.
    \end{aligned}
    \right.
\end{equation}
    \item[$\mathrm{(b)}$]$u = \hat{u}$ where $\hat{u}$ solves
\begin{equation}\label{eqn:PDE_Rn}
    \left\{\begin{aligned}
    \hat{u}_t+\overline{H} \left(D\hat{u}\right) & = 0 \qquad \quad \text{in } \mathbb{R}^n \times (0, \infty), \\
    \hat{u}(x, 0)&=g(x) \, \, \,\quad \text{on } \mathbb{R}^n \times \left\{t=0\right\}.
    \end{aligned}
    \right.
\end{equation}
\end{enumerate}
\end{lem}

\begin{proof} We only provide the proof of $\mathrm{(a)}$, as the proof of $\mathrm{(b)}$ is analogous.
    Let $x \in \overline{\Omega}_\varepsilon$ and $t>0$. The optimal control formulas for $u^\varepsilon$ and $\hat{u}^\varepsilon$ are given by
    \begin{align*}
    u^\varepsilon(x,t) = 
    \left\lbrace 
	    \int_\tau^t L\left( \frac{\xi(s)}{\varepsilon} ,\dot{\xi}(s)\right) \,ds
	    	+
	    	f_\varepsilon \left(\xi(\tau),\tau \right)
	    	:  
	    \tau \in [0, t), \xi\in \mathcal{C}(x,t;\overline{\Omega}_\varepsilon), (\xi(\tau), \tau) \in \partial Q^\varepsilon_t
    \right\rbrace, 
    \end{align*}
and
    \begin{align*}
        \hat{u}^\varepsilon (x, t) = \inf \left\lbrace \int_0^t L\left( \frac{\xi(s)}{\varepsilon} ,\dot{\xi}(s)\right) \,ds+g\left(\xi(0)\right): \xi\in \mathcal{C}(x,t;\overline{\Omega}_\varepsilon)\right\rbrace, 
    \end{align*}
    respectively. 
    By definition, $u^\varepsilon(x,t) \leq \hat{u}^\varepsilon(x,t)$. Conversely, let $\xi\in \mathcal{C}(x,t;\overline{\Omega}_\varepsilon)$ with $\left(\xi (\tau), \tau\right) \in \left(\partial \Omega_\varepsilon \times (0, t)\right) \cup \left(\overline{\Omega}_\varepsilon \times \{0\} \right)$ for some $\tau \in [0, t)$ be an optimal path for $u^\varepsilon$, that is,
    \[
    u^\varepsilon (x, t) = \int_\tau^t L\left( \frac{\xi(s)}{\varepsilon} ,\dot{\xi}(s)\right) \,ds+f_\varepsilon \left(\xi(\tau),\tau \right).
    \]
    If $\tau=0$, then $\xi$ is also admissible for $\hat{u}^\varepsilon$, and $u^\varepsilon(x,t)\geq \hat{u}^\varepsilon(x,t)$.
   If $\tau >0$, define 
\[
\tilde{\xi}(s) := 
\begin{cases} 
\xi(\tau), & 0 \leq s \leq \tau,\\
\xi(s), & \tau \leq s \leq t.
\end{cases}
\] 
Then, by Lemma \ref{prop:Lmin0},
    \[
    \begin{aligned}
   \hat{u}^\varepsilon(x, t) &\leq  \int_0^t L\left( \frac{ \tilde{\xi}(s)}{\varepsilon} ,\dot{\tilde{\xi}}(s)\right) \,ds+g\left(\tilde{\xi}(0)\right)\\
   &=\int_0^\tau L\left( \frac{\xi(\tau)}{\varepsilon} ,0\right) \,ds +\int_\tau^t L\left( \frac{\xi(s)}{\varepsilon} ,\dot{\xi}(s)\right) \,ds+g\left(\xi(\tau)\right) \\
    &\leq \int_\tau^t L\left( \frac{\xi(s)}{\varepsilon} ,\dot{\xi}(s)\right) \,ds+ f_\varepsilon \left(\xi(\tau),\tau \right) = u^\varepsilon(x, t).
    \end{aligned}
    \]
Hence, $\hat{u}^\varepsilon (x, t)\leq u^\varepsilon (x, t)$.
\end{proof}

In light of Lemma \ref{lem:dirichletsc}, the results in \cite{han_quantitative_2024} for both the dilute case and the domain-defect case also hold for our Dirichlet problem. For completeness, we summarize the relevant settings and results below.

\subsubsection{Diluted setting} 
As previously noted, \ref{itm:D2} corresponds to the case where $\mathbb{R}^n \setminus \overline{\Omega}_\varepsilon$ has no unbounded connected components and the obstacles shrink faster than $\varepsilon$ as $\varepsilon \to 0$.
Thus in the limit it is expected to coincide with the homogenized solution of the problem without obstacles. Let $\overline{H}_{\R^n}$ be the effective Hamiltonian corresponding to $H$ in the whole space, i.e., for each $p\in \R^n$, $\overline{H}_{\R^n}(p)$ is the unique constant such that the cell problem
\begin{equation*}
    H(y, p+Dv(y)) = \overline{H}_{\R^n}(p) \qquad\text{in}\;\R^n
\end{equation*}
has a viscosity solution that is $\Z^n$-periodic. Let $\tilde{u}^\varepsilon$ and $\tilde{u}$ be solution to
\begin{equation}\label{eq:PDEeps-Rn}
    \begin{cases}
        \begin{aligned}
            \tilde{u}_t^\varepsilon + H\left(\frac{x}{\varepsilon}, D\tilde{u}^\varepsilon\right) &= 0 && \text{in}\;\R^n\times (0,\infty), \\
            \tilde{u}^\varepsilon(x,0) &= g(x) &&\text{on}\; \R^n,
        \end{aligned}
    \end{cases}
\end{equation}
and 
\begin{equation}\label{eq:PDElimitRn}
    \begin{cases}
        \begin{aligned}
            \tilde{u}_t + \overline{H}_{\R^n}\left(D\tilde{u}\right) &= 0 && \text{in}\;\R^n\times (0,\infty), \\
            \tilde{u}(x,0) &= g(x) &&\text{on}\; \R^n,
        \end{aligned}
    \end{cases}
\end{equation}
respectively. Thanks to Lemma \ref{lem:dirichletsc}, we have the following consequence. 

\begin{cor}[Restatement of {\cite[Theorem 1.2]{han_quantitative_2024}}] \label{cor:diluted} Assume \ref{itm:A1}--\ref{itm:A7} and \ref{itm:D2}. Let $u^\varepsilon, \tilde{u}^\varepsilon$ and $\tilde{u}$ be solution to \eqref{eqn:PDE_epsilon}, \eqref{eq:PDEeps-Rn} and \eqref{eq:PDElimitRn}, respectively. Then for $(x,t) \in \overline{\Omega}_\varepsilon\times [0,\infty)$, 
\begin{align*}
    \tilde{u}^\varepsilon(x,t) \leq u^\varepsilon(x,t) \leq \tilde{u}^\varepsilon(x,t) + C(\varepsilon +\eta(\varepsilon)t)
\end{align*}
for $C = C(n,\partial D, H, \mathrm{Lip}(g))$. Consequently, we have 
\begin{align*}
    \Vert u^\varepsilon(\cdot, t) - \tilde{u}(\cdot, t)\Vert_{L^\infty(\overline{\Omega}_\varepsilon)} \leq C(\varepsilon + \eta(\varepsilon) t).
\end{align*}
\end{cor}

\subsubsection{Domain with defects} Recall that \ref{itm:D3} corresponds to the case where the obstacles form non-diluted holes, though some may be missing. Let $W$ be the set defined in \ref{itm:D3}, and let $w^\varepsilon$ denote the viscosity solution of
\begin{equation}\label{eq:PDEeps-defects}
    \begin{cases}
        \begin{aligned}
            w_t^\varepsilon + H\left(\frac{x}{\varepsilon}, Dw^\varepsilon\right) &= 0 && \text{in}\;W_\varepsilon\times (0,\infty), \\
            w^\varepsilon(x,t) &= b\left(x, \frac{x}{\varepsilon}\right) &&\text{on}\;\partial W_\varepsilon\times(0,\infty),\\
            w^\varepsilon(x,0) &= g(x) &&\text{on}\; \overline{W}_\varepsilon\times \{0\}. 
        \end{aligned}
    \end{cases}
\end{equation}
By \cite[Theorem 1.3]{han_quantitative_2024} and Lemma \ref{lem:dirichletsc}, we have the following consequence.

\begin{cor}[Restatement of {\cite[Theorem 1.3]{han_quantitative_2024}}] \label{cor:defect} Assume \ref{itm:A1}--\ref{itm:A7} and \ref{itm:D3}. Let $u^\varepsilon, \tilde{w}^\varepsilon$ and $u$ be solution to \eqref{eqn:PDE_epsilon}, \eqref{eq:PDEeps-defects} and \eqref{eqn:PDE_limit}, respectively. Then for $(x,t) \in \overline{W}_\varepsilon\times [0,\infty)$, 
\begin{align*}
    - C\left(\varepsilon + (M_0t + |x|+1)\omega\left(\frac{\varepsilon}{M_0t + |x|}\right) \right) \leq w^\varepsilon (x,t) - u^\varepsilon(x,t)\leq 0
\end{align*}
for $C = C(n,\partial D, H, \mathrm{Lip}(g))$. Consequently, for $(x,t) \in \overline{W}_\varepsilon\times [0,\infty)$, 
\begin{align*}
   |w^\varepsilon(x,t) - u(x,t)| \leq C(M_0t+|x|+1)\omega\left(\frac{\varepsilon}{M_0 t + |x|}\right) + C\varepsilon.
\end{align*}
\end{cor}

\begin{proof}[Proof of Proposition \ref{prop:DilutedDefect}] The proof follows from Corollaries \ref{cor:diluted} and \ref{cor:defect}. 
\end{proof}

\vspace{0.6cm}

\appendix

\section{Distance on $\partial\Omega$ and the metric in the state-constraint setting} \label{appendix:costm}

We recall the following lemma from \cite{han_quantitative_2024}, which we use to redefine paths to avoid \(\Omega^c\). 

\begin{lem}[Restatement of {\cite[Lemma A.1]{han_quantitative_2024}}]\label{lem:A1} Let $p,q\in \partial\Omega$ such that $[p,q]\cap \overline{\Omega} = \{p,q\}$. 
Then $p$ and $q$ belong to a connected component $M$  of $\partial\Omega$, and there exists a $C^1$ curve $\gamma:[0,1]\to M$ such that $\gamma(0) = p, \gamma(1) = p$, and $\int_0^1 |\dot{\gamma}(s)|\;ds \leq C_b|p-q|$
for some constant $C_b = C_b(n,\partial\Omega)$. 
\end{lem}

We recall the cost function \(m\) and its properties from \cite{han_quantitative_2024} to study the convergence of solutions of the state-constraint problem \eqref{eqn:PDE_sc} to the limit problem \eqref{eqn:PDE_Rn}.

\begin{defn}[{\cite[Definitions~1 and 2]{han_quantitative_2024}}]
Assume \ref{itm:A1}--\ref{itm:A3}. For $t>0$, the cost and extended cost are defined by
\begin{align*}
    m(t,y,x) &:= \inf 
    \left\lbrace 
        \int_0^t L(\gamma(s), \dot{\gamma}(s))\;ds: 
        \gamma\in \mathrm{AC}\left([0,t];\overline{\Omega}\right), 
        \gamma(0) = y, \gamma(t) = x
    \right\rbrace, && y,x \in \overline{\Omega}, \\
    m^*(t,y,x) &:= \inf 
    \left\lbrace 
        m(t,\tilde{x}, \tilde{y}): \tilde{x}, \tilde{y} \in \partial\Omega, 
        \tilde{x} - x\in Y, \tilde{y} - y\in Y
    \right\rbrace, && y,x\in \R^n.
\end{align*}
\end{defn}

Under \ref{itm:A1}--\ref{itm:A5}, two important properties of \(m^*\) include, as established in \cite[Proposition 2.3 and Lemmas 3.1 and 3.2]{han_quantitative_2024} and extending \cite[Lemma 4]{burago_1992}:
\begin{enumerate}
    \item if \(t, s > 0\) with \(t \geq 1\) or \(s \geq 1\) and $x,y,z\in \R^n$ with $|x-y|\leq M_0 t, |y-z|\leq M_0s$, then
    \begin{equation*}
        m^*(t+s, x,z) \leq m^*(t,x,y) + m^*(s, y, z) + C;
    \end{equation*}
    \item if $t\geq 1$ and $y\in \R^n$ such that $|y|\leq M_0t$ 
    \begin{align*}
        2m^*(t,0,y) - C\leq m^*(2t,0,2y) \leq 2m^*(t,0,y) + C.
    \end{align*}
\end{enumerate}
    Here all the constants depend only on $n, \partial\Omega, M_0$ and $K_0$ from \eqref{eqn:K_0H}. These properties ensure that the average metric \(\overline{m}^\ast\) is well-defined, and allow an error estimate as in Proposition \ref{prop:mbarrate}.

\section{Derivative along absolutely continuous curves}
Next, we present an auxiliary result extending \cite[Lemma 4.2]{mitake_asymptotic_2008}. We note that the condition (B) in \cite{mitake_asymptotic_2008} is weaker than assuming $\partial\Omega$ is of class $C^1$ (i.e., locally given by the graph of a $C^1$ function). In contrast, both in this paper and in \cite{han_quantitative_2024}, we restrict to $C^1$ domains. 

We first present a simple proof for $C^2$ domains, noting that the vector field is only $C^1$ rather than $C^\infty$ as in \cite[Lemma 4.2]{mitake_asymptotic_2008}. The argument uses the fact that, for a $C^2$ domain, the distance function is $C^2$ in a neighborhood of the boundary (true for bounded or periodic domains), which allows the construction of a $C^1$ diffeomorphism that shrinks the domain inward along the normal direction. This classical approach has been applied in other state-constraint problems \cite[Proposition 3.2]{dutta_rate_2025}, \cite[Theorem VII.1]{Capuzzo-Dolcetta1990HamiltonJacobiConstraints}, and \cite[Lemma 2.10]{tu_generalized_2024} (we note a typo in \cite[Lemma 2.10]{tu_generalized_2024}, where the assumption should be $C^2$ boundary rather than $C^1$).

\begin{lem}\label{lem:smoothC2push} Let $\Omega$ be open, $\Z^n$-periodic in $\R^n$ with $C^2$-boundary. Then there exists a $C^1$ function $\zeta:\overline{\Omega}\to \R^n$ such that $x+\delta \zeta(x) \in \Omega$ for all $x\in \overline{\Omega}$ and $\delta\in(0,1)$.
\end{lem}
\begin{proof}[Proof of Lemma \ref{lem:smoothC2push}]
    If $\partial\Omega$ is $C^2$, the distance function $d(x) = \mathrm{dist}(x,\partial\Omega) = \inf_{y\in \partial\Omega} |x-z|$ for $x\in \overline{\Omega}$ is locally $C^2$ in a neighborhood of each point in $\partial\Omega$. By periodicity $\Omega+\Z^n=\Omega$, this neighborhood can be chosen uniformly: there exists $\rho>0$ such that 
        \begin{equation}\label{eq:C2strip}
            d \in C^2(U_\rho) 
            \qquad\text{where}\qquad 
            U_\rho = \{x \in \R^n : d(x)<\rho\}.
        \end{equation}
    For $x\in U_\rho$, the projection $x_0=\mathrm{proj}_{\partial\Omega}(x)$ is unique, $x\mapsto \mathrm{proj](x)}$ is $C^1$, and
    \begin{equation}\label{eq:dist1}
        \nabla d(x) = \frac{x-x_0}{|x-x_0|} = \mathbf{n}_\Omega(x_0), 
        \qquad        
        d(x_0+s\,\mathbf{n}_\Omega(x_0))=s \;\; \text{for } s\in[0,\tfrac{3}{4}\rho].
    \end{equation}
    Since $\partial\Omega\in C^2$, the unit inward normal map $x\mapsto \mathbf{n}_\Omega(x)$ is $C^1$. 
    \medskip

    \noindent 
    Let $\phi: \R\to \R$ be a decreasing $C^2$-function such that $\eta(s) = 0$ for $s\geq \frac{1}{2}\rho$ and $\eta = 1$ for $s\in [0,\frac{1}{4}\rho]$. We then define
        \begin{align}\label{eq:mapTdelta}
            T_\delta(x) = x + \delta\cdot\frac{\rho}{8}\phi(d(x))\cdot\mathbf{n}_\Omega(\mathrm{proj_{\partial\Omega}(x)}), \qquad x\in \overline{\Omega}, 0\leq \delta \leq 1. 
        \end{align}
        We have $T_\delta(x) = x$ for $x\in \{x\in \Omega: d(x) \geq \frac{1}{2}\rho\}$, and for $0\leq d(x) < \frac{1}{2}\rho$ we have $\mathrm{proj}_{\partial\Omega}(x) = x_0$, hence
        \begin{equation*}
            T_\delta(x) = x_0 + (x-x_0) + \delta\cdot \frac{\rho}{8}\phi(d(x))\cdot \mathbf{n}_\Omega(x_0) = x_0 + \left(d(x) + \frac{\rho}{8}\phi(d(x))\right)\cdot \mathbf{n}_\Omega(x_0).
        \end{equation*}
        We have
        \begin{align*}
            d(x) + \delta \cdot \frac{\rho}{8}\phi(d(x)) < \frac{\rho}{2} + \frac{\rho}{8} = \frac{5\rho}{8} < \frac{3\rho}{4}\qquad\text{if}\; 0< d(x) < \frac{1}{2}\rho. 
        \end{align*}
        By \eqref{eq:dist1} we have $\mathrm{proj}_{\partial\Omega}(T_\delta(x)) = x_0$ if $0< d(x) < \frac{1}{2}\rho$, and 
        \begin{align*}
            d\left(T_\delta(x)\right) =  d(x) + \delta \cdot \frac{\rho}{8}\phi(d(x)) = \begin{cases}
                \begin{aligned}
                    &d(x) + \delta \cdot \tfrac{\rho}{8} \geq \tfrac{\rho}{8} &&\text{if}\; 0\leq d(x) \leq \tfrac{\rho}{4}, \\[3mm]
                    &d(x) + \delta \cdot \tfrac{\rho}{8}\cdot \phi(d(x)) \geq \tfrac{\rho}{4} &&\text{if}\; \tfrac{\rho}{4} < d(x) \leq \tfrac{\rho}{2}. 
                \end{aligned}
            \end{cases}
        \end{align*}
        This gives us a $C^1$ map $\zeta:\overline{\Omega}\to \R^n$ such that $x+\delta \zeta(x) \in \Omega$ for all $x\in \overline{\Omega}$ and $0 < \delta < 1$. 
\end{proof}

    


Next, we give the general proof for a $C^1$ periodic domain by working on $\mathbb{T}^n$ and transferring the construction back; the argument is similar to \cite[Lemma 4.2]{mitake_asymptotic_2008}, with some care needed in the quotient space.

\begin{lem}\label{lem:smoothC1pushGeneral} Let $\Omega$ be open, $\Z^n$-periodic in $\R^n$ with $C^1$-boundary. Then there exists a $C^\infty$ function $\zeta:\overline{\Omega}\to \R^n$ such that $x+t \zeta(x) \in \Omega$ for all $x\in \overline{\Omega}$ and $t\in(0,1]$. 
\end{lem}

\begin{proof}[Proof of Lemma \ref{lem:smoothC1pushGeneral}] Let $\pi:\R^n\to \T^n$ be the quotient map. It is clear that $\pi(-x) = -\pi(x)$ and $\pi(x+y) = \pi(x)+\pi(y)$. For $\tilde{x}\in \T^n$, we denote by
\begin{align*}
    B_{\T^n}(\tilde{x},r) = \{\tilde{y}\in \T^n: |\tilde{x}-\tilde{y}| < r\}, \quad\text{where}\quad |\tilde{x}| = \inf_{z\in \R^n} \{|x-z|\}.
\end{align*}
Equip $\mathbb{T}^n$ with the metric $\delta(\tilde{x},\tilde{y}) = |\tilde{x}-\tilde{y}|$ so that $\mathbb{T}^n$ is a complete metric space.
For $\xi\in \R^n$, the operation $\tilde{x} + \xi \in \T^n$ as $\tilde{x} + \xi = \pi(x+\xi)$ for any $x\in \R^n$ with $\pi(x)=\tilde{x}$ is well-defined. Implicitly here we can write (locally) $\xi = \pi(\xi)$, by abused of notation instead of $\pi(\xi) = \xi + \Z^n$. If $\Omega \subset\R^n$ is connected then it is path connected, hence $\widetilde{\Omega}=\pi(\Omega)\subset\T^n$ is a path-connected open subset of $\T^n$. Let $\Sigma = \pi(\partial\Omega) \subset \T^n$. Then by property of the quotient space, $\Sigma = \partial\widetilde{\Omega}$ and is a compact $C^1$ hypersurface in $\T^n$. \medskip

Let $\tilde{p}\in \Sigma$. We choose $p\in \partial\Omega$ such that $\pi(p) = \tilde{p}$. Since $\partial\Omega$ is $C^1$, we can find $\xi_p\in \R^n$, $r_p>0$, and open neighborhoods $W_p\subset\subset U_p$ such that
\begin{align}\label{eq:pusha}
    y + s\xi_p \in U_p\cap \Omega \qquad\text{for all}\; y\in W_p\cap \overline{\Omega}\;\text{and}\; s\in (0,r_p].
\end{align}
We can choose $U_p\subset B(p,\frac{1}{4})$ without loss of generality by shrinking the size of the neighborhood. This ensures that the inverse map $\pi^{-1}: \tilde{y}\in \pi(U_p)\to U_p$ is uniquely defined. Let $\widetilde{U}_p = \pi(U_p)$ and $\widetilde{W}_p = \pi(W_p)$. For each $\tilde{y}\in \widetilde{W}_p\cap \overline{\pi(\Omega)}$, there is a unique $y\in W_p\cap \overline{\Omega}$ such that $\pi(y) = \tilde{y}$. From \eqref{eq:pusha} we have $y+s\xi_p \in U_p\cap \Omega$ for all $s\in (0,r_p]$. Therefore
\begin{align}\label{eq:pushb}
    \tilde{y} + s\xi_p  = \pi(y+s\xi_p)\in \widetilde{U}_p\cap \widetilde{\Omega}\qquad \text{for all}\; \tilde{y}\in \widetilde{W}_p\cap \overline{\pi(\Omega)} \;\text{and}\; s\in (0,r_p].
\end{align}
Since $\Sigma$ is a compact subset of $\T^n$, there exist finitely many points $\{\tilde{p}_i : i=1,2,\ldots, N\} \subset \Sigma$, together with open neighborhoods $\widetilde{W}_i \subset\subset  \widetilde{U}_i \subset \mathbb{T}^n$, a vector $\xi_i \in \mathbb{R}^n$, and a radius $r_i > 0$ such that $\Sigma = \bigcup_{i=1}^N \widetilde{W}_i$ and for each $i=1,2,\ldots, N$,
\begin{align*}
    \tilde{y} + s \xi_i \in \widetilde{U}_i \cap \pi(\Omega) 
\quad \text{for all}\; \tilde{y} \in \widetilde{W}_i \cap \overline{\pi(\Omega)} 
\;\text{and}\; s \in (0,r_i].    
\end{align*}
Without loss of generality, we can assume $|\xi_i|\leq 1$ for all $i=1,2,\ldots, N$, and by shrinking the neighborhoods if necessary, we may assume that $\pi(\Omega)\setminus \bigcup_{i=1}^N \widetilde{U}_i \neq \emptyset$. 
We then define
\begin{align*}
    K_0 = \pi(\Omega)\setminus \bigcup_{i=1}^N \widetilde{W}_i \neq \emptyset. 
\end{align*}
As $K_0\cap \Sigma  =\emptyset$, we can pick an open set $W_0$ such that $K\subset W_0\subset \overline{W}_0 \subset \pi(\Omega)\backslash \Sigma$. 
The compact space $\overline{\pi(\Omega)}$ has an open cover 
\begin{equation*}
    \overline{\pi(\Omega)}\subset W_0\cup \bigcup_{i=1}^N \widetilde{W}_i, 
\end{equation*}
thus we can find a smooth partition of unity $\{\zeta_0,\zeta_1,\ldots, \zeta_N\}\subset C^\infty(\overline{\pi(\Omega)})$ such that $\mathrm{supp}(\zeta_i)\subset \widetilde{W}_i$ for $i=1,2,\ldots, N$, $\mathrm{supp}(\zeta_0)\subset W_0$ and $\zeta_0 \equiv 1$ on $K_0$ such that
\begin{align*}
    0\leq \zeta_i\leq 1 \qquad\text{and}\qquad \sum_{i=1}^N \zeta_i(\tilde{x}) = 1\qquad\text{for}\; i=1,2,\ldots, N \;\text{and}\;\tilde{x}\in \overline{\pi(\Omega)}. 
\end{align*}
Let $r_0=\min\{r_1,\ldots, r_N\}$. We then define
\begin{align*}
    \zeta(\tilde{x}) = \frac{r_0}{2}\sum_{i=1}^N \zeta_i(\tilde{x})\xi_i, \qquad \tilde{x}\in \overline{\pi(\Omega)}.
\end{align*}
We have $\zeta \in C^\infty(\overline{\pi(\Omega)})$ with $|\zeta(\tilde{x})| \leq \tfrac{1}{2} r_0$ for all $\tilde{x} \in \overline{\pi(\Omega)}$. For any $\tilde{x} \in \overline{\pi(\Omega)}$,  
if $\tilde{x} \in W_0$, then $\tilde{x} + t \zeta(\tilde{x}) = \tilde{x} \in \pi(\Omega)$.
Otherwise, assume
\begin{equation*}
    \tilde{x} \in \widetilde{W}_j \qquad \text{for}\; j \in \{i_1,\ldots,i_k\} \subset \{1,2,\ldots,N\}.
\end{equation*}
Since $\widetilde{W}_i$ are open for $i=1,2,\ldots, k$, we can find $t_0$ small such that 
\begin{align*}
    \tilde{x} \in \widetilde{W}_i \qquad\Longrightarrow\qquad \tilde{w}+t\xi \in \widetilde{W}_i \qquad\text{for all}\; i=1,2,\ldots, n, |\xi|\leq \frac{r_0}{2}, t\in [0,t_0]. 
\end{align*}
With $t\in (0,t_0]$ and by \eqref{eq:pushb} we have
\begin{align*}
    \tilde{x} \in \widetilde{W}_{i_1}\cap \overline{\pi(\Omega)}&\qquad\Longrightarrow\qquad\tilde{x} + t\left(\frac{r_0} {2}\zeta_{i_1}(\tilde{x})\xi_{i_1}\right) \in \Omega\cap \widetilde{W}_{i_2}\\
    \tilde{x} \in \widetilde{W}_{i_2}\cap \overline{\pi(\Omega)}&\qquad\Longrightarrow\qquad\tilde{x} + t\left(\frac{r_0} {2}\zeta_{i_1}(\tilde{x})\xi_{i_1} + \frac{r_0}{2}\zeta_{i_2}(\tilde{x})\xi_{i_2}\right) \in \Omega\cap \widetilde{W}_{i_3}.
\end{align*}
By an induction argument we obtain $\tilde{x} + t\zeta(\tilde{x})\in \pi(\Omega)$.
Alternatively, one may note that any convex combination of inward directions is still inward, since locally $\Sigma$ is $C^1$. Finally, we extend $\zeta$ to $\overline{\Omega} \to \mathbb{R}^n$ in the natural periodic way, 
and include a scaling factor to ensure that $x + \delta \, \zeta(x) \in \Omega$ for all $\delta\in (0,1]$.
\end{proof}



\begin{proof}[Proof of Lemma \ref{lem:acdiff}] The argument combines Lemma~\ref{lem:smoothC1pushGeneral} (or Lemma~\ref{lem:smoothC2push} for a simpler proof in the $C^2$ case) with the method of \cite[Proposition~4.1]{mitake_asymptotic_2008}, which in turn builds on \cite[Proposition~2.4]{ishii_asymptotic_2008}.
\end{proof}





\section*{Acknowledgments} The authors thank Hung Vinh Tran for suggesting the problem and Hiroyoshi Mitake for insightful discussions on his previous work on the Dirichlet problem.



\bibliography{refs.bib}
\bibliographystyle{acm}
\end{document}